\documentclass[a4paper,10pt,reqno]{amsart}

\usepackage{amssymb,amsmath,amsfonts,amsthm,mathtools}
\usepackage{latexsym,mathrsfs,pb-diagram}
\usepackage[pdftex]{graphicx}
\usepackage[all]{xy}
\usepackage[hmargin=2.5cm,vmargin=3.5cm]{geometry}
\usepackage[plainpages=false,colorlinks,pdfpagelabels]{hyperref}
\hypersetup{
  urlcolor=black,
  citecolor=green,
  linkcolor=blue
}

\numberwithin{equation}{section}

\theoremstyle{definition}
\newtheorem{Def}{Definition}[section]

\theoremstyle{remark}
\newtheorem{Exa}[Def]{Example}
\newtheorem{Rem}[Def]{Remark}

\theoremstyle{plain}
\newtheorem{Prop}[Def]{Proposition}
\newtheorem{Cor}[Def]{Corollary}
\newtheorem{Thm}[Def]{Theorem}
\newtheorem{Lem}[Def]{Lemma}

\newtheorem{MainThm}{Theorem}

\newcommand{\dfn}{\mathrel{\dot{=}}}

\newcommand{\st}{ \ ; \ }
\newcommand{\rarr}{\rightarrow}
\newcommand{\sset}{\subset}

\newcommand{\eset}{\emptyset}
\newcommand{\Z}{\mathbb{Z}}
\newcommand{\N}{\mathbb{N}}
\newcommand{\R}{\mathbb{R}}
\newcommand{\Q}{\mathbb{Q}}
\newcommand{\C}{\mathbb{C}}

\newcommand{\TR}[5]{\begin{array}{c c c c c}
    {#1} & : & {#3} & \longrightarrow & {#5}\\
    & & {#2} & \longmapsto & {#4}
  \end{array}
}

\DeclareMathOperator{\Span}{\mathrm{span}}
\DeclareMathOperator{\ran}{\mathrm{ran}}

\newcommand{\del}{\partial}
\newcommand{\dd}{\mathrm{d}}

\newcommand{\D}{\mathscr{D}}

\newcommand{\cinfty}{\mathscr{C}^\infty}

\newcommand{\MM}{\mathrm{M}}

\DeclareMathOperator{\lie}{\mathrm{Lie}}

\newcommand{\hsum}[1]{\underset{#1}{\widehat{\bigoplus}}}

\DeclareMathOperator{\GL}{\mathrm{GL}}

\newcommand{\TT}{\mathbb{T}}

\newcommand{\gr}[1]{\mathfrak{#1}}

\DeclareMathOperator{\ad}{\mathrm{ad}}
\newcommand{\vv}[1]{\mathrm{#1}}

\author{Gabriel Ara\'{u}jo}
\address{Universidade de S{\~a}o Paulo, ICMC-USP, S{\~a}o Carlos, SP, Brazil}
\email{\texttt{gccsa@icmc.usp.br}}

\author{Igor A.~Ferra}
\address{Universidade Federal do ABC, CMCC-UFABC, S{\~a}o Bernardo do Campo, SP, Brazil}
\email{\texttt{ferra.igor@ufabc.edu.br}}

\author{Luis F.~Ragognette}
\address{Universidade Federal de S{\~a}o Carlos, DM-UFSCar, S{\~a}o Carlos, SP, Brazil}
\email{\texttt{luisragognette@dm.ufscar.br}}

\thanks{This work was supported by the S{\~a}o Paulo Research Foundation (FAPESP, grants~2016/13620-5 and~2018/12273-5).}

\title[]{Global hypoellipticity of sums of squares on compact manifolds}

\begin{document}

\begin{abstract} In this work, we present necessary and sufficient conditions for an operator of the type sum of squares to be globally hypoelliptic on a product of compact Riemannian manifolds $T \times G$, where $G$ is also a Lie group.
These new conditions involve the global hypoellipticity of a system of vector fields and are weaker than  H\"ormander's condition, at the same time that they generalize the well known Diophantine conditions on the torus. We were also able to provide examples of operators satisfying these conditions in the general setting.
\end{abstract}

\maketitle


\section*{Introduction}

It is well known that H\"{o}rmander's bracket condition~\cite{hormander67} does not characterize global hypoellipticity for operators of the type sum of squares~\cite{hp00}. Several works investigate global hypoellipticity of such operators on the $N$-dimensional torus, $\mathbb{T}^N$, see for instance~\cite{him95, hp00, hps06} and the references therein, while results about Gevrey or real-analytic regularity can be found in~\cite{chim94, chim98, albanese11}. The latter question was also studied on more general compact manifolds~\cite{christ94, bccp}.

When dealing with an operator $P$ of tube type on a torus, that is an operator defined on a product $\mathbb T_t^n\times \mathbb T_x^m$  whose coefficients depend only on the $t$ variable, there are conditions (weaker than H\"{o}rmander's) involving Diophantine properties about the coefficients of $P$ that completely characterize its global hypoellipticity~\cite{bfp17}. Such number-theoretic conditions naturally arise when one approaches this problem using partial Fourier series on the $x$ variable.

In this work we study global hypoellipticity of operators defined on the product of a compact Riemannian manifold with a compact Lie group. More precisely, let $T$ be a compact, connected and orientable smooth manifold and $G$ be a compact and connected Lie group. Our main result concerns the global hypoellipticity of operators on $T\times G$ of the following kind:
\begin{align*}
  P &\dfn \Delta_T - \sum_{\ell=1}^N \left( \sum_{j=1}^m  a_{\ell j}(t) \vv{X}_j+ \vv{W}_\ell\right)^2
\end{align*}
where $\Delta_T$ is the Laplace-Beltrami operator on $T$ associated to a given Riemannian metric, $\vv{W}_1,\ldots \vv{W}_N$ are skew-symmetric, real, smooth vector fields on $T$, while $a_{\ell j}\in \cinfty(T; \R)$ for every $\ell \in \{1, \ldots, N\}$ and $j \in \{1, \ldots, m\}$, and $\vv{X}_1, \ldots, \vv{X}_m$ is a basis of real left-invariant vector fields on $G$. 

The novelty here is that, in this more general setting, we must replace the Diophantine condition that appeared naturally on tori by the global hypoellipticity of a certain system of vector fields on $G$. We proved that this condition is necessary for global hypoellipticity of $P$, see Theorem~\ref{thm:PGHnecessaell} and Proposition~\ref{Pro:66impliesLGH}; and, under an additional hypothesis, Theorem~\ref{thm:thm15} says that this condition is also sufficient. We stress that when $G=\mathbb T^m$ this additional hypothesis is always satisfied by our operator.

In particular we can state our main result on $T\times \mathbb{T}^m$ as follows:
\begin{MainThm}\label{Thm:Toruscase}
  Let $T$ be a compact manifold as above and consider the LPDO on $T\times \mathbb{T}^m$ defined by
  \begin{align*}
    P &\dfn \Delta_T - \sum_{\ell=1}^N \left( \sum_{j=1}^m  a_{\ell j}(t) \frac{\del}{\del x_j} + \vv{W}_\ell \right)^2.
  \end{align*}
  Then $P$ is globally hypoelliptic in $T\times \mathbb{T}^m$ if and only if the system of vector fields with constant coefficients
  \begin{align*}
    \mathcal{L} &\dfn \left \{ \vv{L} \in \mathfrak{g} \st \vv{L} = \sum_{j=1}^m  a_{\ell j}(t) \frac{\partial}{\partial x_j} \ \text{ for some $\ell \in \{1, \ldots, N\}$ and some $t\in T$} \right \}
  \end{align*}
  is globally hypoelliptic in $\mathbb{T}^m$.
\end{MainThm}
Throughout this work, $\mathfrak g$ denotes the Lie algebra of $G$. When $G = \mathbb{T}^m$, $\mathfrak{g}$ is the space of $\R$-linear combinations of $\frac{\partial}{\partial x_1}, \ldots, \frac{\partial}{\partial x_m}$. We recall that the system $\mathcal{L}$ is globally hypoelliptic in $\mathbb{T}^m$ if every distribution $u$ in $\mathbb{T}^m$ satisfying $\vv{L} u\in\cinfty(\mathbb{T}^m)$ for every $\vv{L}\in \mathcal{L}$ is already smooth.

It turns out that this condition about $\mathcal L$ is equivalent to the Diophantine condition presented in~\cite{bfp17} (see Section~\ref{sec:nsa_vectors} for more details), thus our result above generalizes~\cite[Theorem~1.5]{bfp17}. However, stated as such, our new condition is much easier to check than the number-theoretic one in many practical situations: an immediate application is a generalization of~\cite[Theorem~3]{albanese11} (Example~\ref{exa:lines}). Our techniques also allowed us to prove broader versions of~\cite[Theorem~1.9]{bfp17} (Theorem~\ref{thm:thm19}) and of~\cite[Theorem~1]{albanese11} (Theorem~\ref{thm:albanese}). 

Section~\ref{sec:partial_FPM} is devoted to develop the machinery -- a suitable substitute to partial Fourier series -- that was used throughout the other sections. Although most of the results here are rather expected, we decided to keep some of their proofs (or sketches) in the text as we did not find some of them in the literature.
 
H\"{o}rmander's condition will be explored in Section~\ref{sec:hormander}. On one hand, a finite type condition at a single point implies that the system $\mathcal{L}$ is globally hypoelliptic (Corollary~\ref{cor:hormander_point}). On the other hand, Example~\ref{exa:lines} with a convenient choice of coefficients yields an operator that is globally hypoelliptic while the finite type condition fails to be true everywhere.

We would like to point out that our hypotheses in Theorem~\ref{thm:thm15} were carefully chosen in order to allow us to provide examples where $G$ is not the $m$-dimensional torus as we show, in Section~\ref{sec:orbits}, that slightly stronger assumptions would force $G$ to be Abelian.

\section{Preliminaries} \label{sec:preliminaries}

Let $M$ be a compact, connected, smooth manifold, which for simplicity we further require to be orientable and in fact oriented. We endow it with a Riemannian metric, and we denote by $\dd V$ either its underlying volume form or the Radon measure induced by it on $M$. The $L^2$ norms below are always taken with respect to this measure, which we assume w.l.o.g.~to be normalized (i.e.~$M$ has total mass equal to $1$). For each $x \in M$ we denote by $\langle \cdot, \cdot \rangle_{T_x M}$ the inner product on $T_x M$ induced by our metric, which by means of the Riesz isomorphism
\begin{align*}
  \flat_x: v \in T_x M &\longmapsto \langle \cdot, v \rangle_{T_x M} \in T_x^* M 
\end{align*}
induces an inner product $\langle \cdot, \cdot \rangle_{T_x^* M}$ on $T_x^* M$, yielding a smooth metric on $T^* M$. This in turn produces an inner product on $\cinfty(M; T^* M)$: if $u, v$ are two smooth $1$-forms on $M$ we define
\begin{align}
  \langle u, v \rangle_{L^2(M)} &\dfn \int_M \langle u_x, v_x \rangle_{T_x^* M} \ \dd V(x). \label{eq:inner_prod_1forms}
\end{align}

Let $\dd:  \cinfty(M; \R) \rarr \cinfty(M; T^* M)$ be the exterior derivative and $\dd^*: \cinfty(M; T^* M) \rarr \cinfty(M; \R)$ its formal adjoint w.r.t.~\eqref{eq:inner_prod_1forms}. Both of them are first-order differential operators, and the Laplace-Beltrami operator is then defined as the second-order differential operator 
\begin{align*}
  \Delta \dfn \dd^* \dd : \cinfty(M; \R) \longrightarrow  \cinfty(M; \R).
\end{align*}
Everything above can be complexified by allowing all the objects involved to take values in $\C$.

Let us recall the main properties of $\Delta$ which will be of fundamental importance to us. It is an elliptic operator, and clearly positive semidefinite i.e.~$\langle \Delta f, f \rangle_{L^2(M)} \geq 0$ for all $f \in \cinfty(M)$. We denote by $\sigma(\Delta) \sset \R_+$ its spectrum i.e.~the set of all eigenvalues of $\Delta$: this set is countably infinite, and for each $\lambda \in \sigma(\Delta)$ we denote by
\begin{align*}
  E_\lambda &\dfn \ker (\Delta - \lambda I)
\end{align*}
the eigenspace associated with $\lambda$, which is a finite dimensional vector space containing smooth functions only. These eigenspaces are pairwise orthogonal in $L^2(M)$, and $E_0$ is precisely the space of constant functions since $M$ is connected. The Spectral Theorem tells us that if we endow each $E_\lambda$ with the $L^2$ inner product then
\begin{align}
  L^2(M) &\cong \hsum{\lambda \in \sigma(\Delta)} E_\lambda \label{eq:dec_delta}
\end{align}
as Hilbert spaces. Moreover, the following consequence of Weyl's asymptotic formula~\cite[p.~155]{chavel_eigenvalues} holds
\begin{align}
  \sum_{\lambda \in \sigma(\Delta) \setminus 0} (\dim E_\lambda) \lambda^{-2m} &< \infty \label{eq:weyl}
\end{align}
where $m \dfn \dim M$: indeed, writing the eigenvalues of $\Delta$ -- repeated according to their multiplicities -- as a non-decreasing sequence $\{ \lambda_\nu \}_{\nu \in \N}$ one has $(\dim E_{\lambda_\nu}) \lambda_\nu^{-2m} = \mathrm{O}(\nu^{-3})$.

Let us recall in detail the meaning of~\eqref{eq:dec_delta}. By introducing the space of sequences
\begin{align*}
  \Pi(\Delta) &\dfn \prod_{\lambda \in \sigma(\Delta)} E_\lambda
\end{align*}
we have, by definition,
\begin{align*}
  \hsum{\lambda \in \sigma(\Delta)} E_\lambda &\dfn \left\{ u \in \Pi(\Delta) \st \text{ $ \left( \| u(\lambda) \|_{L^2(M)}^2 \right)_{\lambda \in \sigma(\Delta)}$ is summable in $\R$} \right\}
\end{align*}
which becomes a Hilbert space when endowed with the inner product
\begin{align*}
  \langle u, v \rangle_{L^2(\Delta)} &\dfn \sum_{\lambda \in \sigma(\Delta)} \langle u(\lambda), v(\lambda) \rangle_{L^2(M)}.
\end{align*}
For simplicity we denote this Hilbert space by $L^2(\Delta)$. If for each $\lambda \in \sigma(\Delta)$ we denote by $\mathcal{F}_\lambda: L^2(M) \rarr E_\lambda$ the corresponding orthogonal projection then every $f \in L^2(M)$ can be written as
\begin{align*}
  f &= \sum_{\lambda \in \sigma(\Delta)} \mathcal{F}_\lambda(f) 
\end{align*}
where convergence takes place in $L^2(M)$. We then assemble the linear map
\begin{align}
  \TR{\mathcal{F}}{f}{L^2(M)}{ \left( \mathcal{F}_\lambda(f) \right)_{\lambda \in \sigma(\Delta)}}{\Pi(\Delta)} \label{eq:Fourier_proj_map}
\end{align}
so~\eqref{eq:dec_delta} means that $\mathcal{F}$ is an isometric isomorphism from $L^2(M)$ onto $L^2(\Delta)$. In the same spirit, we may use the projection map $\mathcal{F}$ to identify many spaces of (generalized) functions on $M$ by analyzing the growth of their corresponding sequences in $\Pi(\Delta)$, in a Paley-Wiener-like fashion.

The space $\cinfty(M)$ of all complex-valued smooth functions on $M$ is naturally endowed with a locally convex topology (uniform convergence of all derivatives on compact coordinate sets). As our volume form $\dd V$ allows us to identify the space of all smooth densities on $M$ with $\cinfty(M)$, by the same token we may identify the topological dual of the latter with $\D'(M)$, the space of Schwartz distributions on $M$. The measure $\dd V$ further allows us to embed all the classical spaces of functions in $\D'(M)$: we interpret each $f \in L^1(M)$ as a distribution on $M$ by letting it act on a test function $\phi \in \cinfty(M)$ as
\begin{align*}
  \langle f, \phi \rangle &\dfn \int_M f \phi \ \dd V. 
\end{align*}
In that sense, $\Delta$ acts on distributions (recall this is a real operator) as follows: if $f \in \D'(M)$ then $\langle \Delta f, \phi \rangle = \langle f, \Delta \phi \rangle$ for every test function $\phi \in \cinfty(M)$. Moreover, for each $\lambda \in \sigma(\Delta)$ we have $f|_{E_\lambda} \in E_\lambda^*$, and we denote by $\mathcal{F}_\lambda (f)$ the unique element in $E_\lambda$ that satisfies
\begin{align*}
  \langle \mathcal{F}_\lambda (f), \phi \rangle_{L^2(M)} &= \langle f, \overline{\phi} \rangle, \quad \forall \phi \in E_\lambda.
\end{align*}
Concretely, if $\{ \phi^\lambda_i \st 1 \leq i \leq \dim E_\lambda \}$ is an orthonormal basis for $E_\lambda$ then for $f \in \D'(M)$ we have
\begin{align*}
  \mathcal{F}_\lambda (f) = \sum_{i = 1}^{d_\lambda} \langle \mathcal{F}_\lambda (f), \phi^\lambda_i \rangle_{L^2(M)} \ \phi^\lambda_i = \sum_{i = 1}^{d_\lambda} \langle f, \overline{\phi^\lambda_i} \rangle \ \phi^\lambda_i,
\end{align*}
where $d_\lambda \dfn \dim E_\lambda$, which coincides with the original definition of $\mathcal{F}_\lambda(f)$ when $f \in \cinfty(M)$. We have thus defined a linear map
\begin{align*}
  \TR{\mathcal{F}}{f}{\D'(M)}{ \left( \mathcal{F}_\lambda(f) \right)_{\lambda \in \sigma(\Delta)}}{\Pi(\Delta)}
\end{align*}
that naturally extends~\eqref{eq:Fourier_proj_map}. One then easily proves that:
\begin{Prop} \label{prop:charac_smoothness_fourier_proj} For $a \in \Pi(\Delta)$ the following characterizations hold:
  \begin{enumerate}
  \item $a = \mathcal{F}(f)$ for some $f \in \cinfty(M)$ if and only if for every $s > 0$ there exists $C > 0$ such that
    \begin{align*}
      \| a(\lambda) \|_{L^2(M)} &\leq C (1 + \lambda)^{-s}, \quad \forall \lambda \in \sigma(\Delta).
    \end{align*}
  \item $a = \mathcal{F}(f)$ for some $f \in \D'(M)$ if and only if there exist $C, s > 0$ such that
    \begin{align*}
      \| a(\lambda) \|_{L^2(M)} &\leq C (1 + \lambda)^{s}, \quad \forall \lambda \in \sigma(\Delta).
    \end{align*}
  \end{enumerate}
\end{Prop}
See e.g.~\cite[Lemmas~4.1(1) and~4.3(1)]{araujo19}, and also~\cite[Section~5.1]{araujo19}.

\subsection{Orthogonal expansion of vector-valued distributions} \label{rem:ext_fourier_to_tensors}

Given a complex finite dimensional vector space $\mathscr{V}$, for each $\lambda \in \sigma(\Delta)$ we may extend the action of $\mathcal{F}_\lambda: \D'(M) \rarr E_\lambda$ to $\D'(M; \mathscr{V}) \cong \D'(M) \otimes \mathscr{V}$ in a natural way. Indeed, the map
\begin{align*}
  (f, v) \in \D'(M) \times \mathscr{V} &\longmapsto \mathcal{F}_\lambda(f) \otimes v \in E_\lambda \otimes \mathscr{V}
\end{align*}
-- concisely written $\mathcal{F}_\lambda \otimes \mathrm{id}_{\mathscr{V}}$ -- is surely bilinear, hence lifts to a linear map $\D'(M) \otimes \mathscr{V} \rarr E_\lambda \otimes \mathscr{V}$, which we again denote simply by $\mathcal{F}_\lambda$. Concretely, given a basis $\{ v_1, \ldots, v_d \}$ of $\mathscr{V}$ we can write $f \in \D'(M) \otimes \mathscr{V}$ in a unique fashion as
\begin{align*}
  f &= \sum_{i = 1}^d f_i \otimes v_i, \quad f_i \in \D'(M),
\end{align*}
so that $\mathcal{F}_\lambda(f) \in E_\lambda \otimes \mathscr{V}$ is just
\begin{align*}
  \mathcal{F}_\lambda(f) &= \sum_{i = 1}^d \mathcal{F}_\lambda (f_i) \otimes v_i.
\end{align*}
On the other hand, $\D'(M; \mathscr{V})$ is naturally identified with the topological dual of $\cinfty(M; \mathscr{V}^*)$: any $\phi \in \cinfty(M; \mathscr{V}^*) \cong \cinfty(M) \otimes \mathscr{V}^*$ can be uniquely written as
\begin{align*}
  \phi &= \sum_{i' = 1}^d \phi_{i'} \otimes v^*_{i'}, \quad \phi_{i'} \in \cinfty(M),
\end{align*}
where $\{ v^*_1, \ldots, v^*_d \}$ is the basis of $\mathscr{V}^*$ dual to $\{ v_1, \ldots, v_d \}$, so
\begin{align*}
  \langle f, \phi \rangle = \sum_{i, i' = 1}^d \langle f_i, \phi_{i'} \rangle \langle v^*_{i'}, v_i \rangle = \sum_{i = 1}^d \langle f_i, \phi_i \rangle.
\end{align*}
\subsection{Riemannian metrics on compact Lie groups} \label{sec:introLG}

Let $G$ be a compact and connected Lie group, whose dimension as a manifold we denote by $m$. For each $x \in G$ we denote by $L_x: G \rarr G$ the left translation by $x$, which is a diffeomorphism of $G$ onto itself. A vector field $\vv{X}$ on $G$ is said to be \emph{left-invariant} if $(L_x)_* \vv{X} = \vv{X}$ for every $x \in G$. One defines left-invariant differential forms, tensors, etc., analogously. We denote by $\gr{g}$ the Lie algebra of all \emph{real} vector fields on $G$ that are left-invariant: this is a finite dimensional vector space, canonically isomorphic to $T_e G$ -- where $e \in G$ stands for the identity element -- by means of the map
\begin{align}
  \vv{X} \in \gr{g} &\longmapsto \vv{X}|_e \in T_e G. \label{eq:iso_evale}
\end{align}

Any basis $\vv{X}_1, \ldots, \vv{X}_m \in \gr{g}$ forms a global frame for $TG$, and if $\chi_1, \ldots, \chi_m \in \gr{g}^*$ is the corresponding dual basis -- which we regard as left-invariant $1$-forms on $G$ -- they form a global frame for $T^*G$. In particular, $\chi \dfn \chi_1 \wedge \cdots \wedge \chi_m$ is a non-vanishing left-invariant top-degree form on $G$, and it is easy to check that any other such form must be a multiple of $\chi$: one often calls
\begin{align*}
  \dd V_G &\dfn \left( \int_G \chi \right)^{-1} \chi
\end{align*}
the \emph{Haar volume form} of $G$ associated with the orientation of $G$ given by the global frame $\vv{X}_1, \ldots, \vv{X}_m$. The Radon measure on $G$ induced by $\dd V_G$ is called the \emph{Haar measure} of $G$, and is the unique left-invariant regular Borel measure on $G$ with total mass equal to $1$.

Left-invariant Riemannian metrics on $G$ are in one-to-one correspondence with inner products on $\gr{g}$: any such inner product, which we regard as an inner product on $T_e G$ via~\eqref{eq:iso_evale}, can be pushed forward by $L_x$ to an inner product on $T_x G$ for every $x \in G$, thus producing the desired left-invariant Riemannian metric. Now if we fix an inner product $\langle \cdot, \cdot \rangle$ on $\gr{g}$ and, as above, select $\vv{X}_1, \ldots, \vv{X}_m$ an orthonormal basis for $\gr{g}$ then $\chi$ is precisely the Riemannian volume form w.r.t.~the left-invariant Riemannian metric $\langle \cdot, \cdot \rangle$ and compatible with the orientation of $G$ given by $\vv{X}_1, \ldots, \vv{X}_m$. In particular, the Riemannian volume form w.r.t.~a left-invariant Riemannian metric is always left-invariant, hence a constant multiple of the Haar volume form. As such, with respect to such a metric any left-invariant vector field $\vv{X} \in \gr{g}$ is \emph{(formally) skew-symmetric} i.e.
\begin{align*}
  \langle \vv{X} f, g \rangle_{L^2(G)} &= - \langle f, \vv{X} g \rangle_{L^2(G)}, \quad \forall f, g \in \cinfty(G).
\end{align*}

Particular relevant to what comes next are the so-called \emph{$\ad$-invariant metrics}: these are left-invariant Riemannian metrics $\langle \cdot, \cdot \rangle$ on $G$ with the additional property that
\begin{align}
  \langle [\vv{X}, \vv{Y}], \vv{Z} \rangle &= - \langle \vv{Y}, [\vv{X}, \vv{Z}] \rangle, \quad \forall \vv{X}, \vv{Y}, \vv{Z} \in \gr{g}. \label{eq:adinvariantmetric}
\end{align}
Such metrics always exist since we are assuming $G$ to be compact~\cite[Proposition~4.24]{knapp_lgbi}. The key point is that, in that case, if $\vv{X}_1, \ldots, \vv{X}_m \in \gr{g}$ is an orthonormal basis then the Laplace-Beltrami operator $\Delta_G$ associated to $\langle \cdot, \cdot \rangle$ can be written as
\begin{align}
  \Delta_G &= -\sum_{j = 1}^m \vv{X}_j^2 \label{eq:DeltaG_SS}
\end{align}
and moreover every left-invariant vector field on $G$ commutes with $\Delta_G$.
\section{Partial Fourier projection maps on product manifolds} \label{sec:partial_FPM}

Let $T, G$ be two compact, connected, smooth manifolds, orientable and oriented, and also carrying Riemannian metrics $\langle \cdot, \cdot \rangle^T, \langle \cdot, \cdot \rangle^G $, just like $M$ did in Section~\ref{sec:preliminaries}, and whose dimensions will be denoted by $n \dfn \dim T$ and $m \dfn \dim G$ respectively. Then their product enjoys the very same properties; for instance, the orientations of $T$ and $G$ induce canonically an orientation on $T \times G$: given coordinate charts of $T$ and $G$ compatible with the respective orientations, their ``Cartesian product'' is a coordinate chart of $T \times G$ compatible with its orientation (by definition). Moreover, $T \times G$ carries the product metric 
\begin{align*}
  \langle \cdot, \cdot \rangle &\dfn \pi_T^* \langle \cdot, \cdot \rangle^T +\pi_G^* \langle \cdot, \cdot \rangle^G
\end{align*}
where $\pi_T: T \times G \rarr T$ and $\pi_G: T \times G \rarr G$ are the natural projections. In other words, under the isomorphism $T_{(t,x)} (T \times G) \cong T_{t} T \oplus T_{x} G$ for $(t, x) \in T \times G$ we have
\begin{align*}
  \langle (u_1, v_1), (u_2, v_2) \rangle_{T_{(t,x)} (T \times G)} &\dfn \langle u_1, u_2 \rangle^T_{T_tT} + \langle v_1, v_2 \rangle^G_{T_x G}, \quad \forall u_1, u_2 \in T_tT, \ v_1, v_2 \in T_x G.
\end{align*}
Also, under the corresponding identification in the cotangent bundle $T_{(t,x)}^* (T \times G) \cong T_{t}^* T \oplus T_{x}^* G$ one can show that the following identity holds
\begin{align}
  \langle (\eta_1, \xi_1), (\eta_2, \xi_2) \rangle_{T^*_{(t,x)} (T \times G)} &= \langle \eta_1, \eta_2 \rangle^T_{T_{t}^* T} + \langle \xi_1, \xi_2 \rangle^G_{T_{x}^* G}, \quad \forall \eta_1, \eta_2 \in T_{t}^* T, \ \xi_1, \xi_2 \in T_{x}^* G, \label{eq:cotan_ip}
\end{align}
and also that for any $\psi \in \cinfty(T)$ and $\phi \in \cinfty(G)$ we have
\begin{align}
  \dd (\psi \otimes \phi)(t,x) &= \phi(x) \dd_T \psi(t) + \psi(t) \dd_G \phi(x) \quad \text{in $T^*_{(t,x)} (T \times G) \cong T^*_t T \oplus T^*_x G$} \label{eq:dotimes}
\end{align}
where $\dd$ (resp.~$\dd_T, \dd_G$) is the exterior derivative of $T \times G$ (resp.~$T, G$). If, moreover, we denote by $\dd V$ (resp.~$\dd V_T, \dd V_G$) the Riemannian volume form of $T \times G$ (resp.~$T, G$) with respect to the metric introduced above, then one can prove the following version of Fubini's Theorem:
\begin{Prop} \label{prop:fubini} For every $f \in \cinfty(T \times G)$ we have
  \begin{align*}
    \int_{T \times G} f(t,x) \ \dd V(t,x) &= \int_T \left( \int_G f(t, x) \ \dd V_G(x) \right) \dd V_T(t).
  \end{align*}
\end{Prop}

Given $P$ a differential operator on $T$ we denote by $P^\sharp$ its natural lift to $T \times G$. Formally speaking, if $f \in \cinfty(T \times G)$ then the action of $P^\sharp$ on $f$ is obtained by ``freezing'' $x \in G$, which yields a function $f(\cdot, x) \in \cinfty(T)$ on which we allow $P$ to act, thus producing a smooth function $P(f(\cdot, x))$ in $T$ depending on the variable point $x \in G$; allowing then $x$ to vary produces a smooth function $P^\sharp f$ in $T \times G$, and the mapping $P^\sharp: \cinfty(T \times G) \rarr \cinfty(T \times G)$ thus lifted can be shown to be a differential operator. Concisely:
\begin{align*}
  (P^\sharp f)(t, x) &\dfn \left( P [f( \cdot, x)] \right) (t), \quad (t,x) \in T \times G.
\end{align*}
Notice that if $\psi \in \cinfty(T)$ and $\phi \in \cinfty(G)$ then
\begin{align}
  P^\sharp(\psi \otimes \phi) &= (P \psi) \otimes \phi.  \label{eq:Psharp_on_tensors}
\end{align}
Of course the roles of $T$ and $G$ here are interchangeable.

Let $\Delta$ (resp.~$\Delta_T, \Delta_G$) be the Laplace-Beltrami operator on $T \times G$ (resp.~$T, G$) associated to the underlying metric(s) above.
\begin{Prop} \label{prop:delta12} $\Delta = \Delta_T^\sharp + \Delta_G^\sharp$ as differential operators on $T \times G$.
  \begin{proof} Using~\eqref{eq:cotan_ip},~\eqref{eq:dotimes} and the definition of the Laplace-Beltrami operator one can show that
    \begin{align*}
      \left \langle \Delta (\psi_1 \otimes \phi_1), \psi_2 \otimes \phi_2 \right \rangle_{L^2(T \times G)} &= \left \langle (\Delta_T^\sharp + \Delta_G^\sharp) (\psi_1 \otimes \phi_1), \psi_2 \otimes \phi_2 \right \rangle_{L^2(T \times G)}
    \end{align*}
    for every $\psi_1, \psi_2 \in \cinfty(T)$ and $\phi_1, \phi_2 \in \cinfty(G)$. But $\cinfty(T) \otimes \cinfty(G)$ is dense in $L^2(T \times G)$, hence
    \begin{align*}
      \Delta (\psi_1 \otimes \phi_1) &= (\Delta_T^\sharp + \Delta_G^\sharp) (\psi_1 \otimes \phi_1), \quad \forall \psi_1 \in \cinfty(T), \ \phi_1 \in \cinfty(G).
    \end{align*}

    Now any $f \in \cinfty(T \times G)$ can be approximated in $L^2(T \times G)$ by a sequence in $\cinfty(T) \otimes \cinfty(G)$ -- where $\Delta$ and $\Delta_T^\sharp + \Delta_G^\sharp$ match -- , and such convergence also holds in $\D'(T \times G)$, where these, as differentials operators in $T \times G$, are continuous. Therefore $\Delta f = (\Delta_T^\sharp + \Delta_G^\sharp) f$ for every $f \in \cinfty(T \times G)$.
  \end{proof}
\end{Prop}
    
For each $\mu \in \sigma(\Delta_T)$ (resp.~$\lambda \in \sigma(\Delta_G)$) we denote by $E^T_\mu \sset \cinfty(T)$ (resp.~$E^G_\lambda \sset \cinfty(G)$) the eigenspace of $\Delta_T$ (resp.~$\Delta_G$) associated to $\mu$ (resp.~$\lambda$). We choose bases for them
\begin{align*}
  \{ \psi^\mu_i \st 1 \leq i \leq d^T_\mu \}, &\quad \text{where $d^T_\mu \dfn \dim E^T_\mu$}, \\
  \{ \phi^\lambda_j \st 1 \leq j \leq d^G_\lambda \}, &\quad \text{where $d^G_\lambda \dfn \dim E^G_\lambda$},
\end{align*}
which are orthonormal w.r.t.~the inner products inherited from $L^2(T), L^2(G)$, respectively.
\begin{Prop} \label{prop:hilbert_basis_product} The set
  \begin{align*}
    \mathcal{S} &\dfn \left \{ \psi^\mu_i \otimes \phi^\lambda_j \st 1 \leq i \leq d^T_\mu, \ 1 \leq j \leq d^G_\lambda, \ \mu \in \sigma(\Delta_T), \ \lambda \in \sigma(\Delta_G) \right\}
  \end{align*}
  is a Hilbert basis for $L^2(T \times G)$.
  \begin{proof} A straightforward computation using Proposition~\ref{prop:fubini} proves that $\mathcal{S}$ is an orthonormal family.
    As for the density of $\Span_\C \mathcal{S}$ in $L^2(T \times G)$, given $f \in L^2(T \times G)$ and $\epsilon > 0$ we first select finitely many $\psi_k \in L^2(T)$, $\phi_k \in L^2(G)$, $k \in \{1, \ldots, r\}$, such that
    \begin{align*}
      \left \| f - \sum_{k = 1}^r \psi_k \otimes \phi_k \right \|_{L^2(T \times G)} &< \frac{\epsilon}{2}.
    \end{align*}
    Next, for each $k \in \{1, \ldots, r\}$ we select
    \begin{align*}
      \phi_k' \in \Span_\C \left \{ \phi^\lambda_j \st 1 \leq j \leq d^G_\lambda, \ \lambda \in \sigma(\Delta_G) \right\}
    \end{align*}
    such that $\| \phi_k - \phi_k' \|_{L^2(G)} < \epsilon/4r (1 + \| \psi_k \|_{L^2(T)})$, and then
    \begin{align*}
      \psi_k' \in \Span_\C \left \{ \psi^\mu_i \st 1 \leq i \leq d^T_\mu, \ \mu \in \sigma(\Delta_T) \right\}
    \end{align*}
    such that $\| \psi_k - \psi_k' \|_{L^2(T)} < \epsilon/4r(1 + \| \phi_k' \|_{L^2(G)})$. For instance, one may take for $\phi'_k$ (resp.~$\psi'_k$) a convenient finite sum in the orthogonal expansion of $\phi_k$ (resp.~$\psi_k$):
    \begin{align*}
      \phi_k = \sum_{\lambda \in \sigma(\Delta_G)} \sum_{j = 1}^{d^G_\lambda} \left \langle \phi_k, \phi^\lambda_j \right \rangle_{L^2(G)} \phi^\lambda_j &\left( \text{resp.}~\psi_k = \sum_{\mu \in \sigma(\Delta_T)} \sum_{i = 1}^{d^T_\mu} \left \langle \psi_k, \psi^\mu_i \right \rangle_{L^2(T)} \psi^\mu_i \right).
    \end{align*}
    A simple computation then shows that 
    \begin{align*}
      \left \| \sum_{k = 1}^r \psi_k \otimes \phi_k - \sum_{k = 1}^r \psi_k' \otimes \phi_k' \right \|_{L^2(T \times G)} &< \frac{\epsilon}{2}.
    \end{align*}
  \end{proof}
\end{Prop}

\begin{Prop} \label{prop:alphaissum} Every $\alpha \in \sigma(\Delta)$ is of the form $\alpha = \mu + \lambda$ for some $\mu \in \sigma(\Delta_T)$ and $\lambda \in \sigma(\Delta_G)$.
  \begin{proof} Suppose $\alpha \in \sigma(\Delta)$ i.e.~there exists a nonzero $f \in \cinfty(T \times G)$ such that $\Delta f = \alpha f$. We look at the series expansion of $f$ in terms of the Hilbert basis $\mathcal{S}$ by Proposition~\ref{prop:hilbert_basis_product} -- its convergence holds in $L^2(T \times G)$, hence also in $\D'(T \times G)$, where $\Delta$ is continuous -- while noticing that
    \begin{align*}
      \Delta \left( \psi^\mu_i \otimes \phi^\lambda_j \right) = (\Delta_T^\sharp + \Delta_G^\sharp) \left( \psi^\mu_i \otimes \phi^\lambda_j \right) = \left( \Delta_T \psi^\mu_i \right) \otimes \phi^\lambda_j + \psi^\mu_i \otimes \left( \Delta_G \phi^\lambda_j \right) = (\mu + \lambda) \psi^\mu_i \otimes \phi^\lambda_j
    \end{align*}
    (where we used Proposition~\ref{prop:delta12} and property~\eqref{eq:Psharp_on_tensors}): comparing the series expansions of $\Delta f$ and $\alpha f$ termwise leads us to the desired conclusion.
  \end{proof}
\end{Prop}

\begin{Rem} Given $\alpha \in \R_+$ the set
  \begin{align*}
    \mathcal{P}(\alpha) &\dfn \{ (\mu, \lambda) \in \sigma(\Delta_T) \times \sigma(\Delta_G) \st \mu + \lambda = \alpha \} 
  \end{align*}
  may contain more than one pair i.e.~in principle there may exist distinct $(\mu, \lambda), (\mu', \lambda') \in \sigma(\Delta_T) \times \sigma(\Delta_G)$ for which $\mu + \lambda = \mu' + \lambda'$. However such a set is necessarily finite, since both $\sigma(\Delta_T)$ and $\sigma(\Delta_G)$ are discrete and unbounded.
\end{Rem}

Using similar arguments as in the proof of Proposition~\ref{prop:alphaissum} the following can also be inferred:
\begin{Cor} \label{cor:relationshipeigens} The eigenspace of $\Delta$ associated to $\alpha \in \sigma(\Delta)$ is precisely
  \begin{align*}
    E_\alpha &= \bigoplus_{(\mu, \lambda) \in \mathcal{P}(\alpha)} E^T_{\mu} \otimes E^G_{\lambda}.
  \end{align*}
  An orthonormal basis for this space w.r.t.~the $L^2(T \times G)$ inner product is
  \begin{align*}
    \left\{ \psi^\mu_i \otimes \phi^\lambda_j  \st 1 \leq i \leq d^T_\mu, \ 1 \leq j \leq d^G_\lambda, \ (\mu, \lambda) \in \mathcal{P}(\alpha) \right\}. 
  \end{align*}
\end{Cor}

Now let $f \in \cinfty(T \times G)$ and, given $t \in T$, we once more regard $f(t, \cdot)$ as a smooth function on $G$, for which we consider its orthogonal expansion
\begin{align*}
  f(t, \cdot) &= \sum_{\lambda \in \sigma(\Delta_G)} \mathcal{F}^G_{\lambda} (f(t, \cdot))
\end{align*}
where $\mathcal{F}^G_{\lambda} (f(t, \cdot)) \in E^G_{\lambda}$ can be written, in terms of our previously chosen basis, as
\begin{align}
  \mathcal{F}^G_{\lambda} (f(t, \cdot)) = \sum_{j = 1}^{d^G_{\lambda}} \langle f(t, \cdot), \phi^\lambda_j \rangle_{L^2(G)} \phi^\lambda_j = \sum_{j = 1}^{d^G_{\lambda}} \left(\int_{G} f(t, x) \overline{\phi^\lambda_j(x)} \dd V_G(x) \right) \phi^\lambda_j . \label{eq:partial_ft_smooth}
\end{align}
Allowing now $t$ to vary in $T$ we see at once that for each given $\lambda \in \sigma(\Delta_G)$ the map
\begin{align*}
  t \in T &\longmapsto \mathcal{F}^G_{\lambda} (f(t, \cdot)) \in E^G_{\lambda} 
\end{align*}
is smooth, hence an element of $\cinfty(T; E^G_{\lambda}) \cong \cinfty(T) \otimes  E^{G}_{\lambda}$, which we denote by $\mathcal{F}^G_{\lambda} (f)$ or $\hat{f}(\cdot, \lambda)$ depending on the context. We can then consider the $E_\lambda^G$-valued orthogonal expansion w.r.t.~$\Delta_T$ of $\mathcal{F}^G_{\lambda} (f) \in \cinfty(T;E_\lambda^G)$, in the sense of Section~\ref{rem:ext_fourier_to_tensors}: given $\mu \in \sigma(\Delta_T)$ we have
\begin{align*}
  \mathcal{F}^T_{\mu} \mathcal{F}^G_{\lambda} (f) &= \mathcal{F}^T_{\mu} \left( \sum_{j = 1}^{d^G_{\lambda}} \left(\int_G f(\cdot, x) \overline{\phi^\lambda_j(x)} \dd V_G(x) \right) \otimes \phi^\lambda_j \right) \\
  &= \sum_{j = 1}^{d^G_{\lambda}} \mathcal{F}^T_{\mu} \left(\int_G f(\cdot, x) \overline{\phi^\lambda_j(x)} \dd V_G(x) \right) \otimes \phi^\lambda_j  \\
  &= \sum_{j = 1}^{d^G_{\lambda}} \left( \sum_{i = 1}^{d^T_{\mu}} \left( \int_{T} \left(\int_{G} f(t,x) \overline{\phi^\lambda_j(x)} \dd V_G(x) \right) \overline{\psi^\mu_i(t)} \dd V_T(t) \right) \psi^\mu_i \right) \otimes \phi^\lambda_j \\
  &= \sum_{i = 1}^{d^T_{\mu}} \sum_{j = 1}^{d^G_{\lambda}} \left \langle f, \psi^\mu_i \otimes \phi^\lambda_j \right \rangle_{L^2(T \times G)} \psi^\mu_i \otimes \phi^\lambda_j
\end{align*}
which is an element of $E^T_{\mu} \otimes E^G_{\lambda}$. By Corollary~\ref{cor:relationshipeigens} this is nothing but a portion of $\mathcal{F}_\alpha(f)$, and we actually conclude that
\begin{align}
  \mathcal{F}_\alpha(f) &= \sum_{(\mu, \lambda) \in \mathcal{P}(\alpha)} \mathcal{F}^T_{\mu} \mathcal{F}^G_{\lambda} (f), \quad \forall \alpha \in \sigma(\Delta). \label{eq:Fourier_totalfrompartials}
\end{align}

On time, we notice that for every $\lambda \in \sigma(\Delta_G)$ we have
\begin{align}
  \cinfty(T; E_\lambda^G) &= \left \{ f \in \cinfty(T \times G) \st \Delta_G^\sharp f = \lambda f \right \} \label{eq:eigenspaces_partial}
\end{align}
-- which can be easily checked by analyzing the orthogonal expansion of any $f \in \cinfty(T \times G)$ w.r.t.~our Hilbert basis $\mathcal{S}$ and reasoning as in Proposition~\ref{prop:alphaissum} -- and that $\mathcal{F}_\lambda^G: \cinfty(T \times G) \rarr \cinfty(T; E_\lambda^G)$ is a projection. Indeed, given $f \in \cinfty(T \times G)$ and $\psi \in \cinfty(T; E_\lambda^G)$, which we write
\begin{align*}
  \psi &= \sum_{j' = 1}^{d_\lambda^G} \psi_{j'} \otimes \phi_{j'}^\lambda, \quad \psi_{j'} \in \cinfty(T),
\end{align*}
we have by~\eqref{eq:partial_ft_smooth} that
\begin{align*}
  \left \langle \mathcal{F}_\lambda^G(f), \psi \right \rangle_{L^2(T \times G)} &= \int_T \int_G \sum_{j = 1}^{d^G_{\lambda}} \left(\int_{G} f(t, x) \overline{\phi^\lambda_j(x)} \dd V_G(x) \right) \phi^\lambda_j(y) \sum_{j' = 1}^{d_\lambda^G} \overline{\psi_{j'}(t) \phi_{j'}^\lambda(y)} \dd V_G(y) \dd V_T(t) \\
  &= \sum_{j = 1}^{d^G_{\lambda}} \int_T \int_{G} f(t, x) \overline{\phi^\lambda_j(x)} \dd V_G(x) \overline{\psi_{j}(t)} \dd V_T(t) \\
  &= \int_T \int_{G} f(t, x) \overline{\psi(t,x)} \dd V_G(x) \dd V_T(t) \\
  &= \left \langle f, \psi \right \rangle_{L^2(T \times G)}
\end{align*}
that is, $\mathcal{F}_\lambda^G(f)$ is characterized as the unique element in $\cinfty(T; E_\lambda^G)$ with the property that
\begin{align}
  \left \langle \mathcal{F}_\lambda^G(f), \psi \right \rangle_{L^2(T \times G)} &= \left \langle f, \psi \right \rangle_{L^2(T \times G)}, \quad \forall \psi \in \cinfty(T; E_\lambda^G). \label{eq:charac_Fourier_part}
\end{align}
It follows at once that $\mathcal{F}_\lambda^G: \cinfty(T \times G) \rarr \cinfty(T; E_\lambda^G)$ acts as the identity on $\cinfty(T; E_\lambda^G)$.

In order to extend the definitions above to distributions $f \in \D'(T \times G)$, given $\lambda \in \sigma(\Delta_G)$ we expect to construct an object
\begin{align*}
  \mathcal{F}^G_{\lambda} (f) \in \D'(T; E^{G}_{\lambda}).
\end{align*}
First of all, notice that we may identify $(E^{G}_{\lambda})^*$ with $E^{G}_{\lambda}$ itself by means of the anti-Riesz isomorphism provided by its Hermitian product
\begin{align*}
  \phi \in E^{G}_{\lambda} &\longmapsto \langle \cdot, \bar{\phi} \rangle_{L^2(G)} \in (E^{G}_{\lambda})^*
\end{align*}
for which $\left\{ \overline{\phi^\lambda_j} \st 1 \leq j \leq d^G_\lambda \right\}$ is the corresponding dual basis. Thus an element $g \in \cinfty(T; (E^G_\lambda)^*)$ can be written uniquely as
\begin{align*}
  g &= \sum_{j = 1}^{d^G_\lambda} g_j \otimes \overline{\phi^\lambda_j}, \quad g_j \in \cinfty(T).
\end{align*}
Note that when $f \in \cinfty(T\times G)$ we have seen~\eqref{eq:charac_Fourier_part} that we can apply $\mathcal{F}_\lambda^{G}(f)$, as an element of $\D'(T; E_{\lambda}^{G})$, to $g \in \cinfty(T; (E^G_\lambda)^*)$ and obtain
\begin{align*}
    \langle \mathcal{F}^G_{\lambda} (f), g \rangle = \sum_{j = 1}^{d^G_\lambda} \left( \int_{T} \int_{G} f(t, x) g_j(t)\overline{\phi^\lambda_j(x)} \dd V_G(x) \dd V_T(t) \right) \langle \phi_j^{\lambda}, \phi_j^{\lambda}\rangle_{L^{2}(G)} = \langle f, g \rangle
\end{align*}

Now given $f \in \D'(T \times G)$ its projection $\mathcal{F}^G_{\lambda} (f) \in \D'(T; E^G_\lambda)$ should also be written uniquely as
\begin{align*}
  \mathcal{F}^G_{\lambda} (f) &= \sum_{j = 1}^{d^G_\lambda} F_{j} \otimes \phi^\lambda_{j}, \quad F_{j} \in \D'(T),
\end{align*}
where, as one can now easily guess,
\begin{align*}
  \langle F_j, \psi \rangle &\dfn  \langle f, \psi \otimes \overline{\phi_j^\lambda} \rangle, \quad \forall \psi \in \cinfty(T).
\end{align*}
We have thus defined a linear map
\begin{align*}
  \mathcal{F}^G_{\lambda}: \D'(T \times G) \longrightarrow \D'(T; E^G_\lambda)
\end{align*}
which is essentially the transpose of the inclusion map $\cinfty(T; E^G_\lambda) \hookrightarrow \cinfty(T \times G)$. We can now characterize smoothness in terms of the double partial Fourier maps.
\begin{Prop} \label{prop:pw_product} A distribution $f \in \D'(T \times G)$ is smooth if and only if for every $s > 0$ there exists $C > 0$ such that
  \begin{align*}
    \| \mathcal{F}^T_{\mu} \mathcal{F}^G_{\lambda} (f) \|_{L^2(T \times G)} &\leq C (1 + \mu + \lambda)^{-s}, \quad \forall (\mu, \lambda) \in \sigma(\Delta_T) \times \sigma(\Delta_G).
  \end{align*}
  \begin{proof} Denoting by $\mathcal{F}_\alpha : \D'(T \times G) \rarr E_\alpha$ the ``total'' Fourier projection map, for $\alpha \in \sigma(\Delta)$, we know (Proposition~\ref{prop:charac_smoothness_fourier_proj}) that $f$ is smooth if and only if for each $s > 0$ there exists $C > 0$ such that
    \begin{align*}
      \| \mathcal{F}_{\alpha} (f) \|_{L^2(T \times G)} &\leq C (1 + \alpha)^{-s}, \quad \forall \alpha \in \sigma(\Delta).
    \end{align*}
    But thanks to~\eqref{eq:Fourier_totalfrompartials} -- which also holds for distributions -- we have
    \begin{align}
      \| \mathcal{F}_\alpha(f) \|_{L^2(T \times G)}^2  &= \sum_{(\mu, \lambda) \in \mathcal{P}(\alpha)} \| \mathcal{F}^T_{\mu} \mathcal{F}^G_{\lambda} (f) \|_{L^2(T \times G)}^2, \quad \forall \alpha \in \sigma(\Delta), \label{eq:Fourier_totalfrompartialsnorms}
    \end{align}
    hence assuming $f \in \cinfty(T \times G)$ we have, for $(\mu, \lambda) \in \mathcal{P}(\alpha)$,
    \begin{align*}
      \| \mathcal{F}^T_{\mu} \mathcal{F}^G_{\lambda} (f) \|_{L^2(T \times G)} \leq \| \mathcal{F}_{\alpha} (f) \|_{L^2(T \times G)} \leq C (1 + \alpha)^{-s} = C (1 + \mu + \lambda)^{-s}
    \end{align*}
    and the conclusion follows since every $(\mu, \lambda) \in \sigma(\Delta_T) \times \sigma(\Delta_G)$ belongs to some $\mathcal{P}(\alpha)$.

    As for the converse, by hypothesis for each $s > 0$ there exists $C > 0$ such that
    \begin{align*}
      \| \mathcal{F}^T_{\mu} \mathcal{F}^G_{\lambda} (f) \|_{L^2(T \times G)} &\leq C (1 + \mu + \lambda)^{-s - 2n - 2m}, \quad \forall (\mu, \lambda) \in \sigma(\Delta_T) \times \sigma(\Delta_G),
    \end{align*}
    where $n = \dim T$, $m = \dim G$. Then by~\eqref{eq:Fourier_totalfrompartialsnorms} we have for $\alpha \in \sigma(\Delta)$
    \begin{align*}
      \| \mathcal{F}_\alpha(f) \|_{L^2(T \times G)} &\leq |\mathcal{P}(\alpha)| C (1 + \alpha)^{-s - 2n - 2m}
    \end{align*}
    but from Corollary~\ref{cor:relationshipeigens} and Weyl's asymptotic formula~\eqref{eq:weyl} we have
    \begin{align*}
      |\mathcal{P}(\alpha)| \leq \sum_{(\mu, \lambda) \in \mathcal{P}(\alpha)} d^T_\mu d^G_\lambda = \dim E_\alpha \leq C' \alpha^{2(n + m)}
    \end{align*}
    where $C' > 0$ is independent of $\alpha$, from which it follows that
    \begin{align*}
      \| \mathcal{F}_\alpha(f) \|_{L^2(T \times G)} &\leq C C' (1 + \alpha)^{-s}, \quad \forall \alpha \in \sigma(\Delta).
    \end{align*}
  \end{proof}
\end{Prop}
The next two corollaries of Proposition~\ref{prop:pw_product} are fundamental to our approach later on. Before we state (and prove) them, we will need the following remark, which basically reads: ``all natural definitions of the $L^2$ inner product on $L^2(T) \otimes E^G_\lambda$ are the same''.

\begin{Rem} \label{rem:equivL2lambda} For $f, g \in \cinfty(T; E^G_\lambda)$ given by
  \begin{align*}
    f &= \sum_{i = 1}^{d^G_\lambda} f_i \otimes \phi^\lambda_i, \quad f_i \in \cinfty(T), \\
    g &= \sum_{i' = 1}^{d^G_\lambda} g_{i'} \otimes \phi^\lambda_{i'}, \quad g_{i'} \in \cinfty(T),
  \end{align*}
  we have by Proposition~\ref{prop:fubini}
  \begin{align*}
    \langle f, g \rangle_{L^2(T \times G)} = \int_{T \times G} \sum_{i, i' = 1}^{d^G_\lambda} f_i(t)\phi^\lambda_i(x) \overline{g_{i'}(t)\phi^\lambda_{i'}(x)} \dd V(t,x) = \sum_{i = 1}^{d^G_\lambda} \langle f_i, g_i \rangle_{L^2(T)}.
  \end{align*}
 Moreover, we have
  \begin{align*}
    \mathcal{F}^T_\mu (f) &= \sum_{i = 1}^{d^G_\lambda}  \mathcal{F}^T_\mu(f_i) \otimes \phi^\lambda_i, \quad \forall \mu \in \sigma(\Delta_T)
  \end{align*}
  hence
  \begin{align*}
    \| f \|_{L^2(T \times G)}^2 = \sum_{i = 1}^{d^G_\lambda} \| f_i \|_{L^2(T)}^2 = \sum_{i = 1}^{d^G_\lambda} \sum_{\mu \in \sigma(\Delta_T)} \|  \mathcal{F}^T_\mu(f_i) \|_{L^2(T)}^2 = \sum_{\mu \in \sigma(\Delta_T)} \| \mathcal{F}^T_\mu(f) \|_{L^2(T \times G)}^2.
  \end{align*}
\end{Rem}

\begin{Cor} \label{cor:partial_smoothness} If $f \in \cinfty(T \times G)$ then for every $s > 0$ there exists $C > 0$ such that
  \begin{align}
    \| \mathcal{F}^G_\lambda(f) \|_{L^2(T \times G)} &\leq C (1 + \lambda)^{-s}, \quad \forall \lambda \in \sigma(\Delta_G). \label{eq:partial_smoothness}
  \end{align}
  \begin{proof} By the computations done in Remark~\ref{rem:equivL2lambda} we have
    \begin{align*}
      \| \mathcal{F}^G_\lambda (f) \|_{L^2(T \times G)}^2 = \sum_{\mu \in \sigma(\Delta_T)} \| \mathcal{F}^T_\mu \mathcal{F}^G_\lambda (f) \|_{L^2(T \times G)}^2.
    \end{align*}
    By Proposition~\ref{prop:pw_product} for each $s > 0$ there exists $C > 0$ such that
    \begin{align*}
      \| \mathcal{F}^T_{\mu} \mathcal{F}^G_{\lambda} (f) \|_{L^2(T \times G)} &\leq C (1 + \mu + \lambda)^{-s -n}, \quad \forall (\mu, \lambda) \in \sigma(\Delta_T) \times \sigma(\Delta_G)
    \end{align*}
    where $n = \dim T$, and so
    \begin{align*}
      \| \mathcal{F}^G_\lambda (f) \|_{L^2(T \times G)}^2 \leq \sum_{\mu \in \sigma(\Delta_T)} C^2 (1 + \mu + \lambda)^{-2s-2n} \leq C^2 (1 + \lambda)^{-2s} \sum_{\mu \in \sigma(\Delta_T)} (1 + \mu)^{-2n}
    \end{align*}
    where the last series converges thanks to Weyl's asymptotic formula~\eqref{eq:weyl} for $\Delta_T$. 
  \end{proof}
\end{Cor}

\begin{Cor} \label{cor:partial_smoothness_converse} If $f \in \D'(T \times G)$ is such that
  \begin{enumerate}
  \item for every $s > 0$ there exists $C > 0$ such that~\eqref{eq:partial_smoothness} holds and
  \item for every $s' > 0$ there exist $C' > 0$ and $\theta \in (0, 1)$ such that
    \begin{align}
      \| \mathcal{F}^T_\mu \mathcal{F}^G_\lambda (f) \|_{L^2(T \times G)} &\leq C' (1 + \mu + \lambda)^{-s'}, \quad \forall (\mu, \lambda) \in \Lambda_\theta \label{eq:conethetaestimate}
    \end{align}
    where
    \begin{align}
      \Lambda_\theta &\dfn \{ (\mu, \lambda) \in \sigma(\Delta_T) \times \sigma(\Delta_G) \st (1 + \lambda) \leq (1 + \mu)^\theta \}. \label{eq:Atheta}
    \end{align}
  \end{enumerate}
  Then $f \in \cinfty(T \times G)$.
  \begin{proof} Let $\Lambda_\theta^c \sset \sigma(\Delta_T) \times \sigma(\Delta_G)$ denote the complement of $\Lambda_\theta$. For $(\mu, \lambda) \in \Lambda_\theta^c$ we have
    \begin{align*}
      1 + \mu + \lambda < (1 + \lambda)^{\frac{1}{\theta}} + \lambda \leq (1 + \lambda)^{1 + \frac{1}{\theta}} \leq (1 + \lambda)^{\frac{2}{\theta}}
    \end{align*}
    since $1/\theta > 1$. Therefore, given $s' > 0$ we define $s \dfn 2 \theta^{-1} s'$, hence for $(\mu, \lambda) \in  \Lambda_\theta^c$ we have
    \begin{align*}
      (1 + \lambda)^{-s} \leq  (1 + \mu + \lambda)^{-\frac{\theta s}{2}} = (1 + \mu + \lambda)^{-s'}.
    \end{align*}
    Let then $C, C' > 0$ be such that~\eqref{eq:partial_smoothness} and~\eqref{eq:conethetaestimate} hold, hence
    \begin{align*}
      \| \mathcal{F}^T_\mu \mathcal{F}^G_\lambda (f) \|_{L^2(T \times G)} &\leq
      \begin{cases}
        C (1 + \mu + \lambda)^{-s'}, &\text{in $\Lambda^c_\theta$}, \\ 
        C' (1 + \mu + \lambda)^{-s'}, &\text{in $\Lambda_\theta$}.
      \end{cases}
    \end{align*}    
    Combining both estimates, it follows from Proposition~\ref{prop:pw_product} that $f \in \cinfty(T \times G)$.
  \end{proof}
\end{Cor}

Before we end this section we will prove a result about LPDOs which commute with one of the partial Laplace-Beltrami operators on $T \times G$: such LPDOs will also commute with the partial Fourier projection map associated to the corresponding factor. This is a key property that all of our operators of interest in the forthcoming sections will enjoy.
\begin{Prop} \label{prop:invopscommft} Let $P$ be a LPDO in $T \times G$ which commutes with $\Delta_G^\sharp$. If $u \in \D'(T \times G)$ then $\mathcal{F}^G_\lambda (Pu) = P \mathcal{F}^G_\lambda(u)$ for every $\lambda \in \sigma(\Delta_G)$.
  \begin{proof} We will be content to prove the assertion when $u$ is smooth. First, notice that $P$ maps $\cinfty(T; E_\lambda^G)$ to itself: indeed, if $f \in \cinfty(T; E_\lambda^G)$ then by~\eqref{eq:eigenspaces_partial}
    \begin{align*}
      \Delta_G^\sharp f = \lambda f &\Longrightarrow \Delta_G^\sharp (Pf) = P(\Delta_G^\sharp f) = \lambda (Pf)
    \end{align*}
    from which we conclude that $Pf \in \cinfty(T; E_\lambda^G)$. We claim that $P^*$ -- the formal adjoint of $P$ -- also commutes with $\Delta_G^\sharp$: for $f, g \in \cinfty(T \times G)$ we have $\langle \Delta_G^\sharp f, g \rangle_{L^2(T \times G)} = \langle f, \Delta_G^\sharp g \rangle_{L^2(T \times G)}$ (check this first for $f, g \in \cinfty(T) \otimes \cinfty(G)$ using~\eqref{eq:Psharp_on_tensors} and Proposition~\ref{prop:fubini} and then use a density argument) hence
    \begin{align*}
      \left \langle P^* \Delta_G^\sharp f, g \right \rangle_{L^2(T \times G)} = \left \langle f, \Delta_G^\sharp P g \right \rangle_{L^2(T \times G)} = \left \langle f, P \Delta_G^\sharp g \right \rangle_{L^2(T \times G)} = \left \langle \Delta_G^\sharp P^* f, g \right \rangle_{L^2(T \times G)}
    \end{align*}
    and since this holds for all $f, g \in \cinfty(T \times G)$ our claim follows. In particular, $P^*$ also preserves $\cinfty(T; E_\lambda^G)$ for each $\lambda \in \sigma(\Delta_G)$.

    Now for $u \in \cinfty(T \times G)$ we have, for all $\psi \in \cinfty(T; E_\lambda^G)$,
    \begin{align*}
      \left \langle \mathcal{F}_\lambda^G (Pu), \psi \right \rangle_{L^2(T \times G)} = \left \langle Pu, \psi \right \rangle_{L^2(T \times G)} = \left \langle u, P^* \psi \right \rangle_{L^2(T \times G)} = \left \langle \mathcal{F}_\lambda^G (u), P^*\psi \right \rangle_{L^2(T \times G)}
    \end{align*}
    thanks to~\eqref{eq:charac_Fourier_part}: notice that in the last equality we used that $P^*\psi \in \cinfty(T; E_\lambda^G)$. After a final transposition we conclude that
    \begin{align*}
      \left \langle \mathcal{F}_\lambda^G (Pu), \psi \right \rangle_{L^2(T \times G)} &= \left \langle P \mathcal{F}_\lambda^G (u), \psi \right \rangle_{L^2(T \times G)}, \quad \forall \psi \in \cinfty(T; E_\lambda^G),
    \end{align*}
    which yields our conclusion since both $\mathcal{F}_\lambda^G (Pu)$ and $P \mathcal{F}_\lambda^G (u)$ belong to $\cinfty(T; E_\lambda^G)$.
  \end{proof}
\end{Prop}
\section{A class of sublaplacians on product manifolds} \label{sec:sublaplacians}

From now on we will assume some extra structure in the environment postulated in the previous sections: namely, $G$ will be a Lie group (with $\dim G = m$), while $T$ will remain a smooth manifold (with $\dim T = n$), both of them compact, connected and oriented. We impose no conditions on the Riemannian metric on $T$, but will require the one on $G$ to be $\ad$-invariant~\eqref{eq:adinvariantmetric}. We denote by $\gr{g}$ the Lie algebra of $G$.


Let $\gr{a}: T \rarr \gr{g}$ be a smooth map. If $\vv{X}_1, \ldots, \vv{X}_m \in \gr{g}$ is a basis of left-invariant vector fields then
\begin{align*}
  \gr{a}(t) &= \sum_{j = 1}^m a_j(t) \vv{X}_j, \quad t \in T,
\end{align*}
where $a_1, \ldots, a_m \in \cinfty(T; \R)$ are uniquely determined. We thus regard $\gr{a}$ as a first-order LPDO on $T \times G$, which we may sometimes write $\gr{a}(t,\vv{X})$ when we want to stress this point of view. Notice that
\begin{align*}
  \gr{a}(t, \vv{X}) (\psi \otimes \phi) &= \sum_{j = 1}^m (a_j \psi) \otimes (\vv{X}_j \phi), \quad \forall \psi \in \D'(T), \ \phi \in \D'(G),
\end{align*}
hence in particular $\gr{a}(t, \vv{X}) (\psi \otimes 1_G) = 0$ for every $\psi \in \D'(T)$.

Now we introduce the class of LPDOs on $T \times G$ which is the main theme of the present work. Define
\begin{align}
  P &\dfn Q^\sharp - \sum_{\ell = 1}^N \left( \gr{a}_\ell(t, \vv{X}) + \vv{W}_\ell^\sharp \right)^2 \label{eq:Pdef}
\end{align}
where $\gr{a}_1, \ldots, \gr{a}_N: T \rarr \gr{g}$ are smooth maps, $\vv{W}_1, \ldots, \vv{W}_N$ are real, smooth vector fields on $T$ and $Q$ is a real, positive semidefinite LPDO on $T$ -- meaning that $\langle Q \psi, \psi \rangle_{L^2(T)} \geq 0$ for every $\psi \in \cinfty(T)$ -- which is a wildcard in our model: we will slowly add hypotheses to it, but for now we will assume that
\begin{align}
  \tilde{P} &\dfn Q - \sum_{\ell = 1}^N \vv{W}_\ell^2 \label{eq:tildeP}
\end{align}
is a second-order LPDO on $T$ that kills constants (i.e.~has no zero order term). The main examples we will explore afterwards are $Q = \Delta_T$ and $Q = 0$. 
Our aim in this work is to study necessary and sufficient conditions for an operator $P$ as above to be \emph{globally hypoelliptic}, or~$\mathrm{(GH)}$ for short, in $T \times G$:
\begin{align*}
  \forall u \in \D'(T \times G), \ Pu \in \cinfty(T \times G) &\Longrightarrow u \in \cinfty(T \times G).
\end{align*}

Since $\gr{a}_1, \ldots, \gr{a}_N: T \rarr \gr{g}$ are smooth, for each $\ell \in \{1, \ldots, N\}$ we may write
\begin{align}
  \mathfrak{a}_\ell(t) &= \sum_{j = 1}^m a_{\ell j}(t) \vv{X}_j, \quad t \in T, \label{eq:aellcoordinates}
\end{align}
with $a_{\ell 1}, \ldots, a_{\ell m} \in \cinfty(T; \R)$. Then given $\psi \in \D'(T)$ and $\phi \in \D'(G)$ we have, unwinding the square in the definition of $P$,
\begin{align}
  P(\psi \otimes \phi) &= (\tilde{P}\psi) \otimes \phi - \sum_{\ell = 1}^N \left( \sum_{j,j' = 1}^m (a_{\ell j'} a_{\ell j} \psi) \otimes (\vv{X}_{j'} \vv{X}_j \phi) + \sum_{j = 1}^m \left( (2 a_{\ell j} \vv{W}_\ell + \vv{W}_\ell a_{\ell j}) \psi \right) \otimes (\vv{X}_j \phi) \right). \label{eq:Pontensors}
\end{align}
Roughly speaking, $P$ has ``separated variables'' with ``constant coefficients'' on $G$, and hence behaves nicely under partial the Fourier projection maps on that factor. Rigorously, operators such as $Q^\sharp$ and $\vv{W}_\ell^\sharp$ commute with $\Delta_G^\sharp$, as they act on independent variables, but so does $\gr{a}_\ell(t, \vv{X})$ since each $\vv{X}_j$ commutes with $\Delta_G$ (as pointed out at the end of Section~\ref{sec:introLG}). Thus $P$ also commutes with $\Delta_G^\sharp$; to all of them, Proposition~\ref{prop:invopscommft} applies.

On time, we point out the following energy identity, which will be fundamental later on. Its proof is purely computational, and we leave it to the reader.
\begin{Lem} \label{lem:energy_identity} Let $P$ be as in~\eqref{eq:Pdef}. If we further assume that $\vv{W}_1, \ldots, \vv{W}_N$ are skew-symmetric on $T$ then for each $\lambda \in \sigma(\Delta_G)$ we have
  \begin{align*}
    \left \langle P \psi, \psi \right \rangle_{L^2(T \times G)} &= \left \langle Q^\sharp \psi, \psi \right \rangle_{L^2(T \times G)} + \sum_{\ell = 1}^N \left \| \vv{Y}_\ell \psi \right \|_{L^2(T \times G)}^2, \quad \forall \psi \in \cinfty(T; E_\lambda^G),
  \end{align*}
  where $\vv{Y}_\ell \dfn \gr{a}_\ell(t, \vv{X}) + \vv{W}_\ell^\sharp$ for $\ell \in \{1, \ldots, N\}$.
\end{Lem}
\subsection{Main Results}

We start by discussing necessary conditions for global hypoellipticity of $P$ in~\eqref{eq:Pdef}. 
\begin{Prop} \label{prop:first_nec_condition} If $P$ is~$\mathrm{(GH)}$ in $T \times G$ then $\tilde{P}$ is~$\mathrm{(GH)}$ in $T$.
  \begin{proof} Let $u \in \D'(T)$ be such that $\tilde{P}u \in \cinfty(T)$. Then by~\eqref{eq:Pontensors} we have that $P (u \otimes 1_G) = (\tilde{P}u) \otimes 1_G$ is smooth on $T \times G$, hence by hypothesis $u \otimes 1_G \in \cinfty(T \times G)$ -- which can only happen if $u \in \cinfty(T)$.
  \end{proof}
\end{Prop}

Motivated by this remark, we shall be mostly concerned with the case when $\tilde{P}$ is an \emph{elliptic} operator in $T$, a simplifying assumption that will allow us to make use of microlocal methods. Now we come to our second necessary condition for global hypoellipticity of $P$.
\begin{Thm} \label{thm:PGHnecessaell} If $P$ is~$\mathrm{(GH)}$ in $T \times G$ then the following regularity condition holds:
  \begin{align}
    \forall u \in \D'(G), \ \mathfrak{a}_\ell(t,\vv{X})(1_T \otimes u) \in \cinfty(T \times G) \ \forall \ell \in \{1, \ldots, N\} &\Longrightarrow u \in \cinfty(G). \label{eq:reg_aa}
  \end{align}
\end{Thm}
Its proof is not as simple: we postpone it to Section~\ref{sec:classofsystem}. Under additional conditions, we will see that the necessary conditions in Proposition~\ref{prop:first_nec_condition} and Theorem~\ref{thm:PGHnecessaell} are also sufficient. But first, let us restate condition~\eqref{eq:reg_aa} in terms of a system  of left-invariant vector fields on $G$. To do so, we must recall the notion of global hypoellipticity for such systems:
\begin{Def} \label{def:gh_systems} Let $M$ be a smooth, compact manifold as in Section~\ref{sec:preliminaries}. A family $\mathcal{L}$ of smooth vector fields on $M$ is said to be \emph{globally hypoelliptic} in $M$ -- $\mathrm{(GH)}$ for short -- if for every $u \in \D'(M)$ we have
  \begin{align*}
    \vv{L} u \in \cinfty(M), \ \forall \vv{L} \in \mathcal{L} &\Longrightarrow u \in \cinfty(M).
  \end{align*}
\end{Def}
From now on we denote by $\mathcal{L}$ the system of vector fields on $G$ defined as follows:
\begin{align}
  \mathcal{L} &\dfn \bigcup_{\ell = 1}^N \ran \mathfrak{a}_\ell \sset \gr{g}. \label{eq:sys_fromaell}
\end{align}
Thus a left-invariant vector field $\vv{L}$ belongs to $\mathcal{L}$ if and only if there exist $\ell \in \{1, \ldots, N\}$ and $t \in T$ such that $\vv{L} = \gr{a}_\ell(t)$. Moreover, for each $\ell \in \{1, \ldots, N\}$ we let
\begin{align}
  \mathcal{L}_\ell &\dfn \Span_\R \ran \mathfrak{a}_\ell \sset \gr{g}. \label{eq:Lell}
\end{align}
We will prove in Proposition~\ref{Pro:66impliesLGH} that condition~\eqref{eq:reg_aa} is equivalent to ask that $\mathcal{L}$ is~$\mathrm{(GH)}$ in $G$. In Section~\ref{sec:nsa_vectors} we explore in detail such condition when $G$ is a torus and equate it with the notion of non-simultaneous approximability of a collection of vectors, a Diophantine condition already known to be connected with global hypoellipticity of operators like~\eqref{eq:Pdef} when both $T$ and $G$ are tori~\cite{bfp17}.

When $Q = \Delta_T$, the Laplace-Beltrami operator in $T$, we can state our sufficiency result as follows:
\begin{Thm} \label{thm:thm15} Let
  \begin{align}\label{eq:PdefLaplaceBeltrami}
    P &= \Delta_T^\sharp - \sum_{\ell = 1}^N \left( \mathfrak{a}_\ell(t, \vv{X}) + \vv{W}_\ell^\sharp \right)^2
  \end{align}
 and suppose that $\vv{W}_1, \ldots, \vv{W}_N$ are skew-symmetric real vector fields in $T$. Assume moreover that:
  \begin{enumerate}
  \item \label{it:thm15_hyp1} For each given $\ell \in \{1, \ldots, N\}$, we have that $\mathfrak{a}_\ell(t_1), \mathfrak{a}_\ell(t_2)$ commute as vector fields in $G$, for any $t_1, t_2 \in T$. In other words, each $\mathcal{L}_\ell \sset \gr{g}$ as defined in~\eqref{eq:Lell} spans a commutative Lie subalgebra.
  \item \label{it:thm15_hyp2} The system $\mathcal{L} \sset \gr{g}$ in~\eqref{eq:sys_fromaell} is~$\mathrm{(GH)}$ in $G$.
  \end{enumerate}
  Then $P$ is~$\mathrm{(GH)}$ in $T \times G$. Furthermore, if $R$ is a LPDO in $T \times G$ of the form
  \begin{align}
    R &\dfn - \sum_{\kappa=1}^{M} \left( \gr{b}_\kappa(t, \vv{X}) + \vv{V}_\kappa^\sharp \right)^2 \label{eq:Rdef}
  \end{align}
  where $\vv{V}_1, \ldots, \vv{V}_M$ are skew-symmetric real vector fields in $T$ and $\gr{b}_{1}, \ldots, \gr{b}_{M} \in \cinfty(T; \mathfrak{g})$ do not necessarily satisfy the commutativity condition above, then $P_0 \dfn P + R$ is also~$\mathrm{(GH)}$ in $T \times G$. 
\end{Thm}
Note that, for any $\ell \in \{1, \ldots, N\}$, we can assume that $t \in T \mapsto \gr{a}_\ell(t) \in \gr{g}$ is not identically zero. For operators $P$ as in~\eqref{eq:PdefLaplaceBeltrami} we have that
\begin{align*}
  \tilde{P} &= \Delta_T - \sum_{\ell = 1}^N \vv{W}_\ell^2
\end{align*}
is elliptic in $T$, as we show below. Additionally, note that if $G = \TT^m$ then $\mathcal{L}_\ell$ is always commutative, so Proposition~\ref{prop:first_nec_condition}, Theorem~\ref{thm:PGHnecessaell} and Theorem~\ref{thm:thm15} together yield Theorem~\ref{Thm:Toruscase}, hence our result generalizes~\cite[Theorem~1.5]{bfp17}.

Let us also point out that we were able to prove global hypoellipticity of $P$~\eqref{eq:Pdef} in Theorem~\ref{thm:thm19} and in Theorem~\ref{thm:albanese} when $Q$ is any positive semidefinite operator in $T$, where, on the other hand, we impose more restrictive assumptions on the vector fields $\gr{a}_\ell(t, \vv{X})$, for $\ell \in \{1, \ldots, N\}$.

\section{Consequences of the ellipticity of $\tilde{P}$ on the Fourier projections}

Let us start recalling some basic results of elliptic operators. We evaluate the principal symbol of $\tilde{P}$~\eqref{eq:tildeP} at $(t_0, \tau_0) \in T^* T \setminus 0$ by taking any $\psi \in \cinfty(T; \R)$ such that $\dd_T \psi(t_0) = \tau_0$: we have
  \begin{align*}
    \tilde{P}_2(t_0, \tau_0) = \lim_{\rho \to \infty} \rho^{-2} e^{-i \rho \psi} \left. \left( Q(e^{i \rho \psi}) - \sum_{\ell = 1}^N \vv{W}_\ell^2 (e^{i \rho \psi}) \right) \right|_{t_0} = Q_2(t_0, \tau_0) + \sum_{\ell = 1}^N (\vv{W}_\ell \psi)(t_0)^2.
  \end{align*}
  In particular, if $Q_2$ is a non-negative function and the system of vector fields $\vv{W}_1, \ldots, \vv{W}_N$ is elliptic in $T$ then certainly $\tilde{P}$ is elliptic. If, on the other hand, $Q = \Delta_T$, then $Q_2$ may be evaluated by means of the local expression of the Laplace-Beltrami operator: in an oriented coordinate chart $(U; t_1, \ldots, t_n)$ of $T$ centered at $t_0$ it is
  \begin{align*}
    \Delta_T \psi &= - \frac{1}{\sqrt{\det g}} \sum_{j,k = 1}^n \frac{\del}{\del t_j} \left( g^{jk} \sqrt{\det g} \frac{\del \psi}{\del t_k} \right), \quad \psi \in \cinfty(U).
  \end{align*}
  Here, $g = (g_{jk}) \in \cinfty(U, \GL_n(\R))$ is defined by
  \begin{align}
    g_{jk}(t) &\dfn \left \langle \left. \frac{\del}{\del t_j} \right|_t,  \left. \frac{\del}{\del t_k} \right|_t \right \rangle_{T_t T}, \quad t \in U, \ j, k \in \{1, \ldots, n\} \label{eq:gjk}
  \end{align}
  and $(g^{jk})$ denotes its inverse. It easily follows that for $\psi \in \cinfty(T; \R)$ such that $\dd_T \psi(t_0) = \tau_0$ we have
  \begin{align*}
    Q_2(t_0, \tau_0) = \lim_{\rho \to \infty} \rho^{-2} e^{-i \rho \psi} Q(e^{i \rho \psi})|_{t_0} = \sum_{j,k = 1}^n  g^{jk}(t_0) \frac{\del \psi}{\del t_j}(t_0) \frac{\del \psi}{\del t_k}(t_0)
  \end{align*}
  which vanishes only if $\tau_0 = 0$. It follows that $\tilde{P}$ is automatically elliptic when $Q = \Delta_T$ -- no assumptions needed on $\vv{W}_1, \ldots, \vv{W}_N$.

\begin{Lem} \label{lem:micr_from_fio} Suppose that $\tilde{P}$ is elliptic and that $u \in \D'(T \times G)$ is such that $Pu \in \cinfty(T \times G)$. Then for every $\phi \in \cinfty(G)$ we have that $\tilde{u}(\phi) \dfn \langle u, \cdot \otimes \phi \rangle \in \cinfty(T)$.
  \begin{proof} First, we will show that
    \begin{align}
      \{ (t, \tau) \in T^* T \setminus 0 \st (t, \tau, x, 0) \in \mathrm{Char}(P) \ \text{for some $x \in G$} \} &= \eset \label{eq:charPell}
    \end{align}
    which is a direct consequence of the ellipticity of $\tilde{P}$. Indeed, the principal symbol of $P$ at $(t,\tau, x, 0) \in T_t^* T \times T_x^* G \cong T_{(t,x)}^*(T \times G)$ is given by
    \begin{align*}
      P_2(t, \tau, x, 0) &= \lim_{\rho \to \infty} \rho^{-2} e^{-i \rho f} P (e^{i \rho f})
    \end{align*}
    where $f \in \cinfty(T \times G; \R)$ is any function such that $\dd f(t,x) = (\tau, 0)$. This can certainly be achieved by taking $f \dfn \psi \otimes 1_G$ where $\psi \in \cinfty(T; \R)$ is such that $\dd_T \psi(t) = \tau$, in which case one easily has
    \begin{align*}
      P (e^{i \rho f}) &= \tilde{P} (e^{i \rho \psi}) \otimes 1_G, \quad \forall \rho > 0.
    \end{align*}
    This ultimately implies that
    \begin{align*}
      P_2(t, \tau, x, 0) &= \tilde{P}_2(t, \tau), \quad \forall (t, \tau) \in T^* T, \ x \in G,
    \end{align*}
    hence~\eqref{eq:charPell} follows since $\tilde{P}$ is elliptic.
    
    Now let $\phi \in \cinfty(G)$: at first, we only know that $\tilde{u}(\phi) \in \D'(T)$. Let $\mathscr{U} \dfn \{ U_\alpha \}_{\alpha \in A}$ be a finite covering of $G$ by coordinate open sets, so by means of a partition of unity subordinate to $\mathscr{U}$ we may write
    \begin{align*}
      \phi = \sum_\alpha \phi_\alpha &\Longrightarrow \tilde{u} (\phi) = \sum_\alpha \tilde{u}(\phi_\alpha), \quad \phi_\alpha \in \cinfty_c(U_\alpha),
    \end{align*}
    hence it is enough to prove that $\tilde{u}(\phi_\alpha) \in \cinfty(T)$ for each $\alpha \in A$. In order to do so, let $V \sset T$ be a coordinate open set and define $v \dfn u|_{V \times U_\alpha}$. For $\psi \in \cinfty_c(V)$ we have
    \begin{align*}
      \langle \tilde{v}(\phi_\alpha), \psi \rangle = \langle v, \psi \otimes \phi_\alpha \rangle = \langle u, \psi \otimes \phi_\alpha \rangle = \langle \tilde{u}(\phi_\alpha), \psi \rangle
    \end{align*}
    i.e.~$\tilde{v}(\phi_\alpha) = \tilde{u}(\phi_\alpha)|_V$, and thus for $(t, \tau) \in T^*V \sset T^* T$ we have
    \begin{align*}
      (t, \tau) \in \mathrm{WF}(\tilde{v}(\phi_\alpha)) &\Longleftrightarrow (t, \tau) \in \mathrm{WF}(\tilde{u}(\phi_\alpha)).
    \end{align*}
    On the other hand, by~\cite[Theorem~2.5.12]{hormander71} -- which we can now apply since both $U_\alpha$ and $V$ are Euclidean open sets -- we have that
    \begin{align*}
      (t, \tau) \in \mathrm{WF}(\tilde{v}(\phi_\alpha)) &\Longrightarrow (t, \tau, x, 0) \in \mathrm{WF}(v) \ \text{for some $x \in U_\alpha$} \\
      &\Longrightarrow (t, \tau, x, 0) \in \mathrm{WF}(u)
    \end{align*}
    which is further contained in $\mathrm{Char}(P)$ since $Pu$ is everywhere smooth. But by~\eqref{eq:charPell} we must have $\mathrm{WF}(\tilde{v}(\phi_\alpha)) = \eset$ and hence $\tilde{u}(\phi_\alpha)$ is smooth on $V$.      
  \end{proof}
\end{Lem}

\begin{Cor} \label{cor:ulambdasmooth} Suppose that $\tilde{P}$ is elliptic and that $u \in \D'(T \times G)$ is such that $Pu \in \cinfty(T \times G)$. Then $\mathcal{F}^G_\lambda (u) \in \cinfty(T; E_\lambda^G)$ for every $\lambda \in \sigma(\Delta_G)$.
  \begin{proof} We write, by selecting an orthonormal basis $\phi^\lambda_1, \ldots, \phi^\lambda_{d_\lambda^G}$ for $E_\lambda^G$,
    \begin{align*}
      \mathcal{F}^G_\lambda (u) &= \sum_{i = 1}^{d_\lambda^G} \mathcal{F}^G_\lambda (u)_i \otimes \phi^\lambda_i
    \end{align*}
    where
    \begin{align*}
      \mathcal{F}^G_\lambda (u)_i = \left \langle u, \cdot \otimes \overline{\phi^\lambda_i} \right \rangle = \tilde{u} \left( \overline{\phi^\lambda_i} \right) , \quad i \in \{1, \ldots, d_\lambda^G \},
    \end{align*}
    which are smooth by Lemma~\ref{lem:micr_from_fio}.
  \end{proof}
\end{Cor}
For the next lemma, recall that for $M$ a compact manifold as in Section~\ref{sec:preliminaries} the topology of $\cinfty(M)$ can be given by the following system of (semi)norms, defined, for $f \in \cinfty(M)$, as
\begin{align*}
  \| f \|_{\mathscr{H}^s(M)} &\dfn \left \| (I + \Delta)^s f \right \|_{L^2(M)}, \quad s \in \Z_+.
\end{align*}
We use this fact below with $M = T, G$ and $\Delta = \Delta_T, \Delta_G$, respectively.
\begin{Lem} \label{lem:global_ineq_microl_hormander} Suppose that $u \in \D'(T \times G)$ is such that $\tilde{u}(\phi) = \langle u, \cdot \otimes \phi \rangle \in \cinfty(T)$ for every $\phi \in \cinfty(G)$. Then for each $s > 0$ there exist $C > 0$ and $\theta \in (0, 1)$ such that
  \begin{align*}
    \| \mathcal{F}^T_\mu \mathcal{F}^G_\lambda (u) \|_{L^2(T \times G)} &\leq C (1 + \mu + \lambda)^{-s}, \quad \forall (\mu, \lambda) \in \Lambda_\theta,
  \end{align*}
  where $\Lambda_\theta$ is defined in~\eqref{eq:Atheta}.
  \begin{proof} The hypothesis means that the range of the continuous linear map $\tilde{u}: \cinfty(G) \rarr \D'(T)$ actually lies in $\cinfty(T)$. This yields a new linear map $\tilde{u}: \cinfty(G) \rarr \cinfty(T)$ which is continuous by the Closed Graph Theorem: it follows that for each $s \in \Z_+$ there exist $C > 0$ and $s' \in \Z_+$ such that
    \begin{align*}
      \| \tilde{u}(\phi) \|_{\mathscr{H}^s(T)} &\leq C \| \phi \|_{\mathscr{H}^{s'}(G)}, \quad \forall \phi \in \cinfty(G).
    \end{align*}
    Taking $\phi = \overline{\phi^\lambda_j}$ -- one of our orthonormal basis elements of $E^G_\lambda$ -- we obtain
    \begin{align*}
   \| \tilde{u}(\overline{\phi^\lambda_j}) \|_{\mathscr{H}^s(T)} &\leq C   \| \overline{\phi^\lambda_j} \|_{\mathscr{H}^{s'}(G)} = C(1 + \lambda)^{s'},
    \end{align*}
    while on the other hand
    \begin{align*}
      \| \tilde{u}(\overline{\phi^\lambda_j}) \|_{\mathscr{H}^s(T)}^2 &= \sum_{\mu \in \sigma(\Delta_T)} (1 + \mu)^{2s} \left \| \mathcal{F}^T_\mu [\tilde{u}(\overline{\phi^\lambda_j})] \right \|_{L^2(T)}^2 \\
      &= \sum_{\mu \in \sigma(\Delta_T)} (1 + \mu)^{2s} \sum_{i = 1}^{d^T_\mu} | \langle \tilde{u}(\overline{\phi^\lambda_j}), \overline{\psi^\mu_i} \rangle |^2 \\
      &= \sum_{\mu \in \sigma(\Delta_T)} (1 + \mu)^{2s} \sum_{i = 1}^{d^T_\mu} | \langle u, \overline{\psi^\mu_i \otimes \phi^\lambda_j} \rangle |^2
    \end{align*}
    hence
    \begin{align*}
      \sum_{j = 1}^{d^G_\lambda} \| \tilde{u}(\overline{\phi^\lambda_j}) \|_{\mathscr{H}^s(T)}^2 = \sum_{\mu \in \sigma(\Delta_T)} (1 + \mu)^{2s} \sum_{i = 1}^{d^T_\mu} \sum_{j = 1}^{d^G_\lambda} | \langle u, \overline{\psi^\mu_i \otimes \phi^\lambda_j} \rangle |^2 = \sum_{\mu \in \sigma(\Delta_T)} (1 + \mu)^{2s} \| \mathcal{F}^T_\mu \mathcal{F}^G_\lambda (u) \|_{L^2(T \times G)}^2
    \end{align*}
    from which we conclude that
    \begin{align*}
      (1 + \mu)^{2s} \| \mathcal{F}^T_\mu \mathcal{F}^G_\lambda (u) \|_{L^2(T \times G)}^2 \leq \sum_{j = 1}^{d^G_\lambda} \| \tilde{u}(\overline{\phi^\lambda_j}) \|_{\mathscr{H}^s(T)}^2 \leq d^G_\lambda C^2 (1 + \lambda)^{2s'}
    \end{align*}
    and thus
    \begin{align*}
      (1 + \mu)^{s} \| \mathcal{F}^T_\mu \mathcal{F}^G_\lambda (u) \|_{L^2(T \times G)} &\leq \sqrt{d^G_\lambda} C (1 + \lambda)^{s'}, \quad \forall (\mu, \lambda) \in \sigma(\Delta_T) \times \sigma(\Delta_G)
    \end{align*}
    and since, thanks to~\eqref{eq:weyl}, we have $d^G_\lambda = \mathrm{O}(\lambda^{2m})$ it follows, enlarging $C$ whenever necessary, that
    \begin{align*}
      (1 + \mu)^{s} \| \mathcal{F}^T_\mu \mathcal{F}^G_\lambda (u) \|_{L^2(T \times G)} &\leq C (1 + \lambda)^{s' + m}, \quad \forall (\mu, \lambda) \in \sigma(\Delta_T) \times \sigma(\Delta_G).
    \end{align*}

    Let $\theta \in (0, 1)$ be so small that $\theta(s' + m) \leq s/2$: for $(\mu, \lambda) \in \Lambda_\theta$ we then have
    \begin{align*}
      \| \mathcal{F}^T_\mu \mathcal{F}^G_\lambda (u) \|_{L^2(T \times G)} \leq C (1 + \lambda)^{s' + m} (1 + \mu)^{-s} \leq C (1 + \mu)^{\theta (s' + m) -s} \leq C (1 + \mu)^{-s/2}.
    \end{align*}
    Moreover, on $\Lambda_\theta$ we have
    \begin{align*}
      1 + \mu + \lambda \leq \mu + (1 + \mu)^\theta \leq (1 + \mu)^{\theta + 1} \leq (1 + \mu)^2
    \end{align*}
    from which we finally conclude
    \begin{align*}
      \| \mathcal{F}^T_\mu \mathcal{F}^G_\lambda (u) \|_{L^2(T \times G)} &\leq C (1 + \mu + \lambda)^{-s/4}, \quad \forall (\mu, \lambda) \in \Lambda_\theta,
    \end{align*}
    hence leading to our conclusion.
  \end{proof}
\end{Lem}

Combining Lemmas~\ref{lem:micr_from_fio} and~\ref{lem:global_ineq_microl_hormander} we conclude:
\begin{Cor} \label{cor:est_small_cone} Suppose that $\tilde{P}$ is elliptic. If $u \in \D'(T \times G)$ is such that $Pu \in \cinfty(T \times G)$ then for every $s > 0$ there exist $C > 0$ and $\theta \in (0, 1)$ such that
  \begin{align*}
    \| \mathcal{F}^T_\mu \mathcal{F}^G_\lambda (u) \|_{L^2(T \times G)} &\leq C (1 + \mu + \lambda)^{-s}, \quad \forall (\mu, \lambda) \in \Lambda_\theta.
  \end{align*}
\end{Cor}

\section{Interlude: global hypoellipticity of certain systems of vector fields}

In this section we derive some general results regarding global hypoellipticity of systems of vector fields (Definition~\ref{def:gh_systems}) which are needed to pave the way for the proofs of Theorem~\ref{thm:thm15} and related results later on. We consider  $M$ a compact Riemannian manifold enjoying all the properties described in Section~\ref{sec:preliminaries}, from where we also borrow the notation. We denote its Laplace-Beltrami operator simply by $\Delta$, and $\mathcal{L}$ will stand for any system of smooth vector fields  in $M$.
\begin{Lem} \label{lem:Lspanlie} The following are equivalent:
  \begin{enumerate}
  \item \label{it:gh1} $\mathcal{L}$ is~$\mathrm{(GH)}$ in $M$.
  \item \label{it:gh2} $\Span_\R \mathcal{L}$ is~$\mathrm{(GH)}$ in $M$.
  \item \label{it:gh3} $\lie \mathcal{L}$, the Lie algebra generated by $\mathcal{L}$, is~$\mathrm{(GH)}$ in $M$.
  \end{enumerate}
  \begin{proof} It is clear that if $\mathcal{L} \sset \mathcal{L}'$ are two families of vector fields and $\mathcal{L}$ is~$\mathrm{(GH)}$ in $M$ then so is $\mathcal{L}'$. This observation takes care of the implications~$\eqref{it:gh1} \Rightarrow \eqref{it:gh2} \Rightarrow \eqref{it:gh3}$ since $\mathcal{L} \sset \Span_\R \mathcal{L} \sset \lie \mathcal{L}$. Since moreover
    \begin{align*}
      \lie \mathcal{L} &= \Span_\R \bigcup_{\nu \in \N} \{[\vv{X}_1,[ \cdots [\vv{X}_{\nu - 1}, \vv{X}_\nu] \cdots ]] \st \vv{X}_j \in \mathcal{L}, \ 1 \leq j \leq \nu \}
    \end{align*}
    it is also clear that, given $u \in \D'(M)$, if $\vv{L} u \in \cinfty(M)$ for every $\vv{L} \in \mathcal{L}$ then also $\tilde{\vv{L}} u \in \cinfty(M)$ for every $\tilde{\vv{L}} \in \lie \mathcal{L}$. It follows immediately that $\eqref{it:gh3} \Rightarrow \eqref{it:gh1}$.
  \end{proof}
\end{Lem}

The main advantage of the previous lemma is that it enables us to freely transition between different sets of generators of a given system. The next proposition characterizes global hypoellipticity of certain finitely generated systems in terms of manageable inequalities.
\begin{Prop} \label{prop:ghfinitevfs} Suppose that $\vv{L}_1, \ldots, \vv{L}_r$ are smooth vector fields on $M$ which commute with $\Delta$. Then the system $\{\vv{L}_1, \ldots, \vv{L}_r\}$ is~$\mathrm{(GH)}$ in $M$ if and only if there exist $C, \rho > 0$ and $\lambda_0 \in \sigma(\Delta)$ such that
  \begin{align}
    \left( \sum_{j = 1}^r \| \vv{L}_j \phi \|_{L^2(M)}^2 \right)^{\frac{1}{2}} &\geq C(1 + \lambda)^{-\rho} \| \phi \|_{L^2(M)}, \quad \forall \phi \in E_\lambda, \ \forall \lambda \geq \lambda_0. \label{eq:ineq_gh}
  \end{align}
  \begin{proof} Suppose that $C, \rho > 0$ and $\lambda_0 \in \sigma(\Delta)$ are such that~\eqref{eq:ineq_gh} holds, and let $u \in \D'(M)$ be such that $\vv{L}_1 u, \ldots, \vv{L}_r u \in \cinfty(M)$. Given $s > 0$, for each $j \in \{1, \ldots, r\}$ there exists $C_j > 0$ such that
    \begin{align*}
      \| \mathcal{F}_\lambda(\vv{L}_j u) \|_{L^2(M)} &\leq C_j (1 + \lambda)^{-s - \rho}, \quad \forall \lambda \in \sigma(\Delta),
    \end{align*}
    by Proposition~\ref{prop:charac_smoothness_fourier_proj}. Since $\mathcal{F}_\lambda(\vv{L}_j u) = \vv{L}_j \mathcal{F}_\lambda(u)$ (for $\vv{L}_j$ commutes with $\Delta$: use Proposition~\ref{prop:invopscommft} with $T \dfn \{ \mathrm{pt} \}$, or see~\cite[Proposition~2.2]{araujo19}) we have for every $\lambda \geq \lambda_0$ that
    \begin{align*}
      \| \mathcal{F}_\lambda(u) \|_{L^2(M)} \leq C^{-1} (1 + \lambda)^{\rho} \left( \sum_{j = 1}^r \| \vv{L}_j \mathcal{F}_\lambda(u) \|_{L^2(M)}^2 \right)^{\frac{1}{2}} \leq C^{-1} \left( \sum_{j = 1}^r C_j^2 \right)^{\frac{1}{2}} (1 + \lambda)^{-s}. 
    \end{align*}
    Since the set $\{ \lambda \in \sigma(\Delta) \st \lambda < \lambda_0 \}$ is finite we easily conclude by Proposition~\ref{prop:charac_smoothness_fourier_proj} that $u \in \cinfty(M)$.

    As for the converse, suppose that for every $\nu \in \N$ there exist $\lambda_\nu \in \sigma(\Delta)$ with $\lambda_\nu \geq \nu$ and $\phi_\nu \in E_{\lambda_\nu}$ such that
    \begin{align*}
    \left( \sum_{j = 1}^r \| \vv{L}_j \phi_\nu \|_{L^2(M)}^2 \right)^{\frac{1}{2}} &< 2^{-\nu} (1 + \lambda_\nu)^{-\nu} \| \phi_\nu \|_{L^2(M)}. 
    \end{align*}
    Without loss of generality we may assume that $\| \phi_\nu \|_{L^2(M)} = 1$ and that the sequence $\{ \lambda_\nu \}_{\nu \in \N}$ is strictly increasing. If we then define
    \begin{align*}
      u & \dfn \sum_{\nu \in \N} \phi_\nu
    \end{align*}
    then $u \in \D'(M) \setminus \cinfty(M)$ by Proposition~\ref{prop:charac_smoothness_fourier_proj} since
    \begin{align*}
      \mathcal{F}_\lambda(u) &=
      \begin{cases}
        \phi_\nu, &\text{if $\lambda = \lambda_\nu$}, \\
        0, &\text{otherwise};
      \end{cases}
    \end{align*}
    on the other hand, for each $j \in \{1, \ldots, r\}$ we have, given $s > 0$:
    \begin{itemize}
    \item if $\lambda = \lambda_\nu$ for some $\nu \geq s$:
      \begin{align*}
        \| \mathcal{F}_\lambda (\vv{L}_j u) \|_{L^2(M)} = \| \vv{L}_j \phi_\nu \|_{L^2(M)} \leq 2^{-\nu} (1 + \lambda_\nu)^{-\nu} \leq (1 + \lambda)^{-s};
      \end{align*}
    \item if $\lambda \neq \lambda_\nu$ for every $\nu \in \N$:
      \begin{align*}
        \| \mathcal{F}_\lambda (\vv{L}_j u) \|_{L^2(M)} = 0 \leq (1 + \lambda)^{-s}.
      \end{align*}
    \end{itemize}
    Thus $\vv{L}_j u \in \cinfty(M)$ as, since the set $\{ \nu \in \N \st \nu < s \}$ is finite, there exists a constant $C_j > 0$ such that
    \begin{align*}
      \| \mathcal{F}_\lambda (\vv{L}_j u) \|_{L^2(M)} &\leq C_j(1 + \lambda)^{-s}, \quad \forall \lambda \in \sigma(\Delta).
    \end{align*}
    As this holds for every $j \in \{1, \ldots, r\}$ we conclude that $\{ \vv{L}_1, \ldots, \vv{L}_r\}$ is not~$\mathrm{(GH)}$ in $M$. 
  \end{proof}
\end{Prop}

\section{Sufficiency for operators subject to commutativity assumptions}\label{sec:proofsufficience}

Our aim in this section is to prove Theorem~\ref{thm:thm15}, which still requires some preparation. For each $\ell \in \{1, \ldots, N\}$ we write
\begin{align*}
  \mathfrak{a}_\ell(t) = \sum_{j = 1}^m a_{\ell j}(t) \vv{X}_j, \quad t \in T,
\end{align*}
which we assume to be not identically zero, hence among $a_{\ell 1}, \ldots, a_{\ell m}$ there are exactly $m^\ell\geq 1$ functions that are $\R$-linearly independent. We denote them by $\alpha_{\ell 1}, \ldots, \alpha_{\ell m^\ell}$: writing the remaining coefficients as linear combinations of these allows us to write $\gr{a}_\ell$ as
\begin{align*}
  \mathfrak{a}_\ell(t) = \sum_{p = 1}^{m^\ell} \alpha_{\ell p}(t) \vv{L}_p^\ell,
\end{align*}
where $\vv{L}_1^\ell, \ldots, \vv{L}_{m^\ell}^\ell$ are linear combinations of $\vv{X}_1,\ldots, \vv{X}_m$, hence also elements of $\gr{g}$. One can prove that $\vv{L}_1^\ell, \ldots, \vv{L}_{m^\ell}^\ell$ are linearly independent, and actually a basis for $\mathcal{L}_\ell$ as defined in~\eqref{eq:Lell} (see Section~\ref{sec:nsa_vectors} where we derive explicit expressions for these vector fields w.r.t.~the choice $\alpha_{\ell p} \dfn a_{\ell j_{p}^\ell}$ for $p \in \{1, \ldots, m^\ell \}$).

Linear independence of $\alpha_{\ell 1}, \ldots, \alpha_{\ell m^\ell}$ means that if we define $D_\ell: T \times \R^{m^\ell} \rarr \R$ by
\begin{align*}
  D_\ell (t, \gamma) &\dfn \left( \sum_{p = 1}^{m^\ell} \alpha_{\ell p}(t) \gamma_p \right)^2, \quad t \in T, \ \gamma \in \R^{m^\ell},
\end{align*}
then for each $\gamma \neq 0$ the function $t \in T \mapsto D_\ell (t, \gamma) \in \R$ cannot be identically zero. We then have, as in the proof of~\cite[Lemma~3.1]{bfp17}:
\begin{Lem} \label{lem:lem_functions} There are constants $\alpha, \delta > 0$ such that for every $\gamma$ with $|\gamma| = 1$ there exists a non-empty open set $A_\gamma \sset T$ with $\mathrm{vol}(A_\gamma) \geq \delta$ such that
  \begin{align*}
    D_\ell(t, \gamma) &> \alpha, \quad \forall t \in A_\gamma.
  \end{align*}
\end{Lem}
Of course the inequality above can be extended by positive homogeneity as
\begin{align*}
    D_\ell(t, \gamma) &\geq \alpha |\gamma|^2, \quad \forall t \in A_\gamma, 
\end{align*}
for every $\gamma \in \R^{m^\ell}$, provided $A_\gamma \sset T$ is defined accordingly.

Next we derive the following fundamental inequality, which generalizes~\cite[eqn.~(2.10)]{hp00}.
\begin{Prop} \label{prop:calc_ineq} Given $\delta > 0$ there exists a constant $C > 0$ such that for every open set $A \sset T$ with $\mathrm{vol}(A) \geq \delta$ one has
  \begin{align*}
    \| \psi\|_{L^2(T)}^2 &\leq C \left( \| \psi\|_{L^2(A)}^2 + \| \dd_T \psi\|_{L^2(T)}^2 \right), \quad \forall \psi \in \cinfty(T).
  \end{align*}
\end{Prop}
We start with a local result.
\begin{Lem} \label{lem:calc_ineq_local} For each $t_0 \in T$ there exist $U \sset T$ an open neighborhood of $t_0$ and a constant $C > 0$ such that
  \begin{align*}
    \mathrm{vol}(B) \| \psi\|_{L^2(U)}^2 &\leq C \left( \| \psi\|_{L^2(B)}^2 + \| \dd_T \psi\|_{L^2(U)}^2 \right)
  \end{align*}
  for every open set $B \sset U$ and every $\psi \in \cinfty(U)$.
  \begin{proof} Let $(U; t_1, \ldots, t_n)$ be an oriented coordinate chart of $T$ centered at $t_0$. On $U$ the Riemannian volume form can be written as $\dd V_T = \sqrt{\det g} \ \dd t$, where $\dd t =  \dd t_1 \wedge \cdots \wedge \dd t_n$ and $g = (g_{jk}) \in \cinfty(U, \GL_n(\R))$ is the local expression of the metric~\eqref{eq:gjk}. We may select real smooth vector fields $\vv{Z}_1, \ldots, \vv{Z}_n$ on $U$ forming an orthonormal frame for $TU$, and denote by $\zeta_1, \ldots, \zeta_n \in \cinfty(U; T^*U)$ the corresponding dual coframe, which is of course orthonormal w.r.t.~the cotangent metric. Thus for any $\psi \in \cinfty(U)$ we have
    \begin{align*}
      \dd_T \psi &= \sum_{j = 1}^n \vv{Z}_j \psi \ \zeta_j
    \end{align*}
    hence
    \begin{align*}
      \| \dd_T \psi \|_{L^2(U)}^2 = \int_U \langle \dd_T \psi, \dd_T \psi \rangle_{T^* T} \ \dd V_T = \sum_{j,j' = 1}^n \int_U  \vv{Z}_j \psi \ \overline{\vv{Z}_{j'} \psi} \ \langle \zeta_j, \zeta_{j'} \rangle_{T^* T} \ \dd V_T = \sum_{j = 1}^n \int_U |\vv{Z}_j \psi|^2 \ \dd V_T.
    \end{align*}
    Moreover, there exists a smooth invertible matrix $(\beta_{jk}) \in \cinfty(U, \GL_n(\R))$ relating both frames:
    \begin{align*}
      \frac{\del}{\del t_j} &= \sum_{k = 1}^n \beta_{jk} \vv{Z}_k, \quad j \in \{1, \ldots, n\}.
    \end{align*}
    After shrinking $U$ if necessary we may assume that there exists a constant $c > 0$ such that
    \begin{align}
      c^{-1} \leq \sqrt{\det g} \leq c, \quad |\beta_{jk}| \leq c \quad \text{on $U$} \label{eq:bddness_everything}
    \end{align}
    and actually that $U \cong U_\epsilon \dfn (-\epsilon, \epsilon)^n$ for some $\epsilon > 0$.
    
    Now we perform some computations in coordinates. Given $a = (a_1, \ldots, a_n), b = (b_1, \ldots, b_n) \in U_\epsilon$ let
    \begin{align*}
      (a, s, b)_j &\dfn (a_1, \ldots, a_{j - 1}, s, b_{j + 1}, \ldots, b_n) \in U_\epsilon, \quad s \in (-\epsilon, \epsilon), \ j \in \{1, \ldots, n \},
    \end{align*}
    so $n$ applications of the Fundamental Theorem of Calculus ensure that for $\psi \in \cinfty(U_\epsilon)$ we have
    \begin{align*}
      |\psi(a)| \leq |\psi(b)| + \sum_{j = 1}^n \left| \int_{a_j}^{b_j} \frac{\del \psi}{\del t_j} ((a, s, b)_j)  \dd s \right| \leq |\psi(b)| + \sum_{j = 1}^n \int_{-\epsilon}^{\epsilon} \left| \frac{\del \psi}{\del t_j} ((a, s, b)_j) \right| \dd s
    \end{align*}
    and then by H{\"o}lder's inequality
    \begin{align*}
      |\psi(a)|^2 &\leq (n + 1) |\psi(b)|^2 + (n + 1) \sum_{j = 1}^n \left( \int_{-\epsilon}^{\epsilon} \left| \frac{\del \psi}{\del t_j} ((a, s, b)_j) \right| \dd s \right)^2 \\
      &\leq (n + 1) |\psi(b)|^2 + 2 \epsilon (n + 1) \sum_{j = 1}^n \int_{-\epsilon}^{\epsilon} \left| \frac{\del \psi}{\del t_j} ((a, s, b)_j) \right|^2 \dd s.
    \end{align*}
    Regarding the Lebesgue measure on $U_\epsilon \sset \R^n$ (i.e.~the one induced by $\dd t$), for $B \sset U \cong U_\epsilon$ an open set we integrate both sides of this inequality w.r.t.~$b \in B$, yielding
    \begin{align*}
      m(B) |\psi(a)|^2 &\leq (n + 1) \int_B |\psi(b)|^2 \dd b + 2 \epsilon (n + 1) \sum_{j = 1}^n \int_B \int_{-\epsilon}^{\epsilon} \left| \frac{\del \psi}{\del t_j} ((a, s, b)_j) \right|^2 \dd s \dd b
    \end{align*}
    -- where we provisionally denote $m(B) \dfn \int_B \dd t$ -- and now integrating the latter w.r.t.~$a \in U_\epsilon$ we have
    \begin{align*}
      m(B) \int_{U_\epsilon} |\psi(a)|^2 \dd a &\leq (2 \epsilon)^n (n + 1) \int_B |\psi(b)|^2 \dd b + 2 \epsilon (n + 1) \sum_{j = 1}^n \int_{U_\epsilon} \int_B \int_{-\epsilon}^{\epsilon} \left| \frac{\del \psi}{\del t_j} ((a, s, b)_j) \right|^2 \dd s \dd b \dd a.
    \end{align*}
    Notice, however, that
    \begin{align*}
      \int_{U_\epsilon} \int_B \int_{-\epsilon}^{\epsilon} \left| \frac{\del \psi}{\del t_j} ((a, s, b)_j) \right|^2 \dd s \dd b \dd a &\leq \int_{U_\epsilon} \int_{U_\epsilon} \int_{-\epsilon}^{\epsilon} \left| \frac{\del \psi}{\del t_j} (a_1, \ldots, a_{j - 1}, s, b_{j + 1}, \ldots, b_n) \right|^2 \dd s \dd b \dd a \\
      &= (2 \epsilon)^{n + 1}  \int_{U_\epsilon} \left| \frac{\del \psi}{\del t_j} (t) \right|^2 \dd t.
    \end{align*}
    We conclude that for some constant $C_1 > 0$ depending only on $n$ and $\epsilon$ we have
    \begin{align}
      m(B) \int_{U} |\psi|^2 \dd t &\leq C_1 \left( \int_B |\psi|^2 \dd t + \sum_{j = 1}^n \int_{U} \left| \frac{\del \psi}{\del t_j} \right|^2 \dd t \right). \label{eq:calc_ineq_euclidean}
    \end{align}
    
    Now it is plain from~\eqref{eq:bddness_everything} and previous remarks that
    \begin{align*}
      m(B) = \int_B \dd t \geq c^{-1} \int_B \sqrt{\det g} \ \dd t = c^{-1} \int_B \dd V_T = c^{-1} \mathrm{vol}(B)
    \end{align*}
    and by the same token
    \begin{align*}
      c^{-1} \int |\psi|^2 \dd V_T \leq \int |\psi|^2 \dd t \leq c \int |\psi|^2 \dd V_T.
    \end{align*}
    Moreover, for $j \in \{1, \ldots, n\}$,
    \begin{align*}
      \int_{U} \left| \frac{\del \psi}{\del t_j} \right|^2 \dd t = \int_{U} \left| \sum_{k = 1}^n \beta_{jk} \vv{Z}_k \psi \right|^2 \dd t \leq  n c^3 \sum_{k = 1}^n \int_{U} |\vv{Z}_k \psi|^2 \dd V_T = n c^3 \| \dd_T \psi\|_{L^2(U)}^2.
    \end{align*}
    Plugging everything back into~\eqref{eq:calc_ineq_euclidean} yields our conclusion at once.
  \end{proof}
\end{Lem}
\begin{proof}[Proof of Proposition~\ref{prop:calc_ineq}] Now we globalize Lemma~\ref{lem:calc_ineq_local}. Since $T$ is compact we may select a finite collection of open sets $\{U_i\}_{i \in I}$ covering $T$, each one of them satisfying the conclusion of Lemma~\ref{lem:calc_ineq_local}, namely: for each $i \in I$ there exists a constant $C_i > 0$ such  that
  \begin{align}
    \mathrm{vol}(B) \| \psi\|_{L^2(U_i)}^2 &\leq C_i \left( \| \psi\|_{L^2(B)}^2 + \| \dd_T \psi\|_{L^2(U_i)}^2 \right) \label{eq:ineq_local}
  \end{align}
  for every open set $B \sset U_i$ and every $\psi \in \cinfty(U_i)$. Let $\delta > 0$ and let $A \sset T$ be an open set such that $\mathrm{vol}(A) \geq \delta$. Denoting by $q$ the number of elements in $I$, since
  \begin{align*}
    A = \bigcup_{i \in I} A \cap U_i &\Longrightarrow \mathrm{vol}(A) \leq \sum_{i \in I} \mathrm{vol}(A \cap U_i)
  \end{align*}
  there must exist $i_0 \in I$ such that $\mathrm{vol}(A \cap U_{i_0}) \geq \delta/q$. Let $I_0 \dfn \{ i_0\}$ and define inductively
  \begin{align*}
    I_{\nu + 1} &\dfn \left \{ i \in I \setminus \bigcup_{\kappa = 1}^\nu I_\kappa \st \text{$U_i \cap U_{i'} \neq \eset$ for some $i' \in I_\nu$} \right \}, \quad \nu \in \Z_+,
  \end{align*}
  thus forming a (obviously finite) partition of $I$ (recall that $T$ is assumed connected).

  We claim that for each $\nu \in \Z_+$ there exists $\tilde{C}_\nu > 0$ depending on $\delta > 0$, but not on $A$, such that
  \begin{align*}
    \| \psi\|_{L^2(U_i)}^2 &\leq \tilde{C}_\nu \left( \| \psi\|_{L^2(A)}^2 + \| \dd_T \psi\|_{L^2(T)}^2 \right), \quad \forall \psi \in \cinfty(T), \ \forall i \in I_\nu,
  \end{align*}
  which we prove by induction on $\nu$. The case $\nu = 0$ follows since $I_0 = \{i_0\}$ and by~\eqref{eq:ineq_local} we have
  \begin{align*}
    \| \psi\|_{L^2(U_{i_0})}^2 \leq \frac{C_{i_0}}{\mathrm{vol}(A \cap U_{i_0})}\left( \| \psi\|_{L^2(A \cap U_{i_0})}^2 + \| \dd_T \psi\|_{L^2(U_i)}^2 \right) \leq \frac{q C_{i_0}}{\delta}\left( \| \psi\|_{L^2(A)}^2 + \| \dd_T \psi\|_{L^2(T)}^2 \right).
  \end{align*}
  Now, assuming the claim proved for $\nu$, let $i \in I_{\nu + 1}$ and $i' \in I_\nu$ such that $U_i \cap U_{i'} \neq \eset$: we have
  \begin{align*}
    \| \psi\|_{L^2(U_i)}^2 &\leq \frac{C_i}{\mathrm{vol}(U_i \cap U_{i'})}\left( \| \psi\|_{L^2(U_i \cap U_{i'})}^2 + \| \dd_T \psi\|_{L^2(U_i)}^2 \right) \\
    &\leq \frac{C_i}{\mathrm{vol}(U_i \cap U_{i'})}\left( \| \psi\|_{L^2(U_{i'})}^2 + \| \dd_T \psi\|_{L^2(U_i)}^2 \right) \\
    &\leq \frac{C_i}{\mathrm{vol}(U_i \cap U_{i'})}\left( \tilde{C}_\nu \left( \| \psi\|_{L^2(A)}^2 + \| \dd_T \psi\|_{L^2(T)}^2 \right) + \| \dd_T \psi\|_{L^2(U_i)}^2 \right) \\
    &\leq \frac{C_i (\tilde{C}_\nu + 2)}{\epsilon} \left( \| \psi\|_{L^2(A)}^2 + \| \dd_T \psi\|_{L^2(T)}^2 \right)
  \end{align*}
  again by~\eqref{eq:ineq_local}, where $\epsilon \dfn \min \{ \mathrm{vol}(U_i \cap U_{i'}) \st i, i' \in I, \ U_i \cap U_{i'} \neq \eset \} > 0$ only depends on the finite covering $\{ U_i \}_{i \in I}$. Since there are only finitely many such $i$, our claim is proved.

  To finish, for $\psi \in \cinfty(T)$ we have
  \begin{align*}
    \| \psi\|_{L^2(T)}^2 \leq \sum_{i \in I} \| \psi\|_{L^2(U_i)}^2 = \sum_\nu \sum_{i \in I_\nu} \| \psi\|_{L^2(U_i)}^2 \leq q \max_\nu \tilde{C}_\nu \left( \| \psi\|_{L^2(A)}^2 + \| \dd_T \psi\|_{L^2(T)}^2 \right)
  \end{align*}
  where the constant depends on $\delta$ but not on $A$ or $\psi$; one could object that it depends on the partition $\{I_\nu\}_\nu$ of $I$ -- which apparently depends on $A$, but actually depends (by construction) on $i_0$ only: since $I$ is finite, one could then further maximize these constants over all the possible choices of initial index $i_0 \in I$, hence finally getting rid of the dependence on $A$.
\end{proof}

On passing, we point out that the argument in the proof of Lemma~\ref{lem:calc_ineq_local} also yields the following result, which we will need later on:
\begin{Lem} \label{lem:l2estimvfd} Let $\vv{W}$ be any vector field globally defined on $T$. Then there exists $C > 0$ such that
  \begin{align*}
    \| \vv{W} \psi \|_{L^2(T)} &\leq C \| \dd_T \psi \|_{L^2(T)}, \quad \forall \psi \in \cinfty(T).
  \end{align*}
  \begin{proof} We briefly recall the main argument. In a coordinate open set $U \sset T$ we may write
    \begin{align*}
      \vv{W} &= \sum_{k = 1}^n \omega_k \vv{Z}_k
    \end{align*}
    where once more $\vv{Z}_1, \ldots, \vv{Z}_n$ is an orthonormal frame for $TU$ and $\omega_1, \ldots, \omega_n \in \cinfty(U)$ -- which we may assume to be bounded by shrinking $U$ if necessary. Then for $\psi \in \cinfty(T)$ we have
    \begin{align*}
      \|\vv{W} \psi \|_{L^2(U)}^2 = \int_U \left| \sum_{k = 1}^n \omega_k \vv{Z}_k \psi \right|^2 \dd V_T \leq \left( \sup \sum_{k' = 1}^n |\omega_{k'}|^2 \right) \sum_{k = 1}^n \int_U |\vv{Z}_k \psi|^2 \dd V_T = C' \| \dd_T \psi \|_{L^2(U)}^2.
    \end{align*}
    Since we can cover $T$ by finitely many such $U$ the result follows.
  \end{proof}
\end{Lem}
  
All of this allows us to prove the following:
\begin{Prop} \label{prop:final_inequality} Under the hypotheses of Theorem~\ref{thm:thm15} there exist $C, \rho > 0$ and $\lambda_0 \in \sigma(\Delta_G)$ such that
  \begin{align}
    \langle P \psi, \psi \rangle_{L^2(T\times G)} &\geq C(1 + \lambda)^{-\rho} \| \psi \|^2_{L^2(T\times G)} , \quad \forall \psi \in \cinfty(T; E^G_\lambda), \ \lambda \geq \lambda_0. \label{eq:InequalitypsiPpsi}
  \end{align}
  \begin{proof} By hypothesis~\eqref{it:thm15_hyp1}, the set of left-invariant vector fields $\mathcal{L}_\ell$ acts as a family of commuting, skew-symmetric -- hence normal -- linear endomorphisms of $E_\lambda^G$ for each $\lambda \in \sigma(\Delta_G)$, which then admits an orthonormal basis
    \begin{align*}
      \phi_1^{\lambda, \ell}, \ldots, \phi_{d_\lambda^G}^{\lambda, \ell} \in E_\lambda^G
    \end{align*}
    which are common eigenvectors to all operators in $\mathcal{L}_\ell$; their associated eigenvalues are purely imaginary
    \begin{align*}
      \vv{L}_p^\ell \phi_i^{\lambda, \ell} &= \sqrt{-1} \gamma_{i,p}^{\lambda, \ell}  \phi_i^{\lambda, \ell}, \quad \gamma_{i,p}^{\lambda, \ell} \in \R,
    \end{align*}
    and we  may bound their absolute values thanks to the following easy remark.
    \begin{Lem} \label{lem:bounding_eigenv_livf} For every $\vv{X} \in \gr{g}$ we have
      \begin{align*}
        \| \vv{X} \phi \|_{L^2(G)} &\leq \| \vv{X} \|_{\gr{g}} \lambda^{1/2} \| \phi \|_{L^2(G)}, \quad \forall \phi \in E_\lambda^G, \ \forall \lambda \in \sigma(\Delta_G)
      \end{align*}
      where $\| \cdot \|_{\gr{g}}$ is the norm on $\gr{g}$ induced by the underlying $\ad$-invariant inner product.
      \begin{proof}[Proof of Lemma~\ref{lem:bounding_eigenv_livf}] We may assume w.l.o.g.~$\vv{X} \neq 0$. Let then $\vv{X}_1, \ldots, \vv{X}_m$ be an orthonormal basis for $\gr{g}$ such that $\vv{X}_1 = \vv{X}/ \| \vv{X} \|_{\gr{g}}$. As the sum of their squares equals $- \Delta_G$~\eqref{eq:DeltaG_SS} we have, for $\phi \in E_\lambda^G$,
        \begin{align*}
          \| \vv{X}_1 \phi \|_{L^2(G)}^2 \leq \sum_{j = 1}^m \| \vv{X}_j \phi \|_{L^2(G)}^2 = - \sum_{j = 1}^m \langle \vv{X}_j^2 \phi, \phi \rangle_{L^2(G)} = \langle \Delta_G \phi, \phi \rangle_{L^2(G)} = \lambda \| \phi \|^2_{L^2(G)}
        \end{align*}
        from which our claim follows.
      \end{proof}
    \end{Lem}
    It follows immediately that
    \begin{align*}
      |\gamma_{i, p}^{\lambda, \ell}|^2 &\leq \| \vv{L}_p^\ell \|_{\gr{g}}^2 \lambda.
    \end{align*}

    For each $i, i' \in \{1, \ldots, d_\lambda^G \}$ and $p, p' \in \{1, \ldots, m^\ell\}$:
    \begin{align*}
      \left \langle \vv{L}_p^\ell \phi_i^{\lambda, \ell}, \vv{L}_{p'}^\ell \phi_{i'}^{\lambda, \ell} \right \rangle_{L^2(G)} = \gamma_{i, p}^{\lambda, \ell} \gamma_{i', p'}^{\lambda, \ell} \left \langle  \phi_i^{\lambda, \ell}, \phi_{i'}^{\lambda, \ell} \right \rangle_{L^2(G)} = \delta_{i i'} \gamma_{i, p}^{\lambda, \ell} \gamma_{i', p'}^{\lambda, \ell}
    \end{align*}
    so in particular for each given $t \in T$ we have
    \begin{align*}
      \left \langle \gr{a}_\ell(t)  \phi_i^{\lambda, \ell}, \gr{a}_\ell(t)  \phi_{i'}^{\lambda, \ell} \right \rangle_{L^2(G)} &= \sum_{p, p' = 1}^{m^\ell} \int_G \alpha_{\ell p}(t) \alpha_{\ell p'}(t) (\vv{L}_p^\ell \phi_i^{\lambda, \ell})(x) \overline{(\vv{L}_{p'}^\ell \phi_{i'}^{\lambda, \ell})(x)} \dd V_G(x)  \\
      &= \sum_{p, p' = 1}^{m^\ell} \alpha_{\ell p}(t) \alpha_{\ell p'}(t) \left \langle \vv{L}_p^\ell \phi_i^{\lambda, \ell}, \vv{L}_{p'}^\ell \phi_{i'}^{\lambda, \ell} \right \rangle_{L^2(G)} \\
      &= \sum_{p, p' = 1}^{m^\ell} \alpha_{\ell p}(t) \alpha_{\ell p'}(t) \delta_{i i'} \gamma_{i, p}^{\lambda, \ell} \gamma_{i', p'}^{\lambda, \ell} \\
      &= \delta_{i i'}D_\ell (t, \gamma_{i}^{\lambda, \ell})
    \end{align*}
    where $\gamma_{i}^{\lambda, \ell} \in \R^{m^\ell}$ is defined in the obvious manner.
    
    A general $\psi \in \cinfty(T; E_\lambda^G)$ is written as, given $\ell \in \{1, \ldots, N\}$,
    \begin{align*}
      \psi &= \sum_{i = 1}^{d_\lambda^G} \psi_i^\ell \otimes \phi_i^{\lambda, \ell}
    \end{align*}
    so for each $t \in T$ given we have that
    \begin{align}
      \| \gr{a}_\ell(t) \psi(t) \|_{L^2(G)}^2 &= \int_G \left| \sum_{i = 1}^{d^G_\lambda} \psi_i^\ell(t) (\gr{a}_\ell(t) \phi_i^{\lambda, \ell})(x) \right|^2 \dd V_G(x) \nonumber \\
      &= \sum_{i, i' = 1}^{d^G_\lambda} \int_G \psi_i^\ell(t) \overline{\psi_{i'}^\ell(t)} (\gr{a}_\ell(t) \phi_i^{\lambda, \ell})(x) \overline{(\gr{a}_\ell(t) \phi_{i'}^{\lambda, \ell})(x)} \dd V_G(x) \nonumber \\
      &= \sum_{i, i' = 1}^{d^G_\lambda} \psi_i^\ell(t) \overline{\psi_{i'}^\ell(t)} \left \langle \gr{a}_\ell(t)  \phi_i^{\lambda, \ell}, \gr{a}_\ell(t)  \phi_{i'}^{\lambda, \ell} \right \rangle_{L^2(G)} \nonumber \\
      &= \sum_{i, i' = 1}^{d^G_\lambda} \psi_i^\ell(t) \overline{\psi_{i'}^\ell(t)} \delta_{i i'}D_\ell (t, \gamma_{i}^{\lambda, \ell}) \nonumber \\
      &= \sum_{i = 1}^{d^G_\lambda} |\psi_i^\ell(t)|^2 D_\ell (t, \gamma_{i}^{\lambda, \ell}) \label{eq:Digualdade}
    \end{align}
    but also
    \begin{align*}
      \| \vv{L}_p^\ell \psi(t) \|_{L^2(G)}^2 = \int_G \left| \sum_{i = 1}^{d^G_\lambda} \psi_i^\ell(t) (\vv{L}_p^\ell \phi_i^{\lambda, \ell})(x) \right|^2 \dd V_G(x) = \sum_{i = 1}^{d^G_\lambda} |\psi_i^\ell(t)|^2 |\gamma_{i,_p}^{\lambda, \ell}|^2.
    \end{align*}

    Recall that $\vv{L}_1^\ell, \ldots, \vv{L}_{m^\ell}^\ell$ form a basis for $\mathcal{L}_\ell$~\eqref{eq:Lell} for each $\ell \in \{1, \ldots, N\}$, so in particular the set
    \begin{align}
      \left \{ \vv{L}_p^\ell \st p \in \{1, \ldots, m^\ell \}, \ \ell \in \{1, \ldots, N\} \right \} \label{eq:generators}
    \end{align}
    generates $\Span_\R \mathcal{L}$. By Lemma~\ref{lem:Lspanlie}, our hypothesis~\eqref{it:thm15_hyp2} of global hypoellipticity of $\mathcal{L}$ in $G$ entails the same property for $\Span_\R \mathcal{L}$ and hence for~\eqref{eq:generators}. As these commute with $\Delta_G$, Proposition~\ref{prop:ghfinitevfs} then provides us constants $C, \rho > 0$ and $\lambda_0 \in \sigma(\Delta_G)$ such that
    \begin{align}
      \left( \sum_{\ell = 1}^N \sum_{p = 1}^{m^\ell} \|\vv{L}_p^\ell \phi \|_{L^2(G)}^2 \right)^{\frac{1}{2}} &\geq C(1 + \lambda)^{-\rho} \| \phi \|_{L^2(G)} , \quad \forall \phi \in E^G_\lambda, \ \lambda \geq \lambda_0. \label{eq:diophantine}
    \end{align}

    Fix $\lambda \geq \lambda_0$. We apply~\eqref{eq:diophantine} to $\phi = \psi(t)$, for some $t \in T$ given
    \begin{align*}
      \| \psi(t) \|_{L^2(G)}^2 &\leq C^{-2} (1 + \lambda)^{2 \rho} \sum_{\ell = 1}^N \sum_{p = 1}^{m^\ell} \|\vv{L}_p^\ell \psi(t) \|_{L^2(G)}^2 \\
      &= C^{-2} (1 + \lambda)^{2 \rho} \sum_{\ell = 1}^N \sum_{p = 1}^{m^\ell} \sum_{i = 1}^{d^G_\lambda} |\psi_i^\ell(t)|^2 |\gamma_{i,_p}^{\lambda, \ell}|^2 \\
      &= C^{-2} (1 + \lambda)^{2 \rho} \sum_{\ell = 1}^N \sum_{i = 1}^{d^G_\lambda} |\psi_i^\ell(t)|^2 |\gamma_{i}^{\lambda, \ell}|^2.
    \end{align*}
    and then integrate both sides over $T$, yielding
    \begin{align}
      \| \psi \|_{L^2(T\times G)}^2 &\leq C^{-2} (1 + \lambda)^{2 \rho} \sum_{\ell = 1}^N \sum_{i = 1}^{d^G_\lambda} \int_T |\psi_i^\ell(t)|^2 |\gamma_{i}^{\lambda, \ell}|^2 \dd V_T(t). \label{eq:psinorm}
    \end{align}
    Let us work out the last integral above. By Lemma~\ref{lem:lem_functions} there are constants $\alpha, \delta > 0$ such that for every $\ell \in \{1, \ldots, N\}$ and every $i \in \{1, \ldots, d_\lambda^G\}$ fixed there exists a non-empty open set $A_{i}^{\lambda, \ell} \sset T$ with $\mathrm{vol}(A_{i}^{\lambda, \ell}) \geq \delta$ such that
    \begin{align*}
      D_\ell(t, \gamma_{i}^{\lambda, \ell}) &\geq \alpha |\gamma_{i}^{\lambda, \ell}|^2, \quad \forall t \in A_{i}^{\lambda, \ell}.
    \end{align*}
    Then by Proposition~\ref{prop:calc_ineq} there exists $C_1 > 0$ depending on $\delta$ but not on any other parameters such that
    \begin{align*}
      \int_T |\psi_i^\ell(t)|^2 |\gamma_{i}^{\lambda, \ell}|^2 \dd V_T(t) &\leq C_1 \left( \int_{A_{i}^{\lambda, \ell}} |\psi_i^\ell(t)|^2 |\gamma_{i}^{\lambda, \ell}|^2 \dd V_T(t) + \left \| \dd_T \left(|\gamma_{i}^{\lambda, \ell}| \psi_i^\ell \right) \right \|_{L^2(T)}^2 \right) \\
      &\leq C_1 \left( \alpha^{-1} \int_{A_{i}^{\lambda, \ell}} |\psi_i^\ell(t)|^2  D_\ell(t, \gamma_{i}^{\lambda, \ell}) \dd V_T(t) + |\gamma_{i}^{\lambda, \ell}|^2 \| \dd_T \psi_i^\ell \|_{L^2(T)}^2 \right) \\
      &\leq C_1 \left( \alpha^{-1} \int_T |\psi_i^\ell(t)|^2  D_\ell(t, \gamma_{i}^{\lambda, \ell}) \dd V_T(t) + B_\ell \lambda \| \dd_T \psi_i^\ell \|_{L^2(T)}^2 \right)
    \end{align*}
    where, recalling the conclusion after Lemma~\ref{lem:bounding_eigenv_livf},
    \begin{align*}
      B_\ell &\dfn \sum_{p = 1}^{m^\ell} \| \vv{L}_p^\ell \|_{\gr{g}}^2.
    \end{align*}
   It follows from~\eqref{eq:Digualdade} that
    \begin{align*}
      \sum_{i = 1}^{d^G_\lambda} \int_T |\psi_i^\ell(t)|^2 |\gamma_{i}^{\lambda, \ell}|^2 \dd V_T(t) &\leq C_1 \left( \alpha^{-1} \int_T \sum_{i = 1}^{d^G_\lambda} |\psi_i^\ell(t)|^2  D_\ell(t, \gamma_{i}^{\lambda, \ell}) \dd V_T(t) + B_\ell \lambda \sum_{i = 1}^{d^G_\lambda} \| \dd_T \psi_i^\ell \|_{L^2(T)}^2 \right) \\
      &= C_1 \left( \alpha^{-1} \int_T \| \gr{a}_\ell (t) \psi(t) \|_{L^2(G)}^2  \dd V_T(t) + B_\ell \lambda \left \langle \Delta_T^\sharp \psi, \psi \right \rangle_{L^2(T \times G)} \right) \\
      &= C_1 \left( \alpha^{-1} \| \gr{a}_\ell (t, \vv{X}) \psi \|_{L^2(T \times G)}^2  + B_\ell \lambda \left \langle \Delta_T^\sharp \psi, \psi \right \rangle_{L^2(T \times G)} \right)
    \end{align*}
    since
    \begin{align*}
      \sum_{i = 1}^{d^G_\lambda} \| \dd_T \psi_i^\ell \|_{L^2(T)}^2 = \sum_{i = 1}^{d^G_\lambda} \langle \dd_T \psi_i^\ell , \dd_T \psi_i^\ell \rangle_{L^2(T)} = \sum_{i = 1}^{d^G_\lambda} \langle \Delta_T \psi_i^\ell , \psi_i^\ell \rangle_{L^2(T)} = \left \langle \Delta_T^\sharp \psi, \psi \right \rangle_{L^2(T \times G)}
    \end{align*}
    and hence
    \begin{align*}
      \sum_{\ell = 1}^N \sum_{i = 1}^{d^G_\lambda} \int_T |\psi_i^\ell(t)|^2 |\gamma_{i}^{\lambda, \ell}|^2 \dd V_T(t) &\leq C_1 \left( \alpha^{-1} \sum_{\ell = 1}^N \| \gr{a}_\ell (t, \vv{X}) \psi \|_{L^2(T \times G)}^2  + \sum_{\ell = 1}^N B_\ell \lambda \left \langle \Delta_T^\sharp \psi, \psi \right \rangle_{L^2(T \times G)} \right) \\
      &\leq C_2 (1 + \lambda) \left( \sum_{\ell = 1}^N \| \gr{a}_\ell (t, \vv{X}) \psi \|_{L^2(T \times G)}^2  + \left \langle \Delta_T^\sharp \psi, \psi \right \rangle_{L^2(T \times G)} \right) \\
    \end{align*}
    where $C_2 > 0$ is obtained by maximizing constants. Now notice that
    \begin{align*}
      \| \vv{W}_\ell^\sharp \psi \|_{L^2(T \times G)}^2 = \left \| \vv{W}_\ell^\sharp \left( \sum_{i = 1}^{d_\lambda^G} \psi_i \otimes \phi_i^{\lambda} \right) \right \|_{L^2(T \times G)}^2 =\left \| \sum_{i = 1}^{d_\lambda^G} (\vv{W}_\ell \psi_i) \otimes \phi_i^{\lambda} \right \|_{L^2(T \times G)}^2 = \sum_{i = 1}^{d_\lambda^G} \| \vv{W}_\ell \psi_i \|_{L^2(T)}^2
    \end{align*}
    where we decomposed $\psi$ w.r.t.~some orthonormal basis of $E_\lambda^G$ (its dependence on $\ell$ does not matter any longer); by Lemma~\ref{lem:l2estimvfd} there exists a constant $C_3 > 0$ such that
    \begin{align*}
      \| \vv{W}_\ell \psi_i \|_{L^2(T)} \leq C_3 \| \dd_T \psi_i \|_{L^2(T)}
    \end{align*}
    for all the indices involved, hence
    \begin{align*}
      \| \vv{W}_\ell^\sharp \psi \|_{L^2(T \times G)}^2 \leq C_3^2 \sum_{i = 1}^{d_\lambda^G} \| \dd_T \psi_i \|^2_{L^2(T)} = C_3^2 \left \langle \Delta_T^\sharp \psi, \psi \right \rangle_{L^2(T \times G)}
    \end{align*}
    so in particular     
    \begin{align*}
      \| \gr{a}_\ell (t, \vv{X}) \psi \|_{L^2(T \times G)}^2  &\leq \left(\|(\gr{a}_\ell(t, \vv{X})+ \vv{W}_\ell^{\sharp}) \psi \|_{L^2(T \times G)} + \| \vv{W}_\ell^\sharp \psi \|_{L^2(T \times G)} \right)^2 \\
      &\leq 2 \left(\| (\gr{a}_\ell(t, \vv{X})+ \vv{W}_\ell^{\sharp})\psi \|_{L^2(T \times G)}^2 + \| \vv{W}_\ell^\sharp \psi \|_{L^2(T \times G)}^2 \right) \\
      &\leq 2 \left(\| (\gr{a}_\ell(t, \vv{X})+ \vv{W}_\ell^{\sharp}) \psi \|_{L^2(T \times G)}^2 + C_3^2 \left \langle \Delta_T^\sharp \psi, \psi \right \rangle_{L^2(T \times G)} \right)
    \end{align*}
    from which we conclude that
     \begin{align*}
       \sum_{\ell = 1}^N \sum_{i = 1}^{d^G_\lambda} \int_T |\psi_i^\ell(t)|^2 |\gamma_{i}^{\lambda, \ell}|^2 \dd V_T(t) &\leq C_2 (1 + \lambda) \left( \sum_{\ell = 1}^N \| \gr{a}_\ell (t, \vv{X}) \psi \|_{L^2(T \times G)}^2  + \left \langle \Delta_T^\sharp \psi, \psi \right \rangle_{L^2(T \times G)} \right) \nonumber\\
       &\leq C_4 (1 + \lambda) \left( \sum_{\ell = 1}^N \| (\gr{a}_\ell(t, \vv{X})+ \vv{W}_\ell^{\sharp}) \psi \|_{L^2(T \times G)}^2  + \left \langle \Delta_T^\sharp \psi, \psi \right \rangle_{L^2(T \times G)} \right) \nonumber\\
       &= C_4 (1 + \lambda) \left \langle P \psi, \psi \right \rangle_{L^2(T \times G)}
    \end{align*}
     where the last equality follows from Lemma~\ref{lem:energy_identity}. Plugging this back into~\eqref{eq:psinorm} finishes our proof.

  \end{proof}
\end{Prop}


\begin{proof}[Proof of Theorem~\ref{thm:thm15}] Let $u \in \D'(T \times G)$ be such that $f \dfn Pu \in \cinfty(T \times G)$. Since $\tilde{P}$ is elliptic by Corollary~\ref{cor:ulambdasmooth} we have that $\mathcal{F}^G_\lambda (u) \in \cinfty(T; E_\lambda^G)$ for every $\lambda \in \sigma(\Delta_G)$, and by Proposition~\ref{prop:final_inequality} -- applied to $\psi = \mathcal{F}^G_\lambda (u)$ -- there exist $C, \rho > 0$ and $\lambda_0 \in \sigma(\Delta_G)$ such that
  \begin{align*}
    \| \mathcal{F}^G_\lambda (u) \|_{L^2(T \times G)} &\leq C^{-1} (1 + \lambda)^{\rho} \| \mathcal{F}^G_\lambda (f) \|_{L^2(T \times G)} , \quad \forall \lambda \geq \lambda_0
  \end{align*}
  after a suitable application of Schwarz inequality. But since $f$ is smooth by Corollary~\ref{cor:partial_smoothness} for every $s > 0$ there exists $C_s > 0$ such that
  \begin{align*}
    \| \mathcal{F}^G_\lambda(f) \|_{L^2(T \times G)} &\leq C_s (1 + \lambda)^{-s}, \quad \forall \lambda \in \sigma(\Delta_G)
  \end{align*}
  from which we conclude that for every $s > 0$ there exists $C'_s > 0$ such that
  \begin{align*}
    \| \mathcal{F}^G_\lambda(u) \|_{L^2(T \times G)} &\leq C'_s (1 + \lambda)^{-s}, \quad \forall \lambda \geq \lambda_0.
  \end{align*}
  It is simple to see that, increasing $C'_s$ if necessary, we obtain that the last inequality holds for every $\lambda \in \sigma(\Delta_G)$. We already saw in Corollary~\ref{cor:est_small_cone} that the ellipticity of $\tilde{P}$ entails, for every $s > 0$, the existence of $C > 0$ and $\theta \in (0, 1)$ such that~\eqref{eq:conethetaestimate} holds. Finally, Corollary~\ref{cor:partial_smoothness_converse} ensures smoothness of $u$.

  Furthermore, if $R$ is as in~\eqref{eq:Rdef} then it is certainly a positive semidefinite LPDO in $T \times G$, hence~\eqref{eq:InequalitypsiPpsi} implies that the same inequality holds if we exchange $P$ for $P_0 = P + R$. The latter is also a LPDO on $T \times G$ of the same kind as $P$, and $\tilde{P}_0$ is clearly elliptic too. Thus the argument above applies just as well for $P_0$ in place of $P$, proving its global hypoellipticity in $T \times G$.
\end{proof}

\section{A class of systems}\label{sec:classofsystem}

Our goal in this section is to prove Theorem~\ref{thm:PGHnecessaell}. Notice that its proof would be rather simple -- similar to that of Proposition~\ref{prop:first_nec_condition} -- if there were no vector fields $\vv{W}_\ell$ in~\eqref{eq:Pdef}. Here, however, we are once again studying a general $P$ defined by~\eqref{eq:Pdef} in $T\times G$ and $\mathcal{L}$ denotes the system of vector fields~\eqref{eq:sys_fromaell}. Our next lemma is the key to relate the condition~\eqref{eq:reg_aa} with the global hypoellipticity of $\mathcal{L}$ in $G$.

\begin{Lem} \label{lem:reg_aa} A distribution $u \in \D'(G)$ satisfies $\gr{a}_\ell(t, \vv{X})(1_T \otimes u) \in \cinfty(T \times G)$ for every $\ell \in \{1, \ldots, N\}$ if and only if $\vv{L} u \in \cinfty(G)$ for every $\vv{L} \in \mathcal{L}$.
  \begin{proof} Let $u \in \D'(G)$ be such that $\gr{a}_\ell(t,\vv{X})(1_T \otimes u) \in \cinfty(T \times G)$ for every $\ell \in \{1, \ldots, N\}$. We have
    \begin{align*}
      \gr{a}_\ell(t,\vv{X})(1_T \otimes u)  &= \sum_{j = 1}^m a_{\ell j}(t) \vv{X}_ju, \quad t \in T,
    \end{align*}
    which is smooth in $T \times G$, hence for any given $t_0 \in T$
    \begin{align*}
      \gr{a}_\ell(t_0) u  &= \sum_{j = 1}^m a_{\ell j}(t_0) \vv{X}_ju \in \cinfty(G), \quad \forall \ell \in \{1, \ldots, N\}.
    \end{align*}
    We conclude that $\vv{L} u \in \cinfty(G)$ for every $\vv{L} \in \mathcal{L}$ since $\mathcal{L} = \left\{ \gr{a}_\ell(t_0) \st t_0 \in T, \ \ell \in \{1, \ldots, N\} \right\}$ .

    For the converse, suppose that $u \in \D'(G)$ is such that $\vv{L} u \in \cinfty(G)$ for every $\vv{L} \in \mathcal{L}$. We select $\vv{L}_1, \ldots, \vv{L}_r \in \mathcal{L}$ a basis for $\Span_\R \mathcal{L}$ -- this is a finite dimensional space since it is contained in $\gr{g}$ -- so we can write, for each $\ell \in \{1, \ldots, N\}$,
    \begin{align*}
      \gr{a}_\ell(t) &= \sum_{j = 1}^r \alpha_{\ell j}(t) \vv{L}_j, \quad t \in T,
    \end{align*}
    where $\alpha_{\ell 1}, \ldots, \alpha_{\ell r} \in \cinfty(T; \R)$ are uniquely determined. Indeed, given $\MM_1, \ldots, \MM_{r'} \in \gr{g}$ such that $\vv{L}_1, \ldots, \vv{L}_r, \MM_1, \ldots, \MM_{r'}$ is a basis for $\gr{g}$, and letting $\tau_1, \ldots, \tau_r, \zeta_1, \ldots, \zeta_{r'} \in \gr{g}^*$ be the corresponding dual basis, we have that $\alpha_{\ell j} = \tau_j \circ \gr{a}_\ell$, which is smooth since $\tau_j$ is linear. We thus have
    \begin{align*}
      \gr{a}_\ell(t, \vv{X}) (1_T \otimes u) = \gr{a}_\ell(t) u = \sum_{j = 1}^r \alpha_{\ell j}(t) \vv{L}_j u \in \cinfty(T \times G), \quad \forall \ell \in \{1, \ldots, N\},
    \end{align*}
    since $\vv{L}_1 u, \ldots, \vv{L}_r u \in \cinfty(G)$ by hypothesis.
  \end{proof}
\end{Lem}
 
\begin{Prop}\label{Pro:66impliesLGH} Condition~\eqref{eq:reg_aa} holds if and only if $\mathcal{L}$ is~$\mathrm{(GH)}$ in $G$.
  \begin{proof} Assume first that $\mathcal{L}$ is~$\mathrm{(GH)}$ in $G$ and let $u \in \D'(G)$ be such that $\gr{a}_\ell(t,\vv{X})(1_T \otimes u) \in \cinfty(T \times G)$ for every $\ell \in \{1, \ldots, N\}$. By Lemma~\ref{lem:reg_aa} we have that $\vv{L} u \in \cinfty(G)$ for every $\vv{L} \in \mathcal{L}$, hence $u \in \cinfty(G)$.

    On the other hand, if one assumes~\eqref{eq:reg_aa} and letting $u \in \D'(G)$ be such that $\vv{L} u \in \cinfty(G)$ for every $\vv{L} \in \mathcal{L}$ then by Lemma~\ref{lem:reg_aa} we have that $\gr{a}_\ell(t,\vv{X})(1_T \otimes u) \in \cinfty(T \times G)$ for every $\ell \in \{1, \ldots, N\}$. We conclude that $u \in \cinfty(G)$.
  \end{proof}
\end{Prop}

\begin{proof}[Proof of Theorem~\ref{thm:PGHnecessaell}] Suppose that $P$ is~$\mathrm{(GH)}$ in $T \times G$ and let $u \in \D'(G)$ be such that $\gr{a}_\ell(t,\vv{X})(1_T \otimes u) \in \cinfty(T \times G)$ for every $\ell \in \{1, \ldots, N\}$. By~\eqref{eq:Pontensors} we have (recall that $\tilde{P}$ has no zeroth order terms, hence annihilates constants):
  \begin{align*}
    P(1_T \otimes u) &= - \sum_{\ell = 1}^N \gr{a}_\ell(t, \vv{X})^2 (1_T \otimes u) - \sum_{\ell = 1}^N \sum_{j = 1}^m (\vv{W}_\ell a_{\ell j}) \otimes (\vv{X}_j u).
  \end{align*}
  The first sum is smooth on $T \times G$ by assumption; we claim that so is the second. Indeed, define 
  \begin{align*}
    \tilde{\gr{a}}_\ell(t) &\dfn \sum_{j = 1}^m (\vv{W}_\ell a_{\ell j})(t) \vv{X}_j, \quad t \in T, \ \ell \in \{1, \ldots, N\}.
  \end{align*}
  Hence $\tilde{\gr{a}}_1, \ldots, \tilde{\gr{a}}_N: T \rarr \gr{g}$ are all smooth. We notice that $\ran \tilde{\gr{a}}_\ell \sset \Span_\R \ran \gr{a}_\ell$ for every $\ell \in \{1, \ldots, N\}$: given $t_0 \in T$ and $(U; \chi)= (U; t_1, \ldots, t_n)$ a coordinate chart of $T$ centered at $t_0$ we may write, in $U$,
  \begin{align*}
    \vv{W}_\ell &= \sum_{k = 1}^n b_{\ell k}(t) \frac{\del}{\del t_k}
  \end{align*}
  where $b_{\ell 1}, \ldots, b_{\ell n} \in \cinfty(U; \R)$, hence
  \begin{align*}
    \tilde{\gr{a}}_\ell(t_0) &= \sum_{j = 1}^m (\vv{W}_\ell a_{\ell j})(t_0) \vv{X}_j \\
    &= \sum_{j = 1}^m \sum_{k = 1}^n b_{\ell k}(t_0) \frac{\del a_{\ell j}}{\del t_k}(t_0) \vv{X}_j \\
    &= \sum_{k = 1}^n b_{\ell k}(t_0) \lim_{h \to 0} \sum_{j = 1}^m \frac{1}{h} \left( a_{\ell j} (\chi^{-1}( h e_k) - a_{\ell j} (\chi^{-1}(0)) \right) \vv{X}_j \\
    &= \sum_{k = 1}^n b_{\ell k}(t_0) \lim_{h \to 0} \frac{1}{h} \left( \gr{a}_\ell(\chi^{-1}(h e_k)) - \gr{a}_{\ell} (\chi^{-1}(0)) \right)
  \end{align*}
  certainly belongs to the vector space $\Span_\R \ran \gr{a}_\ell$ -- since all the Newton quotients above obviously do.

  We then define
  \begin{align*}
    \tilde{\mathcal{L}} &\dfn \bigcup_{\ell = 1}^N \ran \tilde{\gr{a}}_\ell
  \end{align*}
  which we have just proved to be contained in $\Span_\R \mathcal{L}$. Now since $\gr{a}_\ell(t,\vv{X})(1_T \otimes u) \in \cinfty(T \times G)$ for every $\ell \in \{1, \ldots, N\}$ it follows from Lemma~\ref{lem:reg_aa}  that $\vv{L} u \in \cinfty(G)$ for every $\vv{L} \in \mathcal{L}$, hence also for every $\vv{L} \in \Span_\R \mathcal{L}$ and, in particular, for every $\vv{L} \in \tilde{\mathcal{L}}$; by a second application of Lemma~\ref{lem:reg_aa}  we conclude that $\tilde{\gr{a}}_\ell(t,\vv{X})(1_T \otimes u) \in \cinfty(T \times G)$ for every $\ell \in \{1, \ldots, N\}$. It then follows that
  \begin{align*}
    P(1_T \otimes u) &= - \sum_{\ell = 1}^N \gr{a}_\ell(t, \vv{X})^2 (1_T \otimes u) - \sum_{\ell = 1}^N \tilde{\gr{a}}_\ell(t, \vv{X}) (1_T \otimes u) \in \cinfty(T \times G)
  \end{align*}
  and since $P$ is~$\mathrm{(GH)}$ in $T \times G$ we conclude that $1_T \otimes u \in \cinfty(T \times G)$ i.e~$u \in \cinfty(G)$, thus proving~\eqref{eq:reg_aa}. 
\end{proof}



\section{Remarks and examples}

We devote this section to motivate our hypotheses, to compare our results with previous ones in the literature and of course to provide some examples of operators that satisfy the hypotheses of Theorem~\ref{thm:thm15}. 

We start by analyzing hypothesis~\eqref{it:thm15_hyp1} in Theorem~\ref{thm:thm15}. The fact that $\gr{a}_{\ell}(t_1)$ and $\gr{a}_{\ell}(t_2)$ commute for every $t_1, t_2 \in T$ does not preclude non-commutativity of the vector fields belonging to distinct $\mathcal{L}_\ell$. In concrete examples, this is what prevents us from being ``thrown back'' to tori: more stringent hypotheses could inadvertently imply that $\gr{g}$ were already commutative to start with, see e.g.~Corollary~\ref{cor:GnonAbeliannoGH}. This leads us to our first example.

\begin{Exa} \label{exa:lines} For instance, choose $\vv{X}_1, \ldots, \vv{X}_N \in \gr{g}$  such that the Lie subalgebra generated by them is $\gr{g}$. Define  $\gr{a}_\ell(t) \dfn a_\ell(t)\vv{X}_\ell^\sharp$, for every $\ell \in \{1, \ldots, N\}$,  where each $a_\ell \in \cinfty(T; \R)$ is a nonzero function. Then condition~\eqref{it:thm15_hyp2} in Theorem~\ref{thm:thm15} is clearly satisfied, but also condition~\eqref{it:thm15_hyp1}: since $\lie \mathcal{L} = \gr{g}$ is obviously~$\mathrm{(GH)}$ in $G$ -- as it forms an elliptic system there --, so does $\mathcal{L}$ as a consequence of Lemma~\ref{lem:Lspanlie}. Now consider
\begin{align*}
  P &\dfn \Delta_T^\sharp - \sum_{\ell=1}^N \left( a_\ell(t)\vv{X}_\ell^\sharp + \vv{W}_\ell^\sharp \right)^2
\end{align*}
where $\vv{W}_1, \ldots, \vv{W}_N$ are skew-symmetric vector fields in $T$. Then, thanks to Theorem~\ref{thm:thm15}, $P$ is~$\mathrm{(GH)}$ in $T \times G$. Notice that this generalizes~\cite[Theorem~3]{albanese11}.
\end{Exa}

Note that if $G = \TT^m$ then $\lie \mathcal{L} = \gr{g}$ is possible if and only if $\mathcal{L}$ already contains $m$ linearly independent vector fields. For a compact connected but non-Abelian Lie group $G$ the  non-commutativity of $\gr{g}$ \emph{helps} us to reach condition~\eqref{it:thm15_hyp2} as $N$, the number of linearly independent vector fields in $\mathcal{L}$ in Example~\ref{exa:lines}, could be much smaller than $m = \dim \gr{g}$. For instance, in $G \dfn \mathrm{SU}(2)$ it is possible to find $\vv{X}_1, \vv{X}_2, \vv{X}_3$ three real vector fields forming a linear basis of $\gr{g}= \gr{su}(2)$ and such that $[\vv{X}_1, \vv{X}_2] = \vv{X}_3$. Then it is enough to choose non-vanishing $a_1, a_2 \in \cinfty(T; \R)$ and skew-symmetric vector fields $\vv{W}_1, \vv{W}_2$ in $T$ to conclude that
\begin{align*}
  P &\dfn \Delta_T^\sharp - \left( a_1(t) \vv{X}_1^\sharp + \vv{W}_1^\sharp \right)^{2} - \left( a_2(t) \vv{X}_2^\sharp + \vv{W}_2^\sharp \right)^{2}
\end{align*}
is globally hypoelliptic in $T \times G$. 

\subsection{Relationship with the notion of simultaneous approximability for vectors} \label{sec:nsa_vectors}

Before we  provide more examples, we compare Theorem~\ref{Thm:Toruscase} with~\cite[Theorem~1.5]{bfp17} where global hypoellipticity of the same model operator was studied. Even though both results established necessary and sufficient conditions for global hypoellipticity when $G$ is a torus, it may seem, at a first glance, that our necessary condition of $\mathcal{L}$ being~$\mathrm{(GH)}$ in $G$ has nothing to do with the notion of simultaneous approximability of a collection of vectors~\cite[Definition~1.2]{bfp17}\footnote{Properly adapted to the smooth setup (see condition~\eqref{it:condI} in Proposition~\ref{Prop:equivalenceNSA}): in that work the authors are interested in hypoellipticity w.r.t.~some classes of ultradifferentiable functions.}. Note that one does not need to assume that $T$ is a torus in order to state the notion of simultaneous approximability.

Yet, now we study the relationship between these two concepts. Still within the general setup, recall that $\gr{g}$ carries an inner product $\langle \cdot, \cdot \rangle$ and select $\vv{X}_1, \ldots, \vv{X}_m \in \gr{g}$ a linear basis. Let $\gr{a}_{1}, \ldots, \gr{a}_N$ be as in~\eqref{eq:aellcoordinates} and for each $\ell \in \{1, \ldots, N\}$ define
\begin{align*}
  \mathcal{A}_\ell &\dfn \Span_\R \{ a_{\ell 1}, \ldots, a_{\ell m} \} \sset C(T; \R).
\end{align*}
Notice that the linear map
\begin{align*}
  \begin{array}{c c c}
    \gr{g} & \longrightarrow & \mathcal{A}_\ell \\
    \vv{X} & \longmapsto & \sum_{j = 1}^m \langle \vv{X}, \vv{X}_j \rangle a_{\ell j}
  \end{array}
\end{align*}
is certainly onto, with kernel precisely $\mathcal{L}_\ell^\bot$: if $\vv{X} \in \gr{g}$ is such that
\begin{align*}
  \sum_{j = 1}^m \langle \vv{X}, \vv{X}_j \rangle a_{\ell j}(t) &= 0, \quad \forall t \in T,
\end{align*}
then
\begin{align*}
  0 = \left \langle \vv{X}, \sum_{j = 1}^m a_{\ell j}(t) \vv{X}_j \right \rangle = \langle \vv{X}, \gr{a}_\ell(t) \rangle, \quad \forall t \in T,
\end{align*}
that is, $\vv{X}$ is orthogonal to every element in $\ran \gr{a}_\ell$, and these generate $\mathcal{L}_\ell$. We thus have an isomorphism $\mathcal{L}_\ell \cong \mathcal{A}_\ell$. Their dimension will be denoted by $m^\ell$, and therefore there are indices $1 \leq j_1^\ell < \cdots < j_{m^\ell}^\ell \leq m$ such that
\begin{align*}
  \text{$a_{\ell j_1^\ell}, \ldots, a_{\ell j_{m^\ell}^\ell}$ form a basis of $\mathcal{A}_\ell$}.
\end{align*}
If we write the remaining indices as $1 \leq i_1^\ell < \cdots < i_{d^\ell}^\ell \leq m$ (where $d^\ell \dfn m - m^\ell$) then we can write
\begin{align*}
  a_{\ell i_q^\ell} &= \sum_{p = 1}^{m^\ell} \lambda_{qp}^\ell a_{\ell j_p^\ell}, \quad q \in \{1, \ldots, d^\ell \}
\end{align*}
where the constants $\lambda_{qp}^\ell \in \R$ are uniquely determined. Thus a $\vv{X} \in \gr{g}$ belongs to $\mathcal{L}_\ell^\bot$ if and only if
\begin{align*}
  0 = \sum_{j = 1}^m \langle \vv{X}, \vv{X}_j \rangle a_{\ell j} = \sum_{p = 1}^{m^\ell} \langle \vv{X}, \vv{X}_{j_{p}^\ell} \rangle a_{\ell j_{p}^\ell} + \sum_{q = 1}^{d^\ell} \langle \vv{X}, \vv{X}_{i_{q}^\ell} \rangle a_{\ell i_{q}^\ell} = \sum_{p = 1}^{m^\ell} \left( \langle \vv{X}, \vv{X}_{j_{p}^\ell} \rangle + \sum_{q = 1}^{d^\ell} \lambda_{qp}^\ell \langle \vv{X}, \vv{X}_{i_{q}^\ell} \rangle \right) a_{\ell j_{p}^\ell}
\end{align*}
i.e.
\begin{align*}
  \langle \vv{X}, \vv{X}_{j_{p}^\ell} \rangle + \sum_{q = 1}^{d^\ell} \lambda_{qp}^\ell \langle \vv{X}, \vv{X}_{i_{q}^\ell} \rangle &= 0, \quad \forall p \in \{1, \ldots, m^\ell\}, 
\end{align*}
meaning that $\vv{X}$ is orthogonal to
\begin{align*}
  \vv{L}_p^\ell &\dfn \vv{X}_{j_{p}^\ell} + \sum_{q = 1}^{d^\ell} \lambda_{qp}^\ell \vv{X}_{i_{q}^\ell}, \quad p \in \{1, \ldots, m^\ell\}. 
\end{align*}
That is, $\vv{L}_1^\ell, \ldots, \vv{L}_{m^\ell}^\ell$ form a basis for $\mathcal{L}_\ell$ (they are clearly linearly independent), so by Proposition~\ref{prop:ghfinitevfs} and Lemma~\ref{lem:Lspanlie} $\mathcal{L}$ is~$\mathrm{(GH)}$ in $G$ if and only if there exist $C, \rho > 0$ and $\lambda_0 \in \sigma(\Delta_G)$ such that
\begin{align}
  \left( \sum_{\ell = 1}^N \sum_{p = 1}^{m^\ell} \|\vv{L}_p^\ell \phi \|_{L^2(G)}^2 \right)^{\frac{1}{2}} &\geq C(1 + \lambda)^{-\rho}\| \phi \|_{L^2(G)}, \quad \forall \phi \in E^G_\lambda, \ \lambda \geq \lambda_0. \label{eq:abstract_diophantine}
\end{align}

Now let us see how these things work on a torus. When $G = \TT^m$ we have that $\vv{X}_j \dfn \del_{x_j}$, $j \in \{1, \ldots, m\}$, form a basis of its Lie algebra $\gr{g} \cong \R^m$ -- which is a commutative Lie algebra, so the standard inner product (i.e.~the one for which $\vv{X}_1, \ldots, \vv{X}_m$ is an orthonormal basis) is automatically $\ad$-invariant, and the associated Laplace-Beltrami operator thus reads
\begin{align*}
  \Delta_G = - \sum_{j = 1}^m \vv{X}_j^2 = - \sum_{j = 1}^m \del_{x_j}^2 .
\end{align*}
Thanks to Fourier Analysis we have that $\sigma(\Delta_G) = \{ n^2 \st n \in \Z_+ \}$ and 
\begin{align*}
  E_\lambda^G &= \Span_\C \{ e^{i x \xi} \st \xi \in \Z^m, \ |\xi|^2 = \lambda \}, \quad \forall \lambda \in \sigma(\Delta_G),
\end{align*}
the exponentials actually forming an orthonormal basis of $E_\lambda^G$, hence
\begin{align*}
  \|\vv{L}_p^\ell e^{i x \xi}\|_{L^2(\TT^m)} = \left\| \left( \del_{x_{j_{p}^\ell}} + \sum_{q = 1}^{d^\ell} \lambda_{qp}^\ell \del_{x_{i_{q}^\ell}} \right) e^{i x \xi} \right\|_{L^2(\TT^m)} = \left| \xi_{j_{p}^\ell} + \sum_{q = 1}^{d^\ell} \lambda_{qp}^\ell \xi_{i_{q}^\ell} \right|.
\end{align*}
It follows from~\eqref{eq:abstract_diophantine} that if $\mathcal{L}$ is ~$\mathrm{(GH)}$ in $G = \TT^m$ then there exist $C, \rho > 0$ and $n_0 \in \N$ such that
\begin{align}
  \left( \sum_{\ell = 1}^N \sum_{p = 1}^{m^\ell} \left| \xi_{j_{p}^\ell} + \sum_{q = 1}^{d^\ell} \lambda_{qp}^\ell \xi_{i_{q}^\ell} \right|^2 \right)^{\frac{1}{2}} &\geq C(1 + |\xi|^2)^{-\rho}, \quad \forall \xi \in \Z^m, \ |\xi| \geq n_0. \label{eq:cond_intermed}
\end{align}
Conversely, since every $\phi \in E_\lambda^G$ can be written as
\begin{align*}
  \phi &= \sum_{|\xi|^2 = \lambda} \phi_\xi e^{ix \xi}, \quad \phi_\xi \in \C,
\end{align*}
if~\eqref{eq:cond_intermed} holds then for $|\xi| \geq n_0$:
\begin{align*}
  \sum_{\ell = 1}^N \sum_{p = 1}^{m^\ell} \|\vv{L}_p^\ell \phi \|_{L^2(\TT^m)}^2 &= \sum_{\ell = 1}^N \sum_{p = 1}^{m^\ell} \left \| \sum_{|\xi|^2 = \lambda} \phi_\xi \vv{L}_p^\ell e^{ix \xi} \right\|_{L^2(\TT^m)}^2 \\
  &= \sum_{\ell = 1}^N \sum_{p = 1}^{m^\ell} \sum_{|\xi|^2 = \lambda} |\phi_\xi|^2  \left| \xi_{j_{p}^\ell} + \sum_{q = 1}^{d^\ell} \lambda_{qp}^\ell \xi_{i_{q}^\ell} \right|^2 \\
  &\geq \sum_{|\xi|^2 = \lambda} |\phi_\xi|^2 C^2 (1 + |\xi|^2)^{-2\rho} \\
  &= C^2 (1 + \lambda)^{-2\rho} \| \phi \|_{L^2(\TT^m)}^2
\end{align*}
so~\eqref{eq:abstract_diophantine} also holds, and  $\mathcal{L}$ is~$\mathrm{(GH)}$ in $\TT^m$.

Inequality~\eqref{eq:cond_intermed} not only resembles the smooth version of the non-simultaneous approximability condition introduced in~\cite[Definition~1.2]{bfp17} but it is actually equivalent to it. This is the content of the next proposition, for which we introduce further notation in order to simplify its statement. For each $\ell \in \{1, \ldots, N\}$, assume that $d^{\ell}>0$ and $m^{\ell}>0$, and denote, for $\xi \in \R^m$,
\begin{align*}
  \xi'_{(\ell)} \dfn \left( \xi_{j_1^\ell}, \ldots, \xi_{j_{m^\ell}^\ell} \right) \in \R^{m^\ell}, &\quad  \xi''_{(\ell)} \dfn \left( \xi_{i_1^\ell}, \ldots, \xi_{i_{d^\ell}^\ell} \right) \in \R^{d^\ell},
\end{align*}
and also
\begin{align*}
  v_p^\ell &\dfn \left( \lambda_{1p}^\ell, \ldots, \lambda_{{d^\ell} p}^\ell \right) \in \R^{d^\ell}, \quad p \in \{1, \ldots, m^\ell\}.
\end{align*}

\begin{Prop}\label{Prop:equivalenceNSA} The following are equivalent:
  \begin{enumerate}
  \item \label{it:condG} There exist $C, \rho > 0$ and $n_0 \in \N$ such that~\eqref{eq:cond_intermed} holds i.e.
    \begin{align*}
      \left( \sum_{\ell = 1}^N \sum_{p = 1}^{m^\ell} \left| \xi_{j_{p}^\ell} + v_p^\ell \cdot \xi''_{(\ell)} \right|^2 \right)^{\frac{1}{2}} &\geq C(1 + |\xi|^2)^{-\rho}, \quad \forall \xi \in \Z^m, \ |\xi| \geq n_0.
    \end{align*}
  \item \label{it:condI} There exist $B, M > 0$ such that for each $\xi \in \Z^m \setminus 0$ there exist $\ell \in \{1, \ldots, N\}$ and $p \in \{1, \ldots, m^\ell\}$ such that
    \begin{align*}
      \left| \xi_{j_{p}^\ell} + v_p^\ell \cdot \xi''_{(\ell)} \right| &\geq B \left( 1 + \left| \xi ''_{(\ell)} \right| \right)^{-M}.
    \end{align*}
  \end{enumerate}
\end{Prop}
We omit the proof as it relies on standard calculations. Now one can immediately recognize condition~\eqref{it:condI} above as the bona fide smooth version of the Diophantine condition in~\cite[Definition~1.2]{bfp17}.

In $T \times \TT^m$ consider an operator $P$ as in~\eqref{eq:PdefLaplaceBeltrami}. We shall say that $P$ satisfies the {\it non-simultaneous approximability condition} if one of the following holds for the family $\gr{a}_1, \ldots, \gr{a}_N$:
\begin{itemize}
\item there  exists $\ell \in \{1, \ldots, N\}$ such that $d^{\ell} = 0$;
\item we can relabel the indices in order to obtain $0 < N'\leq N$ such that none of $\gr{a}_1, \ldots, \gr{a}_{N'}$ is identically zero and when we apply the procedure described above we obtain a collection $v_1^{1}, \ldots, v_{m^{1}}^{1}, v_{1}^{2}, \ldots, v^{N'}_{m^{N'}}$ satisfying one of the equivalent properties in the Proposition~\ref{Prop:equivalenceNSA}.
\end{itemize}

\begin{Cor} \label{cor:equivalence_dc_conditions_torus} When $G = \TT^m$ our system $\mathcal{L}$ in~\eqref{eq:sys_fromaell} is~$\mathrm{(GH)}$ in $G$ if and only if $P$ satisfies the non-simultaneous approximability condition.
\end{Cor}

\begin{Exa} On a compact, connected and oriented manifold $T$, define a LPDO $P$ on $T\times \TT^{2}$ by
  \begin{align*}
    P &\dfn \Delta_T^\sharp - \left(\del_{x_1} + \alpha \del_{x_2} \right)^{2} - \left(\beta \del_{x_1} + \del_{x_2} \right)^{2},
  \end{align*}
  where $\alpha, \beta \in \Q$ and $\alpha \beta \neq 1$. Since both $\alpha$ and $\beta$ are rational it is clear, thanks to a classical result from Greenfield and Wallach~\cite{gw72}, that neither $\vv{L}_1 \dfn \del_{x_1} + \alpha \del_{x_2}$ nor $\vv{L}_2 \dfn \beta \del_{x_1} + \del_{x_2}$ is globally hypoelliptic in $\TT^2$. It is plain however that $\vv{L}_1, \vv{L}_2$ together generate the tangent space of $\TT^{2}$ at every point therefore the system $\mathcal{L} \dfn \{\vv{L}_1, \vv{L}_2\}$ is $\mathrm{(GH)}$ in $\TT^{2}$ and $P$ is $\mathrm{(GH)}$ in $T\times \TT^{2}$. 
\end{Exa}

\subsection{Comparison with H{\"o}rmander's condition} \label{sec:hormander}

Back to a general compact Lie group $G$, with Lie algebra $\gr{g}$, let $\gr{h} \sset \gr{g}$ be a Lie subalgebra. We regard $\cinfty(T; \gr{h})$ as a subset of $\gr{X}(T \times G)$, the Lie algebra of all real, smooth vector fields on $T \times G$: as such, it is a Lie subalgebra of the latter. Indeed, given a basis $\vv{L}_1, \ldots, \vv{L}_r$ of $\gr{h}$, any $\gr{a} \in \cinfty(T; \gr{h})$ can be written as
\begin{align*}
  \gr{a}(t) &= \sum_{j = 1}^r a_j(t) \vv{L}_j, \quad t \in T,
\end{align*}
where $a_1, \ldots, a_r \in \cinfty(T; \R)$ are uniquely determined; from this observation our claim follows easily. Moreover, for any real vector field $\vv{W}$ in $T$ we have that
\begin{align*}
  \gr{a}_{\vv{W}} &\dfn \sum_{j = 1}^r (\vv{W} a_j) \vv{L}_j
\end{align*}
also belongs to $\cinfty(T; \gr{h})$ by definition.
\begin{Lem} The set of all vector fields $\vv{Y}$ in $T \times G$ of the form $\vv{Y} = \gr{a}(t, \vv{X}) + \vv{W}^\sharp$ where $\gr{a} \in \cinfty(T; \gr{h})$ and $\vv{W} \in \gr{X}(T)$ is a Lie subalgebra of $\gr{X}(T \times G)$.
  \begin{proof} Indeed, if $\vv{Y} = \gr{a}(t, \vv{X}) + \vv{W}^\sharp$, $\tilde{\vv{Y}} = \tilde{\gr{a}}(t, \vv{X}) + \tilde{\vv{W}}^\sharp$ are two such fields then
    \begin{align*}
      [\vv{Y}, \tilde{\vv{Y}}] = [\gr{a}(t, \vv{X}), \tilde{\gr{a}}(t, \vv{X})] + ( \tilde{\gr{a}}_{\vv{W}})(t, \vv{X}) - ( \gr{a}_{\tilde{\vv{W}}})(t, \vv{X}) + [\vv{W}, \tilde{\vv{W}}]^\sharp
    \end{align*}
    certainly has the same form.
  \end{proof}
\end{Lem}

Now let $\mathcal{L}$ be as in~\eqref{eq:sys_fromaell} and let $\gr{h} \dfn \lie \mathcal{L} \sset \gr{g}$. Then, as we have seen, $\Theta \dfn \cinfty(T; \gr{h}) + \gr{X}(T)$ is a Lie subalgebra of $\gr{X}(T \times G)$. Given $\vv{Y}_1, \ldots, \vv{Y}_N \in \Theta$, all of them are of form $\vv{Y}_\ell = \gr{a}_\ell(t, \vv{X}) + \vv{W}_\ell^\sharp$ for some $\gr{a}_\ell \in \cinfty(T; \gr{h})$ and $\vv{W}_\ell \in \gr{X}(T)$. Assume that for a given $(t, x) \in T \times G$ the following condition holds:
\begin{align}\label{eq:Hormandercondition}
  \text{$\exists \vv{Z}_1, \ldots, \vv{Z}_\nu \in \gr{X}(T)$ such that the set $\{\vv{Z}_1^\sharp, \ldots, \vv{Z}_\nu^\sharp, \vv{Y}_1, \ldots, \vv{Y}_N\}$ is of finite type at $(t,x)$}.
\end{align}
Then it follows from the fact that $\Theta$ is a Lie algebra containing $\vv{Z}_1^\sharp, \ldots, \vv{Z}_\nu^\sharp, \vv{Y}_1, \ldots, \vv{Y}_N$ that
\begin{align*}
  \Theta_{(t,x)} \dfn \left \{ \vv{Y}|_{(t,x)} \st \vv{Y} \in \Theta \right \} = T_{(t,x)} (T \times G)
\end{align*}
so in particular $(\pi_G)_* \Theta_{(t,x)} = T_x G$ where $(\pi_G)_*: T_{(t,x)} (T \times G) \rarr T_x G$ is the projection map.
\begin{Prop} If $(\pi_G)_* \Theta_{(t,x)} = T_x G$ for some $(t,x) \in T \times G$ then $\lie \mathcal{L} = \gr{g}$.
  \begin{proof} Given $\vv{X} \in \gr{g}$ arbitrary, there exists $\vv{Y} \in \Theta$ such that $(\pi_G)_* \vv{Y}|_{(t,x)} = \vv{X}|_x$. If we write
    \begin{align*}
      \vv{Y} &= \sum_{j = 1}^r a_j \vv{L}_j^\sharp + \vv{W}^\sharp,
    \end{align*}
    then
    \begin{align*}
      \vv{Y}|_{(t,x)} &= \sum_{j = 1}^r a_j(t) \vv{L}_j|_x + \vv{W}|_t,
    \end{align*}
    hence
    \begin{align*}
      \vv{X}|_x = (\pi_G)_* \Theta|_{(t,x)} = \sum_{j = 1}^r a_j(t) \vv{L}_j|_x.
    \end{align*}
    As two left-invariant vector fields are the same if and only if they match at a single point we conclude
    \begin{align*}
      \sum_{j = 1}^r a_j(t) \vv{L}_j &= \vv{X}
    \end{align*}
    where the left-hand side belongs to $\lie \mathcal{L}$ since so do $\vv{L}_1, \ldots, \vv{L}_r$ (recall that $t \in T$ remains fixed).
  \end{proof}
\end{Prop}
Since $\gr{g}$ is~$\mathrm{(GH)}$ in $G$ we conclude from Lemma~\ref{lem:Lspanlie} that:
\begin{Cor} \label{cor:hormander_point} If $\vv{Y}_1, \ldots, \vv{Y}_N$ satisfy property~\eqref{eq:Hormandercondition} at some point $(t,x) \in T \times G$ then $\mathcal{L}$ is~$\mathrm{(GH)}$ in $G$.
\end{Cor}
Yet, very simple examples show that we may have $\lie \mathcal{L} = \gr{g}$ -- which is \emph{stronger} than $\mathcal{L}$ being~$\mathrm{(GH)}$ in $G$ -- while the finite type condition fails by far at every point: back to Example~\ref{exa:lines}, if $m \geq 2$ and $a_1, \ldots, a_N$ have pairwise disjoint supports then the finite type condition~\eqref{eq:Hormandercondition} for $\vv{Y}_\ell \dfn a_\ell(t) \vv{X}_\ell^\sharp + \vv{W}_\ell^\sharp$, $\ell \in \{1, \ldots, N\}$, fails everywhere since no $\vv{X}_\ell$ can generate the whole $\gr{g}$.


\subsection{A necessary condition based on Sussmann's orbits} \label{sec:orbits}

Let $M$ be a compact manifold as in Section~\ref{sec:preliminaries}. We will now show a simple result which illustrates the connection between the topology of the Sussmann's orbits of a system $\mathcal{L}$ of vector fields on $M$ -- or, rather, how they are immersed into the ambient manifold -- and the global hypoellipticity of $\mathcal{L}$ in $M$. This has some interesting consequences (Corollary~\ref{cor:GnonAbeliannoGH}) which better contextualize the hypotheses of Theorem~\ref{thm:thm15}.

Recall that the \emph{orbit} of $\mathcal{L}$ through $x_0$ is the set of all $x \in M$ enjoying the following property: there exists a continuous curve $\gamma:[0, \delta] \rarr M$ (for some $\delta > 0$) with endpoints $\gamma(0) = x_0$ and $\gamma(\delta) = x$ and a partition $0 = t_0 < t_1 < \cdots < t_\kappa = \delta$ such that on each open subinterval $(t_j, t_{j + 1})$ -- for $j \in \{0, \ldots, \kappa - 1\}$ -- the curve $\gamma$ is $\mathscr{C}^1$ and an integral curve of some $\vv{L}_j \in \mathcal{L}$. We denote it by $\mathrm{Orb}_{\mathcal{L}}(x_0)$. Sussmann's Orbit Theorem~\cite{sussmann73} states that the orbits of $\mathcal{L}$ are all immersed connected submanifolds of $M$.

If for simplicity we assume that $M = G$ is a compact Lie group and $\mathcal{L} \sset \gr{g}$ is a system of left-invariant vector fields on $G$, then one has a much more precise result (see e.g.~\cite[Lemma~3.4]{sachkov07}):
\begin{enumerate}
\item $\mathrm{Orb}_\mathcal{L}(e)$ is the connected Lie subgroup of $G$ whose Lie algebra is $\lie \mathcal{L} \sset \gr{g}$; and
\item $\mathrm{Orb}_\mathcal{L}(x_0) = x_0 \cdot \mathrm{Orb}_\mathcal{L}(e)$ for every $x_0 \in G$.
\end{enumerate}
In that case, the orbits are precisely the integral submanifolds of the regular involutive distribution
\begin{align*}
  \lie_x \mathcal{L} &\dfn \{ \vv{X}|_x \st \vv{X} \in \lie \mathcal{L} \} \sset T_x G, \quad x \in G,
\end{align*}
so these results are actually a consequence of Frobenius Theorem.
\begin{Prop} \label{prop:orbits} If $\mathcal{L}$ is~$\mathrm{(GH)}$ then all of its orbits are dense in $G$.
  \begin{proof} It is enough to prove that $\mathrm{Orb}_\mathcal{L}(e)$ is dense in $G$ as the remaining orbits are left translations of it. Let $H \sset G$ denote its closure. It is certainly a subgroup of $G$, and since it is closed it is a Lie subgroup of $G$. Moreover, the set $G / H$ is a smooth manifold with dimension $\dim G - \dim H$, which is positive if one assumes that $H \neq G$, and the canonical projection $\pi: G \rarr G/H$ is a smooth submersion~\cite[Theorem~9.22]{lee_smooth}. In that case, let $v \in \mathscr{C}^1(G/H) \setminus \cinfty(G/H)$ and take $u \dfn \pi^* v \in \mathscr{C}^1(G) \setminus \cinfty(G)$. Then $u$ is annihilated by every $\vv{X} \in \gr{g}$ tangent to $H$, hence in particular by any $\vv{X} \in \mathcal{L}$ since
    \begin{align*}
      \mathcal{L} \sset \lie \mathcal{L} \sset \gr{h} \dfn \text{the Lie algebra of $H$}.
    \end{align*}
    Thus $\mathcal{L}$ would not be~$\mathrm{(GH)}$.
  \end{proof}
\end{Prop}

Having in mind condition~\eqref{it:thm15_hyp1} in Theorem~\ref{thm:thm15} we would like to point out in our next result that one must be really careful when assigning hypotheses to $P$ in order to ensure its global hypoellipticity: too strong ones may inadvertently also ensure that $G$ must have been a torus to start with!
\begin{Cor} \label{cor:GnonAbeliannoGH} If $G$ is a non-commutative Lie group and $\mathcal{L} \sset \gr{g}$ is a family of pairwise commuting vector fields then $\mathcal{L}$ cannot be~$\mathrm{(GH)}$.
  \begin{proof} Notice that $\lie \mathcal{L}$ is a commutative Lie subalgebra of $\gr{g}$, hence must be contained in a maximal commutative Lie subalgebra $\gr{h} \sset \gr{g}$. Indeed, define inductively a sequence $\{\gr{h}_\nu \}_{\nu \in \Z_+}$ of linear subspaces of $\gr{g}$ by $\gr{h}_0 \dfn \lie \mathcal{L}$ and $\gr{h}_{\nu + 1} \dfn \Span_\R (\gr{h}_{\nu} \cup \{\vv{X}_\nu\})$, where $\vv{X}_\nu \notin \gr{h}_{\nu}$ but commutes with every element in $\gr{h}_{\nu}$, for $\nu \in \Z_+$. This is an increasing family of commutative Lie subalgebras of $\gr{g}$ which must stabilize at some step by finite dimensionality -- it means precisely that the said step, call it $\gr{h}$, defines a maximal commutative Lie subalgebra of $\gr{g}$.

    Let then $H \sset G$ be the unique connected Lie subgroup of $G$ whose Lie algebra is $\gr{h}$. It is certainly commutative (since so is $\gr{h}$), and it must be closed: otherwise, its closure $\overline{H}$ would be a bigger commutative, connected subgroup of $G$, hence it is a Lie subgroup of $G$, whose Lie algebra $\overline{\gr{h}}$ would contain $\gr{h}$ properly (as $H \sset \overline{H}$) and would also be commutative (since so is $\overline{H}$), thus violating the maximality of $\gr{h}$.

    Because $\lie \mathcal{L} \sset \gr{h}$ we have that every vector field in $\mathcal{L}$ is tangent to $H$, hence $\mathrm{Orb}_\mathcal{L}(e) \sset H$ and thus
    \begin{align*}
      \overline{\mathrm{Orb}_\mathcal{L}(e)} \sset \overline{H} \neq G
    \end{align*}
    as we are assuming $G$ non-commutative. In particular $\mathrm{Orb}_\mathcal{L}(e)$ is not dense in $G$ and the conclusion follows from Proposition~\ref{prop:orbits}.
  \end{proof}
\end{Cor}

\section{Operators with mostly constant coefficients} \label{sec:mostly_constant}

In this final section we explore other results ensuring global hypoellipticity of operators $P$ as in~\eqref{eq:Pdef}. Here we allow more general ``leading terms'' $Q$, unlike Theorem~\ref{thm:thm15} in which we have $Q = \Delta_T$, but paying the price of more restrictive assumptions on the vector fields $\gr{a}_\ell(t, \vv{X})$. The following one is an extension of~\cite[Theorem~1.9]{bfp17}. 
\begin{Thm} \label{thm:thm19} Let $P$ in~\eqref{eq:Pdef} be of the form
  \begin{align*}
    P &= Q^\sharp - \sum_{\ell = 1}^{N'} \left( \vv{L}_\ell^\sharp + \vv{W}_\ell^\sharp \right)^2 - \sum_{\ell = N' + 1}^{N} \left( \gr{a}_\ell(t, \vv{X}) + \vv{W}_\ell^\sharp \right)^2
  \end{align*}
  where $Q$ is positive semidefinite in $T$ -- i.e.~$\langle Q \psi, \psi \rangle_{L^2(T)} \geq 0$ for every $\psi \in \cinfty(T)$ --, $\vv{L}_1, \ldots, \vv{L}_{N'} \in \gr{g}$ and $\vv{W}_1, \ldots, \vv{W}_N$ are skew-symmetric and such that $\tilde{P}= Q- \vv{W}_1^{2}- \cdots - \vv{W}_{N}^{2}$ is elliptic.

  Assume moreover that
  \begin{enumerate}
  \item \label{it:thm19_hyp1} $\vv{W}_1, \ldots, \vv{W}_{N'}$ commute with $\Delta_T$ and that
  \item \label{it:thm19_hyp3} the system $\{ \vv{Y}_\ell \dfn \vv{L}_\ell^\sharp + \vv{W}_\ell^\sharp \st \ell = 1, \ldots, N'\}$ is~$\mathrm{(GH)}$ in $T \times G$.
  \end{enumerate}
  Then $P$ is~$\mathrm{(GH)}$ in $T \times G$.
\end{Thm}
\begin{Rem} Property~\eqref{it:thm19_hyp3} above is stronger than $\mathcal{L}$ in~\eqref{eq:sys_fromaell} being~$\mathrm{(GH)}$ in $G$ as it clearly implies~\eqref{eq:reg_aa} -- which is equivalent to the latter by Proposition~\ref{Pro:66impliesLGH} -- independently of the remaining assumptions.
\end{Rem}
\begin{proof} Hypothesis~\eqref{it:thm19_hyp1} ensures that the vector fields $\vv{Y}_1, \ldots, \vv{Y}_{N'}$ commute with the full Laplace-Beltrami operator $\Delta = \Delta_T^\sharp + \Delta_G^\sharp$ on $T \times G$. Therefore, hypothesis~\eqref{it:thm19_hyp3} implies, by means of Proposition~\ref{prop:ghfinitevfs} (see also Corollary~\ref{cor:relationshipeigens} and the results in Section~\ref{sec:partial_FPM}), the following: there exist $C, R, \rho > 0$ such that for all $(\mu, \lambda) \in \sigma (\Delta_T) \times \sigma(\Delta_G)$ with $\mu + \lambda \geq R$ we have
  \begin{align}\label{eq:mostlyconstant}
    \left( \sum_{\ell = 1}^{N'} \| \vv{Y}_\ell \varphi \|_{L^2(T \times G)}^2 \right)^\frac{1}{2} &\geq C (1 + \mu + \lambda)^{-\rho} \| \varphi\|_{L^2(T \times G)} , \quad \forall \varphi \in E_\mu^T \otimes E_\lambda^G.
  \end{align}
  
  Let $u \in \D'(T \times G)$ be such that $f \dfn Pu \in \cinfty(G)$. Since we are assuming $\tilde{P}$ elliptic in $T$ we have by Corollary~\ref{cor:ulambdasmooth} that $\mathcal{F}_\lambda^G(u)$ is smooth for every $\lambda \in \sigma(\Delta_G)$. As $\vv{Y}_1, \ldots, \vv{Y}_{N'}$ commute with $\Delta$ they behave well under both the partial Fourier projection maps i.e.~including $\mathcal{F}^T$, and not only $\mathcal{F}^G$: 
  \begin{align}
    \left \| \vv{Y}_\ell \mathcal{F}_\lambda^G(u) \right \|_{L^2(T \times G)}^2 = \sum_{\mu \in \sigma(\Delta_T)} \left \|  \mathcal{F}^T_\mu \left( \vv{Y}_\ell \mathcal{F}_\lambda^G(u) \right) \right \|_{L^2(T \times G)}^2 = \sum_{\mu \in \sigma(\Delta_T)} \left \|  \vv{Y}_\ell \left(  \mathcal{F}^T_\mu \mathcal{F}_\lambda^G(u) \right) \right \|_{L^2(T \times G)}^2 \label{eq:relellpcommute}
  \end{align}
  for $\ell \in \{1, \ldots, N'\}$, whatever $\lambda \in \sigma(\Delta_G)$.
  
  Now let $s > 0$. By Corollary~\ref{cor:est_small_cone} there exist $C_1 > 0$ and $\theta \in (0,1)$ such that
  \begin{align*}
    \| \mathcal{F}^T_\mu \mathcal{F}^G_\lambda (u)\|_{L^2(T \times G)} &\leq C_1 (1 + \mu + \lambda)^{-s - 2n}, \quad \forall (\mu, \lambda) \in \Lambda_\theta,
  \end{align*}
  where $n = \dim T$ and $\Lambda_\theta \sset \sigma(\Delta_T) \times \sigma(\Delta_G)$ is defined as in~\eqref{eq:Atheta}. We look at its complement
  \begin{align*}
    \Lambda_\theta^c &= \{ (\mu, \lambda) \in \sigma(\Delta_T) \times \sigma(\Delta_G) \st (1 + \lambda) > (1 + \mu)^\theta \}
  \end{align*}
  where it holds that
  \begin{align*}
    1 + \mu + \lambda < (1 + \lambda)^{\frac{1}{\theta}} + \lambda \leq (1 + \lambda)^{1 +  \frac{1}{\theta}} \leq (1 + \lambda)^{\frac{2}{\theta}}
  \end{align*}
  since $1/\theta > 1$. Therefore, thanks to \eqref{eq:mostlyconstant}, we have, for $(\mu, \lambda) \in \Lambda_\theta^c$ with $\mu + \lambda \geq R,$ that
  \begin{align*}
    \| \mathcal{F}^T_\mu \mathcal{F}^G_\lambda (u)\|_{L^2(T \times G)}^2 &\leq C^{-2} (1 + \mu + \lambda)^{2 \rho} \sum_{\ell = 1}^{N'} \left \|  \vv{Y}_\ell \left(  \mathcal{F}^T_\mu \mathcal{F}_\lambda^G(u) \right) \right \|_{L^2(T \times G)}^2 \\
    &\leq C^{-2} (1 + \lambda)^{\frac{4 \rho}{\theta}} \sum_{\ell = 1}^{N'} \left \|  \vv{Y}_\ell \left(  \mathcal{F}^T_\mu \mathcal{F}_\lambda^G(u) \right) \right \|_{L^2(T \times G)}^2.
  \end{align*}
 
  Fixing $\lambda \in \sigma(\Delta_G)$ we have by Remark~\ref{rem:equivL2lambda} that
  \begin{align}
    \| \mathcal{F}^G_\lambda (u)\|_{L^2(T \times G)}^2 &= \sum_{\substack{\mu \in \sigma(\Delta_T) \\ (\mu, \lambda) \in \Lambda_\theta}} \| \mathcal{F}^T_\mu \mathcal{F}^G_\lambda (u)\|_{L^2(T \times G)}^2 + \sum_{\substack{\mu \in \sigma(\Delta_T) \\ (\mu, \lambda) \in \Lambda_\theta^c}} \| \mathcal{F}^T_\mu \mathcal{F}^G_\lambda (u)\|_{L^2(T \times G)}^2 \label{eq:mostlypart1}
  \end{align}
  in which the first sum can be bounded by
  \begin{align}
    \sum_{\substack{\mu \in \sigma(\Delta_T) \\ (\mu, \lambda) \in \Lambda_\theta}} \| \mathcal{F}^T_\mu \mathcal{F}^G_\lambda (u)\|_{L^2(T \times G)}^2 &\leq \sum_{\substack{\mu \in \sigma(\Delta_T) \\ (\mu, \lambda) \in \Lambda_\theta}} \| \mathcal{F}^T_\mu \mathcal{F}^G_\lambda (u)\|_{L^2(T \times G)}\frac{C_1}{(1+\mu+ \lambda)^{s+2n}}\nonumber\\
    &\leq  \frac{C_1}{(1+ \lambda)^{s}} \sum_{\substack{\mu \in \sigma(\Delta_T) \\ (\mu, \lambda) \in \Lambda_\theta}} \frac{\| \mathcal{F}^T_\mu \mathcal{F}^G_\lambda (u)\|_{L^2(T \times G)}}{  (1 + \mu)^{2n}}\nonumber\\
    &\leq  \frac{C_1}{(1+ \lambda)^{s}} \| \mathcal{F}^G_\lambda (u)\|_{L^2(T \times G)}  \sum_{\mu \in \sigma(\Delta_T)} \frac{1}{(1 + \mu)^{2n}}\label{eq:mostlypart2},
  \end{align}
 where the latter series converges by Weyl's asymptotic formula~\eqref{eq:weyl}.

 For the second sum in~\eqref{eq:mostlypart1} we define $\Lambda_{\theta, R}^{c} \dfn \{ (\mu,\lambda) \in \Lambda_{\theta}^{c} \st \mu+ \lambda\geq R\}$: it follows that 
 \begin{align*}
   \sum_{\substack{\mu \in \sigma(\Delta_T) \\ (\mu, \lambda) \in \Lambda_{\theta,R}^c}} \| \mathcal{F}^T_\mu \mathcal{F}^G_\lambda (u)\|_{L^2(T \times G)}^2\leq  C_2 (1 + \lambda)^{\frac{4 \rho}{\theta}} \sum_{\ell = 1}^{N'} \sum_{\substack{\mu \in \sigma(\Delta_T) \\ (\mu, \lambda) \in \Lambda_{\theta,R}^c}} \left \|  \vv{Y}_\ell \left(  \mathcal{F}^T_\mu \mathcal{F}_\lambda^G(u) \right) \right \|_{L^2(T \times G)}^2
 \end{align*}
 which can be further bounded by
  \begin{align*}
    \sum_{\ell = 1}^{N'} \sum_{\substack{\mu \in \sigma(\Delta_T) \\ (\mu, \lambda) \in \Lambda_{\theta,R}^c}} \left \|  \vv{Y}_\ell \left(  \mathcal{F}^T_\mu \mathcal{F}_\lambda^G(u) \right) \right \|_{L^2(T \times G)}^2 &\leq \sum_{\ell = 1}^{N'} \sum_{\mu \in\sigma(\Delta_T)} \left \|  \vv{Y}_\ell \left(  \mathcal{F}^T_\mu \mathcal{F}_\lambda^G(u) \right) \right \|_{L^2(T \times G)}^2 \\
    &= \sum_{\ell = 1}^{N'} \left \| \vv{Y}_\ell \mathcal{F}_\lambda^G(u) \right \|_{L^2(T \times G)}^2 \\
    &\leq \left \langle Q^\sharp \mathcal{F}_\lambda^G(u), \mathcal{F}_\lambda^G(u) \right \rangle_{L^2(T \times G)} + \sum_{\ell = 1}^{N'} \left \| \vv{Y}_\ell \mathcal{F}_\lambda^G(u) \right \|_{L^2(T \times G)}^2 \\
    &\leq \left \langle \mathcal{F}_\lambda^G(f) , \mathcal{F}_\lambda^G(u) \right \rangle_{L^2(T \times G)} \\
    &\leq \| \mathcal{F}^G_\lambda (f)\|_{L^2(T \times G)} \| \mathcal{F}^G_\lambda (u)\|_{L^2(T \times G)}
  \end{align*}
  where we used Proposition~\ref{prop:invopscommft}, Lemma~\ref{lem:energy_identity} and the fact that $Q$ is positive semidefinite. But since $f$ is smooth Corollary~\ref{cor:partial_smoothness} asserts the existence of a constant $C_3 > 0$ such that
  \begin{align*}
    \| \mathcal{F}^G_\lambda (f)\|_{L^2(T \times G)} &\leq C_3 (1 + \lambda)^{ -s - \frac{4 \rho}{\theta}}, \quad \forall \lambda \in \sigma(\Delta_G),
  \end{align*}
  from which we conclude that there exists $C_4 > 0$ such that
  \begin{align}\label{eq:mostlypart3}
       \sum_{\substack{\mu \in \sigma(\Delta_T) \\ (\mu, \lambda) \in \Lambda_{\theta,R}^c}} \| \mathcal{F}^T_\mu \mathcal{F}^G_\lambda (u)\|_{L^2(T \times G)}^2 &\leq C_4 (1 + \lambda)^{-s} \| \mathcal{F}^G_\lambda (u)\|_{L^2(T \times G)}.
  \end{align}
  Using the fact that $\Lambda_{\theta}^{c}\setminus\Lambda_{\theta, R}^{c}$ is finite, it follows from~\eqref{eq:mostlypart1}, \eqref{eq:mostlypart2} and~\eqref{eq:mostlypart3} that 
  \begin{align*}
    \|\mathcal{F}_{\lambda}^{G}(u)\|_{L^{2}(T\times G)} \leq C_5(1+ \lambda)^{-s}, \quad \forall \lambda \in \sigma(\Delta_G),
  \end{align*}
  for some constant $C_5>0$, and the smoothness of $u$ follows from Corollary~\ref{cor:partial_smoothness_converse}.
\end{proof}

The next one is very similar and generalizes~\cite[Theorem~2]{albanese11}.
\begin{Thm} \label{thm:albanese}  Let $P$ in~\eqref{eq:Pdef} be of the form
  \begin{align*}
    P &= Q^\sharp - \sum_{\ell = 1}^{N'} \left( \vv{L}_\ell^\sharp \right)^2 - \sum_{\ell = N' + 1}^{N} \left( \gr{a}_\ell(t, \vv{X}) + \vv{W}_\ell^\sharp \right)^2
  \end{align*}
  where $Q$ is positive semidefinite, $\vv{L}_1, \ldots, \vv{L}_{N'} \in \gr{g}$ and $\vv{W}_{N' + 1}, \ldots, \vv{W}_N$ are skew-symmetric and such that $\tilde{P}=Q- \vv{W}_{N'+1}^{2}-\cdots - \vv{W}_N^{2}$ is elliptic. Assume moreover that the system $\{ \vv{L}_1, \ldots, \vv{L}_{N'}\}$ is~$\mathrm{(GH)}$ in $G$. Then $P$ is~$\mathrm{(GH)}$ in $T \times G$.
  \begin{proof} By Proposition~\ref{prop:ghfinitevfs} there exist $C, \rho > 0$ and $\lambda_0 \in \sigma(\Delta_G)$ such that 
    \begin{align*}
      \left( \sum_{\ell = 1}^{N'} \| \vv{L}_\ell \phi \|_{L^2(G)}^2 \right)^\frac{1}{2} &\geq C (1 + \lambda)^{-\rho} \| \phi\|_{L^2(G)} , \quad \forall \phi \in E_\lambda^G
    \end{align*}
    for all $\lambda \in \sigma(\Delta_G)$ with $\lambda \geq \lambda_0$. In particular for arbitrary $\psi \in \cinfty(T)$ and $\phi \in E_\lambda^G$ we have 
    \begin{align*}
      \sum_{\ell = 1}^{N'} \| \vv{L}_\ell^{\sharp} (\psi \otimes \phi) \|_{L^2(T \times G)}^2 &= \sum_{\ell = 1}^{N'} \| \psi \otimes (\vv{L}_\ell \phi) \|_{L^2(T \times G)}^2  \\
      &= \sum_{\ell = 1}^{N'} \| \psi \|_{L^2(T)}^2 \| \vv{L}_\ell \phi \|_{L^2(G)}^2 \\
      &\geq C^2 (1 + \lambda)^{-2\rho} \| \psi \|_{L^2(T)}^2 \| \phi\|_{L^2(G)}^2 \\
      &= C^2 (1 + \lambda)^{-2\rho} \| \psi \otimes \phi \|_{L^2(T \times G)}^2.
    \end{align*}
    Now if we select, as usual, orthonormal bases $\psi_1^\mu, \ldots, \psi_{d^T_\mu}^\mu$ and $\phi_1^\lambda, \ldots, \phi_{d^G_\lambda}^\lambda$ of $E_\mu^T$ and $E_\lambda^G$ respectively, we may write any $\varphi \in E_\mu^T \otimes E_\lambda^G$ as
    \begin{align*}
      \varphi &= \sum_{i = 1}^{d_\lambda^G} \sum_{j = 1}^{d_\mu^T} \varphi_{ij} \psi_j^\mu \otimes \phi_i^\lambda, \quad \varphi_{ij} \in \C,
    \end{align*}
    or, alternatively,
    \begin{align*}
      \varphi &= \sum_{j = 1}^{d_\mu^T} \psi_j^\mu \otimes \tilde{\varphi}_j , \quad \tilde{\varphi}_j \in E_\lambda^G.
    \end{align*}
    Notice that given $j, j' \in \{1, \ldots, d_\mu^T\}$ we have
    \begin{align*}
      \left \langle \psi_j^\mu \otimes \tilde{\varphi}_j, \psi_{j'}^\mu \otimes \tilde{\varphi}_{j'} \right \rangle_{L^2(T \times G)} = \left \langle \psi_j^\mu , \psi_{j'}^\mu \right \rangle_{L^2(T)} \left \langle \tilde{\varphi}_j, \tilde{\varphi}_{j'} \right \rangle_{L^2(G)} = \delta_{j j'} \left \langle \tilde{\varphi}_j, \tilde{\varphi}_{j'} \right \rangle_{L^2(G)}
    \end{align*}
    hence the terms in the last sum are pairwise orthogonal, and thus
    \begin{align*}
      \| \varphi \|_{L^2(T \times G)}^2 &= \sum_{j = 1}^{d_\mu^T} \left \| \psi_j^\mu \otimes \tilde{\varphi}_j \right \|_{L^2(T \times G)}^2 .
    \end{align*}
    Moreover
    \begin{align*}
      \vv{L}^{\sharp}_\ell \varphi = \sum_{j = 1}^{d_\mu^T} \vv{L}^{\sharp}_\ell \left( \psi_j^\mu \otimes \tilde{\varphi}_j \right) = \sum_{j = 1}^{d_\mu^T} \psi_j^\mu \otimes (\vv{L}_\ell \tilde{\varphi}_j)
    \end{align*}
    so by the previous argument we have that all the terms above are orthogonal in $L^2(T \times G)$, hence
    \begin{align*}
      \| \vv{L}^{\sharp}_\ell \varphi \|_{L^2(T \times G)}^2 &= \sum_{j = 1}^{d_\mu^T} \left \| \vv{L}^{\sharp}_\ell \left( \psi_j^\mu \otimes \tilde{\varphi}_j \right) \right \|_{L^2(T \times G)}^2 .
    \end{align*}
    If $\lambda \geq \lambda_0$ then
    \begin{align*}
      \sum_{\ell = 1}^{N'} \| \vv{L}^{\sharp}_\ell \varphi \|_{L^2(T \times G)}^2 &= \sum_{j = 1}^{d_\mu^T} \sum_{\ell = 1}^{N'} \left \| \vv{L}^{\sharp}_\ell \left( \psi_j^\mu \otimes \tilde{\varphi}_j \right) \right \|_{L^2(T \times G)}^2 \\
      &\geq \sum_{j = 1}^{d_\mu^T} C^2 (1 + \lambda)^{-2\rho} \left \| \psi_j^\mu \otimes \tilde{\varphi}_j \right \|_{L^2(T \times G)}^2 \\
      &= C^2 (1 + \lambda)^{-2\rho} \| \varphi \|_{L^2(T \times G)}^2.
    \end{align*}
    We conclude that for every $(\mu, \lambda) \in \sigma(\Delta_T) \times \sigma(\Delta_G)$ with $\lambda \geq \lambda_0$ we have
    \begin{align}
      \sum_{\ell = 1}^{N'} \| \vv{L}^{\sharp}_\ell \varphi \|_{L^2(T \times G)}^2 &\geq C^2 (1 + \lambda)^{-2\rho} \| \varphi \|_{L^2(T \times G)}^2, \quad \forall \varphi \in E_\mu^T \otimes E_\lambda^G. \label{eq:ineq_full_tensors} 
    \end{align}
    
    Let $u \in \D'(T \times G)$ be such that $f \dfn Pu \in \cinfty(G)$, so again $\mathcal{F}_\lambda^G(u) \in \cinfty(T; E_\lambda^G)$ for every $\lambda \in \sigma(\Delta_G)$. For each $\ell \in \{1, \ldots, N'\}$, since $\vv{L}_\ell$ is a left-invariant vector field on $G$, and as such commutes with $\Delta_G$, we have that $\vv{Y}_\ell \dfn \vv{L}_\ell^\sharp$ commutes with $\Delta$, so again~\eqref{eq:relellpcommute} holds. Therefore for $(\mu, \lambda) \in \sigma(\Delta_T) \times \sigma(\Delta_G)$ with $\lambda \geq \lambda_0$ we have, by~\eqref{eq:ineq_full_tensors} and~\eqref{eq:relellpcommute},
    \begin{align*}
      \| \mathcal{F}^T_\mu \mathcal{F}^G_\lambda (u)\|_{L^2(T \times G)}^2 \leq C^{-2} (1 + \lambda)^{2 \rho} \sum_{\ell = 1}^{N'} \left \|  \mathcal{F}^T_\mu \left( \vv{L}^{\sharp}_\ell \mathcal{F}_\lambda^G(u) \right) \right \|_{L^2(T \times G)}^2 
    \end{align*}
    so summing both sides over $\mu \in \sigma(\Delta_T)$ yields
    \begin{align*}
      \| \mathcal{F}^G_\lambda (u)\|_{L^2(T \times G)}^2 &\leq C^{-2} (1 + \lambda)^{2 \rho} \sum_{\ell = 1}^{N'} \left \|  \vv{L}^{\sharp}_\ell\mathcal{F}_\lambda^G(u) \right \|_{L^2(T \times G)}^2 \\
      &\leq C^{-2} (1 + \lambda)^{2 \rho} \| \mathcal{F}^G_\lambda (f)\|_{L^2(T \times G)} \| \mathcal{F}^G_\lambda (u)\|_{L^2(T \times G)}
    \end{align*}
    for every $\lambda \geq \lambda_0$, where we proceed as in the previous theorem; as such, we conclude smoothness of $u$, keeping in mind the finiteness of the set $\{ \lambda \in \sigma(\Delta_G) \st \lambda < \lambda_0 \}$.
  \end{proof}
\end{Thm}

\bibliographystyle{plain}
\bibliography{bibliography}

\def\cprime{$'$}
\begin{thebibliography}{10}

\bibitem{albanese11}
A.~A. Albanese.
\newblock On the global {$C^\infty$} and {G}evrey hypoellipticity on the torus
  of some classes of degenerate elliptic operators.
\newblock {\em Note Mat.}, 31(1):1--13, 2011.

\bibitem{araujo19}
G.~Ara\'ujo.
\newblock Global regularity and solvability of left-invariant differential
  systems on compact {L}ie groups.
\newblock {\em Ann. Glob. Anal. Geom.}, 56(4):631--665, 2019.

\bibitem{bfp17}
R.~F. Barostichi, I.~A. Ferra, and G.~Petronilho.
\newblock Global hypoellipticity and simultaneous approximability in
  ultradifferentiable classes.
\newblock {\em J. Math. Anal. Appl.}, 453(1):104--124, 2017.

\bibitem{bccp}
N.~Braun~Rodrigues, G.~Chinni, P.~D. Cordaro, and M.~R. Jahnke.
\newblock Lower order perturbation and global analytic vectors for a class of
  globally analytic hypoelliptic operators.
\newblock {\em Proc. Amer. Math. Soc.}, 144(12):5159--5170, 2016.

\bibitem{chavel_eigenvalues}
I.~Chavel.
\newblock {\em Eigenvalues in {R}iemannian geometry}, volume 115 of {\em Pure
  and Applied Mathematics}.
\newblock Academic Press, Inc., Orlando, FL, 1984.
\newblock Including a chapter by Burton Randol, With an appendix by Jozef
  Dodziuk.

\bibitem{christ94}
M.~Christ.
\newblock Global analytic hypoellipticity in the presence of symmetry.
\newblock {\em Math. Res. Lett.}, 1(5):559--563, 1994.

\bibitem{chim94}
P.~D. Cordaro and A.~A. Himonas.
\newblock Global analytic hypoellipticity of a class of degenerate elliptic
  operators on the torus.
\newblock {\em Math. Res. Lett.}, 1(4):501--510, 1994.

\bibitem{chim98}
P.~D. Cordaro and A.~A. Himonas.
\newblock Global analytic regularity for sums of squares of vector fields.
\newblock {\em Trans. Amer. Math. Soc.}, 350(12):4993--5001, 1998.

\bibitem{gw72}
S.~J. Greenfield and N.~R. Wallach.
\newblock Global hypoellipticity and {L}iouville numbers.
\newblock {\em Proc. Amer. Math. Soc.}, 31:112--114, 1972.

\bibitem{him95}
A.~A. Himonas.
\newblock On degenerate elliptic operators of infinite type.
\newblock {\em Math. Z.}, 220(3):449--460, 1995.

\bibitem{hp00}
A.~A. Himonas and G.~Petronilho.
\newblock Global hypoellipticity and simultaneous approximability.
\newblock {\em J. Funct. Anal.}, 170(2):356--365, 2000.

\bibitem{hps06}
A.~A. Himonas, G.~Petronilho, and L.~A.~C. dos Santos.
\newblock Regularity of a class of sub{L}aplacians on the 3-dimensional torus.
\newblock {\em J. Funct. Anal.}, 240(2):568--591, 2006.

\bibitem{hormander67}
L.~H{\"o}rmander.
\newblock Hypoelliptic second order differential equations.
\newblock {\em Acta Math.}, 119:147--171, 1967.

\bibitem{hormander71}
L.~H\"{o}rmander.
\newblock Fourier integral operators. {I}.
\newblock {\em Acta Math.}, 127(1-2):79--183, 1971.

\bibitem{knapp_lgbi}
A.~W. Knapp.
\newblock {\em Lie groups beyond an introduction}, volume 140 of {\em Progress
  in Mathematics}.
\newblock Birkh\"auser Boston, Inc., Boston, MA, 1996.

\bibitem{lee_smooth}
J.~M. Lee.
\newblock {\em Introduction to smooth manifolds}, volume 218 of {\em Graduate
  Texts in Mathematics}.
\newblock Springer-Verlag, New York, 2003.

\bibitem{sachkov07}
Yu.~L. Sachkov.
\newblock Control theory on {L}ie groups.
\newblock {\em Sovrem. Mat. Fundam. Napravl.}, 27:5--59, 2007.

\bibitem{sussmann73}
H.~J. Sussmann.
\newblock Orbits of families of vector fields and integrability of
  distributions.
\newblock {\em Trans. Amer. Math. Soc.}, 180:171--188, 1973.

\end{thebibliography}
\end{document}